\documentclass[12pt]{article}
\usepackage{latexsym, amssymb, amsthm}
\usepackage{dsfont}

\DeclareMathAlphabet\gothic{U}{euf}{m}{n}

\DeclareFontFamily{U}{mathx}{}
\DeclareFontShape{U}{mathx}{m}{n}{<-> mathx10}{}
\DeclareSymbolFont{mathx}{U}{mathx}{m}{n}
\DeclareMathAccent{\widehat}{0}{mathx}{"70}
\DeclareMathAccent{\widecheck}{0}{mathx}{"71}

\makeatletter
\def\eqnarray{\stepcounter{equation}\let\@currentlabel=\theequation
\global\@eqnswtrue
\tabskip\@centering\let\\=\@eqncr
$$\halign to \displaywidth\bgroup\hfil\global\@eqcnt\z@
  $\displaystyle\tabskip\z@{##}$&\global\@eqcnt\@ne
  \hfil$\displaystyle{{}##{}}$\hfil
  &\global\@eqcnt\tw@ $\displaystyle{##}$\hfil
  \tabskip\@centering&\llap{##}\tabskip\z@\cr}

\def\endeqnarray{\@@eqncr\egroup
      \global\advance\c@equation\m@ne$$\global\@ignoretrue}

\def\@yeqncr{\@ifnextchar [{\@xeqncr}{\@xeqncr[5pt]}}
\makeatother

\begin{document}
\bibliographystyle{tom}

\newtheorem{lemma}{Lemma}[section]
\newtheorem{thm}[lemma]{Theorem}
\newtheorem{cor}[lemma]{Corollary}
\newtheorem{prop}[lemma]{Proposition}

\theoremstyle{definition}

\newtheorem{remark}[lemma]{Remark}
\newtheorem{exam}[lemma]{Example}
\newtheorem{definition}[lemma]{Definition}

\newcommand{\gota}{\gothic{a}}
\newcommand{\gotb}{\gothic{b}}
\newcommand{\gotc}{\gothic{c}}
\newcommand{\gote}{\gothic{e}}
\newcommand{\gotf}{\gothic{f}}
\newcommand{\gotg}{\gothic{g}}
\newcommand{\gothh}{\gothic{h}}
\newcommand{\gotk}{\gothic{k}}
\newcommand{\gotm}{\gothic{m}}
\newcommand{\gotn}{\gothic{n}}
\newcommand{\gotp}{\gothic{p}}
\newcommand{\gotq}{\gothic{q}}
\newcommand{\gotr}{\gothic{r}}
\newcommand{\gots}{\gothic{s}}
\newcommand{\gott}{\gothic{t}}
\newcommand{\gotu}{\gothic{u}}
\newcommand{\gotv}{\gothic{v}}
\newcommand{\gotw}{\gothic{w}}
\newcommand{\gotz}{\gothic{z}}
\newcommand{\gotA}{\gothic{A}}
\newcommand{\gotB}{\gothic{B}}
\newcommand{\gotG}{\gothic{G}}
\newcommand{\gotL}{\gothic{L}}
\newcommand{\gotS}{\gothic{S}}
\newcommand{\gotT}{\gothic{T}}

\newcounter{teller}
\renewcommand{\theteller}{(\alph{teller})}
\newenvironment{tabel}{\begin{list}%
{\rm  (\alph{teller})\hfill}{\usecounter{teller} \leftmargin=1.1cm
\labelwidth=1.1cm \labelsep=0cm \parsep=0cm}
                      }{\end{list}}

\newcounter{tellerr}
\renewcommand{\thetellerr}{(\roman{tellerr})}
\newenvironment{tabeleq}{\begin{list}%
{\rm  (\roman{tellerr})\hfill}{\usecounter{tellerr} \leftmargin=1.1cm
\labelwidth=1.1cm \labelsep=0cm \parsep=0cm}
                         }{\end{list}}

\newcounter{tellerrr}
\renewcommand{\thetellerrr}{(\Roman{tellerrr})}
\newenvironment{tabelR}{\begin{list}%
{\rm  (\Roman{tellerrr})\hfill}{\usecounter{tellerrr} \leftmargin=1.1cm
\labelwidth=1.1cm \labelsep=0cm \parsep=0cm}
                         }{\end{list}}

\newcounter{proofstep}
\newcommand{\nextstep}{\refstepcounter{proofstep}\vertspace \par 
          \noindent{\bf Step \theproofstep} \hspace{5pt}}
\newcommand{\firststep}{\setcounter{proofstep}{0}\nextstep}

\newcommand{\Ni}{\mathds{N}}
\newcommand{\Qi}{\mathds{Q}}
\newcommand{\Ri}{\mathds{R}}
\newcommand{\Ci}{\mathds{C}}
\newcommand{\Ti}{\mathds{T}}
\newcommand{\Zi}{\mathds{Z}}
\newcommand{\Fi}{\mathds{F}}

\renewcommand{\proofname}{{\bf Proof}}

\newcommand{\vertspace}{\vskip10.0pt plus 4.0pt minus 6.0pt}

\newcommand{\simh}{{\stackrel{{\rm cap}}{\sim}}}
\newcommand{\ad}{{\mathop{\rm ad}}}
\newcommand{\Ad}{{\mathop{\rm Ad}}}
\newcommand{\alg}{{\mathop{\rm alg}}}
\newcommand{\clalg}{{\mathop{\overline{\rm alg}}}}
\newcommand{\Aut}{\mathop{\rm Aut}}
\newcommand{\arccot}{\mathop{\rm arccot}}
\newcommand{\capp}{{\mathop{\rm cap}}}
\newcommand{\rcapp}{{\mathop{\rm rcap}}}
\newcommand{\diam}{\mathop{\rm diam}}
\newcommand{\divv}{\mathop{\rm div}}
\newcommand{\dom}{\mathop{\rm dom}}
\newcommand{\codim}{\mathop{\rm codim}}
\newcommand{\RRe}{\mathop{\rm Re}}
\newcommand{\IIm}{\mathop{\rm Im}}
\newcommand{\tr}{{\mathop{\rm Tr \,}}}
\newcommand{\Tr}{{\mathop{\rm Tr \,}}}
\newcommand{\Vol}{{\mathop{\rm Vol}}}
\newcommand{\card}{{\mathop{\rm card}}}
\newcommand{\rank}{\mathop{\rm rank}}
\newcommand{\supp}{\mathop{\rm supp}}
\newcommand{\sgn}{\mathop{\rm sgn}}
\newcommand{\essinf}{\mathop{\rm ess\,inf}}
\newcommand{\esssup}{\mathop{\rm ess\,sup}}
\newcommand{\Int}{\mathop{\rm Int}}
\newcommand{\lcm}{\mathop{\rm lcm}}
\newcommand{\loc}{{\rm loc}}
\newcommand{\HS}{{\rm HS}}
\newcommand{\Trn}{{\rm Tr}}
\newcommand{\n}{{\rm N}}
\newcommand{\WOT}{{\rm WOT}}

\newcommand{\at}{@}

\newcommand{\mod}{\mathop{\rm mod}}
\newcommand{\spann}{\mathop{\rm span}}
\newcommand{\one}{\mathds{1}}

\hyphenation{groups}
\hyphenation{unitary}

\newcommand{\tfrac}[2]{{\textstyle \frac{#1}{#2}}}

\newcommand{\ca}{{\cal A}}
\newcommand{\cb}{{\cal B}}
\newcommand{\cc}{{\cal C}}
\newcommand{\cd}{{\cal D}}
\newcommand{\ce}{{\cal E}}
\newcommand{\cf}{{\cal F}}
\newcommand{\ch}{{\cal H}}
\newcommand{\chs}{{\cal HS}}
\newcommand{\ci}{{\cal I}}
\newcommand{\ck}{{\cal K}}
\newcommand{\cl}{{\cal L}}
\newcommand{\cm}{{\cal M}}
\newcommand{\cn}{{\cal N}}
\newcommand{\co}{{\cal O}}
\newcommand{\cp}{{\cal P}}
\newcommand{\cs}{{\cal S}}
\newcommand{\ct}{{\cal T}}
\newcommand{\cx}{{\cal X}}
\newcommand{\cy}{{\cal Y}}
\newcommand{\cz}{{\cal Z}}

\thispagestyle{empty}

\vspace*{1cm}
\begin{center}
{\Large\bf  Commutator estimates and Poisson bounds for  \\[2mm]
Dirichlet-to-Neumann operators  \\[2mm]
with variable coefficients\\[10mm]
\large A.F.M. ter Elst$^1$ and E.M. Ouhabaz$^2$}

\end{center}

\vspace{2mm}

\begin{center}
{\bf Abstract}
\end{center}

\begin{list}{}{\leftmargin=1.8cm \rightmargin=1.8cm \listparindent=10mm 
   \parsep=0pt}
\item
We consider the Dirichlet-to-Neumann operator $\cn$ associated with a general elliptic operator
\[
\ca u = - \sum_{k,l=1}^d \partial_k (c_{kl}\, \partial_l u)
+ \sum_{k=1}^d \Big( c_k\, \partial_k u - \partial_k (b_k\, u) \Big) +c_0\, u
\in \cd'(\Omega)
\]
with possibly complex coefficients. 
We study three problems: \\
1) Boundedness on $C^\nu$ and on $L_p$ of the   commutator $[\cn, M_g]$, 
where $M_g$ denotes the multiplication operator
by a  smooth function $g$. \\
2) H\"older and $L_p$-bounds for  the  harmonic lifting associated with~$\ca$.  \\
3) Poisson bounds for the heat kernel of $\cn$. \\
We solve these problems  in the case where  the coefficients are H\"older continuous and the  underlying domain
is bounded and of class $C^{1+\kappa}$ for some $\kappa > 0$.  
For the Poisson bounds we assume in addition that the coefficients are real-valued. 
We also prove gradient estimates for  the heat kernel and the 
Green function $G$ of the elliptic operator with Dirichlet boundary conditions. 
\end{list}


\vspace{3mm}

\vspace{5mm}

\noindent
{\bf Home institutions:}    \\[3mm]
\begin{tabular}{@{}cl@{\hspace{10mm}}cl}
1. & Department of Mathematics  & 
  2. & Institut de Math\'ematiques de Bordeaux \\
& University of Auckland   & 
  & Universit\'e de Bordeaux, UMR 5251,  \\
& Private bag 92019 & 
  &351, Cours de la Lib\'eration  \\
& Auckland 1142 & 
  &  33405 Talence \\
& New Zealand  & 
  & France\\
  & terelst@math.auckland.ac.nz&
 & Elmaati.Ouhabaz@math.u-bordeaux.fr \\[8mm]
\end{tabular}

\newpage

\section{Introduction and the main results} \label{Scom1}

Let $\Omega \subset \Ri^d$ be an open set with $d \geq 2$ and  
let $c_{kl}, b_k, c_k, c_0 \in L_\infty(\Omega) = L_\infty(\Omega, \Ci)$ for all
$k,l \in \{ 1,\ldots,d \}$. 
We assume the usual ellipticity condition, that is,  there is a $\mu > 0$ such that 
\begin{equation}
\RRe \sum_{k,l=1}^d c_{kl}(x) \, \xi_k \, \overline{\xi_l} \geq \mu \, |\xi|^2
\label{ecomS1;1}
\end{equation}
for all $\xi \in \Ci^d$ and ($x$-a.e.) $x \in \Omega$.
Define the sesquilinear form $\gota \colon W^{1,2}(\Omega) \times W^{1,2}(\Omega) \to \Ci$
by 
\begin{equation}\label{eScom1;1}
\gota(u,v) 
= \sum_{k,l=1}^d \int_\Omega c_{kl} \, (\partial_l u) \,  \, \overline{\partial_k v}
   + \sum_{k=1}^d \int_\Omega b_k \, u \, \overline{\partial_k v}
   + \sum_{k=1}^d \int_\Omega c_k \, (\partial_k u) \, \overline v
   + \int_\Omega c_0 \, u \, \overline v.  
   \end{equation} 
It is a basic fact that this form is continuous and quasi-coercive on $W^{1,2}(\Omega)$.
Its associated operator $A_N$ is the elliptic operator with Neumann type boundary conditions. 
Consider  $\gota_D = \gota|_{W^{1,2}_0(\Omega) \times W^{1,2}_0(\Omega)}$, 
the restriction of $\gota$ to the subspace $W^{1,2}_0(\Omega)$.
Its associated operator $A_D$ is the corresponding elliptic operator with 
Dirichlet boundary conditions.

Further define $\ca \colon W^{1,2}(\Omega) \to  \cd'(\Omega)$ by 
\[
\langle \ca u,v \rangle_{\cd'(\Omega) \times \cd(\Omega)}
= \gota(u,v)
.  \]
The operator $\ca$ is given by 
\[
\ca u = - \sum_{k,l=1}^d \partial_k (c_{kl}\, \partial_l u)
+ \sum_{k=1}^d \Big( c_k\, \partial_k u - \partial_k (b_k\, u) \Big) +c_0\, u
\in \cd'(\Omega)
.  \]
A function $u \in W^{1,2}(\Omega)$ is called {\bf harmonic} (or {\bf $\ca$-harmonic}) if $\ca u = 0$.

Next suppose $\Omega$ is bounded and has Lipschitz boundary $\Gamma = \partial \Omega$. 
We provide $\Gamma$ with the $(d-1)$-dimensional Hausdorff measure.
Let $\Tr \colon W^{1,2}(\Omega) \to L_2(\Gamma)$ be the trace map.
Note that $\Tr W^{1,2}(\Omega) = H^{1/2}(\Gamma)$.
If $0 \not\in \sigma(A_D)$, then there exists a unique operator 
$\gamma \colon H^{1/2}(\Gamma) \to W^{1,2}(\Omega)$
such that $\Tr \gamma(\varphi) = \varphi$ and $\gamma(\varphi)$ is
harmonic for all $\varphi \in H^{1/2}(\Gamma)$.
We call $\gamma$ the {\bf harmonic lifting}.

Next we define the weak conormal derivative $\partial_n^\gota$.
If $u \in W^{1,2}(\Omega)$ and $\psi \in L_2(\Gamma)$, then we say that 
$u \in D(\partial_n^\gota)$ and $\partial_n^\gota u = \psi$ if 
$\ca u \in L_2(\Omega)$ and 
\[
\gota(u,v) - \int_\Omega (\ca u) \, \overline v
= \int_\Gamma \psi \, \overline{\Tr v}
\]
for all $v \in W^{1,2}(\Omega)$.
It is a consequence of the Stone--Weierstra\ss\ theorem that 
the space $\Tr W^{1,2}(\Omega)$ is dense in $L_2(\Gamma)$ and 
hence the function $\psi$ is unique.
We call $\psi$ the {\bf (weak) conormal derivative} of $u$.
If for example  $c_{kl},b_k \in W^{1,\infty}(\Omega)$ 
for all $k,l \in \{ 1,\ldots,d \} $ and the set $\Omega$ is bounded and Lipschitz, then
\[
\partial_n^\gota u 
= \sum_{k,l=1}^d n_k \, c_{kl} \, \Tr(\partial_l u)
   + \sum_{k=1}^d n_k \, b_k \, \Tr u
\]
for all $u \in W^{2,2}(\Omega)$, where $n = (n_1,\ldots,n_d)$ is the outer normal on~$\Gamma$.

If $0 \not\in \sigma(A_D)$ and the set $\Omega$ is 
bounded and Lipschitz, then we define the {\bf Dirichlet-to-Neumann operator}
$\cn$ in $L_2(\Gamma)$ by 
\[
\cn = \partial_n^\gota \circ \gamma
.  \]
So if $\varphi,\psi \in L_2(\Gamma)$, then $\varphi \in D(\cn)$ 
and $\cn \varphi = \psi$ if and only if there exists a harmonic function
$u \in W^{1,2}(\Omega)$ such that $\Tr u = \varphi$ and 
$\partial_n^\gota u = \psi$.
An alternative characterisation is as follows.
Let $\varphi,\psi \in L_2(\Gamma)$, then $\varphi \in D(\cn)$ 
and $\cn \varphi = \psi$ if and only if there exists a 
$u \in W^{1,2}(\Omega)$ such that $\Tr u = \varphi$ and 
\[
\gota(u,v) 
= \int_\Gamma \psi \, \overline{\Tr v}
\]
for all $v \in W^{1,2}(\Omega)$.

Actually, $\cn$ can also be defined as the associated operator with the sesquilinear form
$\gotb \colon H^{1/2}(\Gamma) \times H^{1/2}(\Gamma) \to \Ci$ given by
\[ 
\gotb(\varphi, \eta) = \gota(\gamma(\varphi),\gamma(\eta))
.  \]
It turns out that this form is densely defined, continuous and quasi-coercive.
Therefore, as an operator associated with a sesquilinear form, $\cn$ is 
densely defined, quasi-accretive  and $-\cn$ generates a 
strongly continuous holomorphic semigroup on $L_2(\Gamma)$.
If  $c_{kl} = \overline{c_{lk}}$, $c_k = \overline{b_k}$ and $c_0$ is real, for all $k,l \in  \{ 1,\ldots,d \}$,
then $\cn$ is a self-adjoint operator.
For all these basic facts, we refer to \cite{EO6}, \cite{EO4} and the references therein. 
As an example, if $\Omega$ is the open unit ball and $- A = \Delta$ is the 
Laplacian on $\Omega$, then $\cn = \sqrt{- \Delta_{LB} + c_d^2} - c_d$,
where $\Delta_{LB}$ is the Laplace--Beltrami operator on the unit sphere and
$c_d = \frac{d-2}{2}$ (see (C.29) in \cite{Tay5}).

In this paper we are interested in a number of important properties of $\cn$, 
among them $L_p(\Gamma)$ and  H\"older continuity  
estimates for the commutator $[\cn, M_g]$ of $\cn$ with the multiplication operator $M_g$
with a function $g$.
For the harmonic lifting $\gamma$ we are interested in 
$L_p(\Gamma) \to L_p(\Omega)$ and  H\"older continuity  estimates.
We emphasize that we consider these problems for 
non-symmetric $\ca$ and $\cn$ with complex coefficients.
Further we are interested in 
$L_p \to L_q$ boundedness for the semigroup $(e^{-t\cn})_{t > 0}$ 
and Poisson upper bounds for the heat kernel 
of $\cn$ when the coefficients are real.
No results of these types are available in this setting in the literature. 
On a domain of class $C^{1+\kappa}$ we solve some of these problems and  
we state now our main results.

\begin{thm} \label{tcom101}
Let $\kappa \in (0,1)$.
Let $\Omega \subset \Ri^d$ be a bounded open set of class $C^{1+\kappa}$.  
Suppose $c_{kl}, b_k, c_k \in C^\kappa(\Omega)$,
and $c_0 \in L_\infty(\Omega)$ for all $k,l \in \{ 1,\ldots,d \} $.
Further suppose that the ellipticity condition (\ref{ecomS1;1}) is satisfied. 
Suppose that $0 \not\in \sigma(A_D)$.
Then we have  the following assertions.
\begin{tabel}
\item \label{tcom101-2}
For all $\nu \in (0,\kappa)$ there exists a $c > 0$ such that 
\[
\|[\cn, M_g]\|_{C^\nu(\Gamma) \to C^\nu(\Gamma)} \leq c
\]
for all $g \in W^{2,\infty}(\Ri^d)$ such that 
$\|\partial_k g\|_\infty \leq 1$ and $\|\partial_k \, \partial_l g\|_\infty \leq 1$
for all $k,l \in \{ 1,\ldots,d \} $.
\item \label{tcom101-1}
For all $p \in (1,\infty)$ there exists a $c > 0$ such that 
\[
\|[\cn, M_g]\|_{L_p(\Gamma) \to L_p(\Gamma)} \leq c
\]
for all $g \in W^{2,\infty}(\Ri^d)$ such that 
$\|\partial_k g\|_\infty \leq 1$ and $\|\partial_k \, \partial_l g\|_\infty \leq 1$
for all $k,l \in \{ 1,\ldots,d \} $. 
\item \label{tcom101-3}
The operator $[\cn, M_g]$ is of weak type $(1,1)$
for all $g \in W^{2,\infty}(\Ri^d)$.
\end{tabel}
\end{thm}

In order to be more precise, for Statement~\ref{tcom101-2} we prove that for all 
$g \in W^{2,\infty}(\Ri^d)$ and $\nu \in (0,\kappa)$
there is a bounded operator $\widetilde T_g \colon C^\nu(\Gamma) \to C^\nu(\Gamma)$
such that for all $\varphi \in D(\cn) \cap C^\nu(\Gamma)$ one 
has $g|_\Gamma \, \varphi \in D(\cn) \cap C^\nu(\Gamma)$
and $[\cn, M_g] \varphi = \widetilde T_g \varphi$.
A similar formulation is valid for the other statements. 
We note that even in the case where $A$ is (minus) the Laplacian, 
the $C^\nu$-estimate for  the commutator is new. 
These commutator estimates play an important role in the proof of the Poisson bounds
for the heat kernel of the Dirichlet-to-Neumann operator.

For the harmonic lifting we have the following result.

\begin{thm} \label{tcom102}
Let $\kappa \in (0,1)$.
Let $\Omega \subset \Ri^d$ be a bounded open set of class $C^{1+\kappa}$.  
Suppose $c_{kl}, b_k, c_k \in C^\kappa(\Omega)$,
and $c_0 \in L_\infty(\Omega)$ for all $k,l \in \{ 1,\ldots,d \} $.
Further suppose that the ellipticity condition (\ref{ecomS1;1}) is satisfied. 
Suppose that $0 \not\in \sigma(A_D)$.
Then one has the following.
\begin{tabel}
\item \label{tcom102-1}
For all $\nu \in (0,\kappa)$ the harmonic lifting 
extends consistently to a continuous map from $C^\nu(\Gamma)$ into $C^\nu(\Omega)$.
\item \label{tcom102-2}
For all $p \in [1,\infty)$ the harmonic lifting extends consistently to a 
continuous map from $L_p(\Gamma)$ into $L_p(\Omega)$.
\end{tabel}
\end{thm}

Concerning $L_p \to L_q$ boundedness and Poisson bounds, we deal with operators with 
real-valued coefficients.

\begin{thm}\label{tcom103}
Let $\kappa \in (0,1)$.
Let $\Omega \subset \Ri^d$ be a bounded open set of class $C^{1+\kappa}$.  
Suppose $c_{kl}, b_k, c_k \in C^\kappa(\Omega)$,
and $c_0 \in L_\infty(\Omega)$ for all $k,l \in \{ 1,\ldots,d \} $.
Further suppose that the ellipticity condition (\ref{ecomS1;1}) is satisfied. 
Suppose also  that all these coefficients are real-valued and that $0 \not\in \sigma(A_D)$.
Then the semigroup generated by $-\cn$ has a kernel $K$ and  there exist
$c > 0$ and $\omega_0 \in \Ri$ such that 
\[
| K_t(z,w) | 
\leq \frac{c \, t^{-(d-1)} \, e^{\omega_0 t}}
         {\displaystyle \Big( 1 + \frac{|z-w|}{t} \Big)^d }
\]
for all $z, w \in \Gamma$ and $t > 0$. 
\end{thm} 

Note that $L_p(\Gamma) \to L_q(\Gamma)$ estimates for $e^{-t \cn}$ are satisfied 
for all $1 \le p \le q \le \infty$ and are 
implicit in this theorem. 
We also note that this later theorem remains valid for complex $c_0$. 
We explain this in  Section~\ref{Scom10}, where the proof of the theorem is given.
We believe that the same 
result holds for complex coefficients  as in the first two theorems,
but this turns out to be a challenging problem which  
remains to be solved. 

The Dirichlet-to-Neumann map is an important operator which appears 
in electrical impedance tomography, Calder\'on's inverse problems, spectral theory 
and many other subjects in mathematics and modeling.

The famous inverse problem of Calder\'on consists in determining the 
potential $V$ or the diffusion coefficient $c(x)$ from the knowledge of the 
Dirichlet-to-Neumann operator when  $\ca = -\Delta + V$ (Schr\"odinger operator) 
or $\ca = -\sum_{k=1}^d \partial _k (c(x) \partial_k)$ for the isotropic case.
This problem is nowadays quite well understood at least when the 
Dirichlet-to-Neumann operator is known on the whole boundary.
The anisotropic case, i.e., when $\ca = - \sum_{k,l=1}^d \partial_k(c_{kl} \partial_l)$, 
is more complicate and one can only hope for uniqueness up to a diffeomorphism.
We refer to the papers of Lee and Uhlmann \cite{LeeUhlmann} and 
Astala, P\"aiv\"arinta and Lassas \cite{APL} and the references therein.
For a weaker version of Calder\'on's problem with general second-order 
operator $\ca$ as we consider here, we refer to 
Behrndt and Rohleder \cite{BehR} or Ouhabaz \cite{Ouh8}.
The use of the Dirichlet-to-Neumann operator in spectral theory is a 
subject which has attracted attention.
We refer the reader to the papers of Friedlander \cite{Friedlander}, 
Arendt and Mazzeo \cite{ArM2} 
and the more recent paper by Girouard et al.\ \cite{GKLP}, or for Stokes operators
as in  Denis and ter Elst  \cite{DeE}. 
It is also used in scattering theory,
see for example Behrndt, Malamud and Neidhardt \cite{BMN}.
For  modelling of electrical impedance tomography and historical notes on 
Calder\'on's inverse problem we refer to Uhlmann \cite{uhlmann2009}.   
The Dirichlet-to-Neumann operator with variable coefficients plays an 
important role in modelling and analysis of water waves, see e.g.\ Lannes \cite{Lannes}.
It  is also used in the theory of homogenization 
and analysis of elliptic systems with rapidly oscillating coefficients 
(see e.g.\ Kenig, Lin and Shen \cite{KLS} and the references therein). 
When $\ca = -\Delta$ and $g(x) = x_k$, then  $L_p$-estimates for the commutator 
$[\cn, M_g]$ were proved in \cite{KLS} 
for $C^1$-domains.
The result there is also valid with $p= 2$ for any  bounded Lipschitz domain.
This later case 
was extended to the case of variable coefficients by Shen  \cite{She2}, 
i.e., $\ca = -\sum_{k,l} \partial_k (c_{kl} \, \partial_l)$.
It is assumed 
in \cite{She2} that the coefficients are symmetric ($c_{kl} = c_{lk}$), 
H\"older continuous and real-valued.
The function $g$ is any  smooth function on $\Gamma$. 
All the assumptions on the coefficients made in \cite{She2} play 
an important role in the proof.
Based on this $L_2$-result from \cite{She2}, it was proved by ter Elst and Ouhabaz  
in  \cite{EO6} that for $C^{1+\kappa}$-domains, 
this commutator is of weak type $(1,1)$ and it is bounded on $L_p(\Gamma)$ 
for all $p \in (1, \infty)$.
In addition it was allowed 
there to add a real-valued potential $V$ to the elliptic operator~$\ca$.
No terms of order one or non-symmetric 
principal $c_{kl}$, however,  were considered in \cite{EO6}.
Recently, $L_2$-estimates for the commutator were proved  in the setting of the 
half-space by Hofmann and Zhang \cite{HofmannZhang} for complex coefficients 
which are a small perturbation of real symmetric coefficients. 
Xu, Zhao and Zhou \cite{XuZhaoZhou} extended Shen's commutator result to the case of Stokes systems.

Concerning Poisson bounds we mention the results from \cite{EO4} and \cite{EO6}.
The first paper deals with the case of the 
Dirichlet-to-Neumann operator associated with the positive Laplacian and 
a positive measurable potential on a 
$C^\infty$ bounded domain.
Without potential the Dirichlet-to-Neumann operator
$\cn$ is a pseudo-differential operator and we used there $L_p$-estimates of 
Coifman and Meyer \cite{CM2} for such operators. 
Related results in this setting were also proved by Grimperlein and Grubb \cite{GG}.
The case of variable coefficients is treated in \cite{EO6} for
$C^{1+\kappa}$-domains and real-valued symmetric H\"older coefficients.
In \cite{EO6} we relied on the $L_2$-estimate of \cite{She2} for the commutator together with  H\"older 
bounds for the Green function of the elliptic operator on $L_2(\Omega)$ 
in order to extend the commutator to $L_p(\Gamma)$.
Once this extension was proved, we combined the commutator estimates
with several $L_p \to L_q$ estimates 
for the semigroup $(e^{-t\cn})_{t > 0}$ in order to obtain 
$L_1 \to L_\infty$ estimates for  iterated commutators of $e^{-t \cn}$  and 
$ M_g$ and then optimised over~$g$.
One extremely 
helpful fact with real coefficients (without lower-order terms) is that the 
semigroup $(e^{-t\cn})_{t > 0}$ is (sub-)Markovian and hence by an extrapolation 
argument together with the standard Sobolev inequality on the 
boundary one obtains appropriate $L_2 \to L_\infty$ estimates for $e^{-t\cn}$ 
(and hence $L_p \to L_q$ estimates for all $1 \leq p \leq q \leq \infty$
by interpolation and duality).
This method is no longer applicable for operators with terms of order one 
or operators with only principal part but with complex coefficients.  

Considering the above problems in the setting of non-symmetric operators 
with real or complex coefficients (as in the first two theorems) 
is not a simple extension of  the methods used in the setting of symmetric  real-valued coefficients.
We had to develop new arguments since there are no  
available $L_2$-estimates which could serve us as a starting point.
We rely heavily on Gaussian and H\"older bounds for the 
heat kernel of the elliptic operator $A_D$ with Dirichlet boundary conditions.
These bounds are then used to prove  bounds 
(and H\"older bounds) for the Schwartz kernel of the harmonic lifting.
They are also used to prove that the Schwartz 
kernel $K_g$ of the commutator $[\cn, M_g]$ satisfies bounds
\[ 
|K_g(z,w) | \leq \frac{c}{|z-w|^{d-1}}
\]
and 
\[
| K_g(z,w) - K_g(z',w') | 
\leq c \, \frac{ (|z-z'| + |w-w'|)^\kappa}{ |z-w|^{d-1 + \kappa}}
\]
for all $z, z', w, w' \in \Gamma$ with 
$z \neq w$, $z' \neq w'$ and $|z-z'| + |w-w'| \leq \frac{1}{2} \, |z-w|$.
These bounds  serve us as a keystone to prove the boundedness of $[\cn, M_g]$ on  $C^\nu(\Gamma)$.
The boundedness on 
$L_2(\Gamma)$ is then obtained by the help of Krein's lemma.
The extension to $L_p(\Gamma)$ for $p \in (1, \infty)$ is obtained
by Calder\'on--Zygmund theory.
Using techniques of \cite{EO8} one can transfer certain invariant closed convex sets 
for the semigroup on $L_2(\Omega)$ generated by (minus-) the Neumann elliptic operator $A_N$
 to an invariant closed convex set 
for the semigroup generated by $-\cn$ on $L_2(\Gamma)$.
In that way we show bounds 
$\|e^{-t \cn}\|_{L_\infty(\Gamma) \to L_\infty(\Gamma)}
\leq M \, e^{\omega t}$
uniformly for $t \in (0,\infty)$ if the coefficients are real valued.
As in \cite{EO4} and \cite{EO6}, we use commutator estimates and $L_p \to L_q$ 
estimates of 
the semigroup $(e^{-t\cn})_{t > 0}$ (which are also  new in this setting) to prove 
$L_1 \to L_\infty$ estimates for the iterated commutator 
$[M_g, [\ldots ,[M_g, e^{-t\cn}] \ldots ]]$ of order $d$.
The Poisson bound is then obtained easily by optimising over the function~$g$. 

Concerning heat kernel and Green function bounds for the elliptic operator, 
we use various estimates in Morrey and 
Campanato spaces.
This part does not differ substantially from \cite{EO6} although a few arguments 
from that paper use self-adjointness of the operator $A_D$  and we have 
to provide different arguments. 

The outline of this paper is as follows.
In Section~\ref{Scom2} we give the definition  of Morrey and Campanato spaces and  make precise some notation which will be used throughout the paper.
Section~\ref{Scom3} is devoted to various elliptic regularity results.
Many of the results there are similar to those in \cite{EO6}, the difference 
is that we deal here with possibly non self-adjoint operators with in general complex coefficients.
Also as in \cite{EO6}, we use these elliptic regularity results to obtain Gaussian bounds, space-derivative bounds and H\"older continuity
for the heat kernel of the elliptic operator $A_D$.
This is done in Section~\ref{Scom4}.
Section~\ref{Scom5} is devoted to various regularity bounds for the Green function of $A_D$.
In Section~\ref{Scom6} we prove pointwise bounds and H\"older bounds for the kernel of the harmonic lifting~$\gamma$.  
We also prove in Theorem~\ref{tcom204}\ref{tcom204-4} the
$L_p(\Gamma) \to L_p(\Omega)$ bounds of Theorem~\ref{tcom102}\ref{tcom102-2} for~$\gamma$.
We emphasise that several results in Sections~\ref{Scom3}, \ref{Scom4} and \ref{Scom5},
and a few in Section~\ref{Scom6}, are also valid for unbounded $\Omega$
satisfying a local $C^{1+\kappa}$-property in a uniform manner.
The H\"older continuity bounds for $\gamma$ as stated in Theorem~\ref{tcom102}\ref{tcom102-1}
are proved in Section~\ref{Scom7}.
In  Section~\ref{Scom8} we prove global bounds and H\"older continuity for the 
Schwartz kernel of the Dirichlet-to-Neumann operator~$\cn$.
The commutator estimates of Theorem~\ref{tcom101} are proved in Section~\ref{Scom9}.
The Poisson bounds of Theorem~\ref{tcom103} together with some additional 
results are proved in Section~\ref{Scom10}.

\subsection*{Acknowledgements.}  
This project started in 2020 during a visit of the second named author at the 
university of Auckland and continued during two visits of the first named author 
at the university of Bordeaux.
 Both authors thank these institutions
for financial supports.
The second named author wishes to thank Yuri Tomilov who has drawn  his attention to 
Krein's lemma during a conference at the university of Bordeaux several years ago.

\section{Notation} \label{Scom2}

Since we consider $d \geq 2$ as fixed throughout the paper, we will not write that 
a given constant depends on~$d$.
Define  
\[
W_1 = \{ g \in W^{1,\infty}(\Ri^d, \Ri) : \|\nabla g\|_\infty \leq 1 \} 
\]
and 
\[
W_2 = \{ g \in W^{2,\infty}(\Ri^d, \Ri) : \|\nabla g\|_\infty \leq 1
\mbox{ and } \|\partial_k \, \partial_l g\|_\infty \leq 1 \mbox{ for all }
k,l \in \{ 1,\ldots,d \} \} 
.  \]

Let $\Omega \subset \Ri^d$ be an open set.
Note that we do not require $\Omega$ to be bounded.
For all $\nu \in (0,1]$ define $|||\cdot|||_{C^\nu(\Omega)} \colon C(\Omega) \to [0,\infty]$
by 
\[
|||u|||_{C^\nu(\Omega)}
= \sup_{\scriptstyle x,y \in \Omega \atop
        \scriptstyle 0 < |x-y| \leq 1 }
     \frac{|u(x) - u(y)|}{|x-y|^\nu}
\]
and set $C^\nu(\Omega) = \{ u \in C(\Omega) : |||u|||_{C^\nu(\Omega)} < \infty \} $.
Moreover, define 
\[
\|u\|_{C^\nu(\Omega)} 
= |||u|||_{C^\nu(\Omega)} + \|u\|_{L_\infty(\Omega)}
\in [0,\infty]
.  \]
We emphasise that elements in $C^\nu(\Omega)$ may be unbounded in general.

If $x \in \overline \Omega$ and $r > 0$, then define $\Omega(x,r) = \Omega \cap B(x,r)$.
For all $\gamma \in [0,d]$, $x \in \overline \Omega$ and $R_e \in (0,1]$
define $\|\cdot\|_{M,\gamma,x,\Omega,R_e} \colon L_2(\Omega) \to [0,\infty]$ by
\[
\|u\|_{M,\gamma,x,\Omega,R_e}
= \sup_{r \in (0,R_e]} \Big( r^{-\gamma} \int_{\Omega(x,r)} |u|^2 \Big)^{1/2}
.  \]
Moreover, define the {\bf Morrey norm}
$\|\cdot\|_{M,\gamma,R_e} \colon L_2(\Omega) \to [0,\infty]$ by 
\[
\|u\|_{M,\gamma,R_e} 
= \sup_{x \in \overline \Omega} \|u\|_{M,\gamma,x,\Omega,R_e}
.  \]
Set 
\[
M_\gamma(\Omega)
= \{ u \in L_2(\Omega) : \|u\|_{M,\gamma,1} < \infty \}
.  \]
Further, for all $\gamma \in [0,d+2]$, $x \in \overline \Omega$ and $R_e \in (0,1]$
define $|||\cdot|||_{\cm,\gamma,x,\Omega,R_e} \colon L_2(\Omega) \to [0,\infty]$ by
\[
|||u|||_{\cm,\gamma,x,\Omega,R_e}
= \sup_{r \in (0,R_e]} \Big( r^{-\gamma} \int_{\Omega(x,r)} |u - \langle u \rangle_{\Omega(x,r)}|^2 \Big)^{1/2}
.  \]
We define the {\bf Campanato seminorm} $|||\cdot|||_{\cm,\gamma,R_e} \colon L_2(\Omega) \to [0,\infty]$ by 
\[
|||u|||_{\cm,\gamma,R_e} 
= \sup_{x \in \overline \Omega} |||u|||_{\cm,\gamma,x,\Omega,R_e}
.  \]
Set 
\[
\cm_\gamma(\Omega)
= \{ u \in L_2(\Omega) : |||u|||_{\cm,\gamma,1} < \infty \}
.  \]
For explicit estimates between the various seminorms we refer to \cite{ERe2} Lemma~3.1.
If $R_e = 1$, then we drop the symbol $R_e$ in the notation.

Define 
\[
E = (-4,4)^d
\quad \mbox{and} \quad
E^- = (-4,4)^{d-1} \times (-4,0)
\]
the open cube in $\Ri^d$ and 
its lower half $E^-$.
The {\bf midplate} is $P = E \cap \{ x \in \Ri^d : x_d = 0 \} $.
Let $\kappa \in (0,1)$, $K > 0$ and $\Omega \subset \Ri^d$.
We say that $\Omega$ is of {\bf class $C^{1+\kappa}$ with parameter $K$}
if for all $x \in \partial \Omega$ there exists an open 
neighbourhood $U$ of $x$ and a 
$C^{1+\kappa}$-diffeomorphism $\Phi$ from $U$ onto $E$ such that 
\begin{itemize}
\item
$\Phi(x) = 0$, 
\item
$\Phi(U \cap \Omega) = E^-$
\item
$\Phi(U \cap \partial \Omega) = P$,
\item
$K$ is larger than the Lipschitz constant for $\Phi$ and $\Phi^{-1}$, and 
\item
$|||(D \Phi)_{ij}|||_{C^\kappa} \leq K$ and 
$|||(D (\Phi^{-1}))_{ij}|||_{C^\kappa} \leq K$ for all $i,j \in \{ 1,\ldots,d \} $
where $D \Phi$ denotes the derivative of $\Phi$.
\end{itemize}
We emphasise that $\Omega$ does not have to be bounded.
If $\Omega$ is of class $C^{1+\kappa}$ with parameter $K$, then there is a 
$\tilde c > 0$, depending only on $K$, such that 
$|\Omega(x,r)| \geq \tilde c \, r^d$ for all $x \in \overline \Omega$ and $r \in (0,1]$.
Further if $\Omega$ is of class $C^{1+\kappa}$ and $\nu \in (0,1]$, 
then every $u \in C^\nu(\Omega)$ can be extended to an element $\tilde u \in C^\nu(\overline \Omega)$.
We will identify $u$ and $\tilde u$ as no confusion is possible.
In particular, we occasionally write $u|_\Gamma$ for $\tilde u|_\Gamma$.

\smallskip

If $\gota \colon W^{1,2}(\Omega) \times W^{1,2}(\Omega) \to \Ci$ is as in 
(\ref{eScom1;1}) and $\lambda \in \Ri$, then we define the shifted form 
$\gota_\lambda \colon W^{1,2}(\Omega) \times W^{1,2}(\Omega) \to \Ci$ by 
\[
\gota_\lambda(u,v) = \gota(u,v) + \lambda \, (u,v)_{L_2(\Omega)}
.  \]
For a function $H$ of two variables, we use the notation 
$\partial^{(1)}_k H$ and $\partial^{(2)}_l H$ to denote the 
partial derivative with respect to the first and second variable in the $k$-th
and $l$-th direction, respectively. 
 
Finally, for a given bounded operator from an $L_p(X)$ space to an $L_q(Y)$ space, its norm is denoted  
$\|T\|_{L_p(X) \to L_q(Y)}$ or simply $\| T \|_{p\to q}$ if $X = Y$.

\section{Elliptic regularity} \label{Scom3}

We start with some elementary estimates regarding Morrey and Campanato seminorms.

\begin{lemma} \label{lcom217}
\mbox{}
\begin{tabel}
\item \label{lcom217-1}
Let $0 < R_1 \leq R_2 \leq 1$.
Then there is a $c > 0$ such that 
\[
\|u\|_{M,\gamma,R_1}
\leq \|u\|_{M,\gamma,R_2}
\leq c \, \|u\|_{M,\gamma,R_1}
\]
for all open $\Omega \subset \Ri^d$, $\gamma \in [0,d]$
and $u \in L_2(\Omega)$.
\item \label{lcom217-2}
If $\gamma \in [0,d+2]$, $0 < R_1 \leq R_2 \leq 1$ and $\Omega \subset \Ri^d$ is open, then 
\[
|||u|||_{\cm,\gamma,R_1}
\leq |||u|||_{\cm,\gamma,R_2}
\leq |||u|||_{\cm,\gamma,R_1} + R_1^{-\gamma} \, \|u\|_{M,0}
\]
for all $u \in L_2(\Omega)$.
\end{tabel}
\end{lemma}
\begin{proof}
`\ref{lcom217-1}'.
By an elementary estimate there is an $N \in \Ni$ such that 
for all $x \in \overline \Omega$ there exist $x_1,\ldots,x_N \in \overline \Omega$
such that $\Omega(x,1) \subset \bigcup_{n=1}^N \Omega(x_n,R_1)$.
Then Statement~\ref{lcom217-1} follows.

Statement~\ref{lcom217-2} is trivial.
\end{proof}

\begin{lemma} \label{lcom218}
Let $\kappa \in (0,1)$, $K > 0$ and $\gamma \in [0,d]$.
Then there are $c > 0$ and $R_e \in (0,1]$ such that the following is valid.

Let $\Omega \subset \Ri^d$ be an open set of class $C^{1+\kappa}$ with parameter $K$.
Then
\[
|||u|||_{\cm,\gamma+2,R_e}
\leq c \, \|\nabla u\|_{M,\gamma,R_e}
\]
for all $u \in W^{1,2}(\Omega)$.
\end{lemma}
\begin{proof}
It follows from \cite{EO6} Lemma~3.12 (also if $\gamma = d$) 
and two coordinate transformations
that $|||u|||_{\cm,\gamma+2,R_e} \leq c \, \|\nabla u\|_{M,\gamma,1}$
for suitable $R_e \in (0,1]$ and $c > 0$.
Then use Lemma~\ref{lcom217}\ref{lcom217-1}.
\end{proof}

We emphasise that $\Omega$ does not have to be bounded in the 
next four propositions.

\begin{prop} \label{pcom211}
Let $\kappa \in (0,1)$, $\mu,M,K > 0$, $\gamma \in [0,d)$ and $\delta \in [0,2]$ 
with $\gamma + \delta < d$.
Then there exists a $c > 0$ such that the following is valid.

Let $\Omega \subset \Ri^d$ be an open set of class $C^{1+\kappa}$ with parameter $K$.
For all $k,l \in \{ 1,\ldots,d \} $ let $c_{kl}\in C^\kappa(\Omega) \cap L_\infty(\Omega)$
with $\|c_{kl}\|_{C^\kappa(\Omega)} \leq M$.
Suppose $\RRe \sum_{k,l=1}^d c_{kl}(x) \, \xi_k \, \overline{\xi_l} \geq \mu \, |\xi|^2$
for all $\xi \in \Ci^d$ and $x \in \Omega$.
Let $f_0,f_1,\ldots,f_d \in L_2(\Omega)$ and $u \in W^{1,2}_0(\Omega)$.
Suppose
\[
\sum_{k,l=1}^d \int_\Omega c_{kl} \, (\partial_l u) \,  \, \overline{\partial_k v}
= (f_0,v)_{L_2(\Omega)} + \sum_{k=1}^d (f_k, \partial_k v)_{L_2(\Omega)}
\]
for all $v \in W^{1,2}_0(\Omega)$.
Then 
\[
\|\nabla u\|_{M,\gamma + \delta}
\leq c \, \Big( \varepsilon^{2-\delta} \, \|f_0\|_{M,\gamma}
   + \sum_{k=1}^d \|f_k\|_{M,\gamma + \delta}
   + \varepsilon^{-(\gamma + \delta)} \, \|\nabla u\|_{L_2(\Omega)}
           \Big)
\]
for all $\varepsilon \in (0,1]$.
\end{prop}
\begin{proof}
This follows from \cite{EO6} Propositions~3.9 and 3.10.
\end{proof}

Many statements in this article are formulated as in the previous proposition 
in the sense that the constants in the estimates  depend only on some fixed parameters. 
The main reason to do so is that at some places we will use  approximation arguments.  
The approximation process of  the coefficients by smooth ones requires uniform estimates 
in order to take  the limit.

\begin{prop} \label{pcom212}
Let $\kappa \in (0,1)$ and $\mu,M,K > 0$.
Then there exists a $c > 0$ such that the following is valid.

Let $\Omega \subset \Ri^d$ be an open set of class $C^{1+\kappa}$ with parameter $K$.
For all $k,l \in \{ 1,\ldots,d \} $ let $c_{kl} \in C^\kappa(\Omega) \cap L_\infty(\Omega)$,
with $\|c_{kl}\|_{C^\kappa(\Omega)} \leq M$.
Suppose $\RRe \sum_{k,l=1}^d c_{kl}(x) \, \xi_k \, \overline{\xi_l} \geq \mu \, |\xi|^2$
for all $\xi \in \Ci^d$ and $x \in \Omega$.
Let $f_0,f_1,\ldots,f_d \in L_2(\Omega)$ and $u \in W^{1,2}_0(\Omega)$.
Suppose
\[
\sum_{k,l=1}^d \int_\Omega c_{kl} \, (\partial_l u) \,  \, \overline{\partial_k v}
= (f_0,v)_{L_2(\Omega)} + \sum_{k=1}^d (f_k, \partial_k v)_{L_2(\Omega)}
\]
for all $v \in W^{1,2}_0(\Omega)$.
Then 
\[
\sum_{k=1}^d |||\partial_k u|||_{\cm,d + 2 \kappa}
\leq c \, \Big( \|f_0\|_{M,d + 2 \kappa - 2}
   + \sum_{k=1}^d (|||f_k|||_{\cm,d + 2 \kappa} + \|f_k\|_{L_2(\Omega)})
   + \|\nabla u\|_{L_2(\Omega)}
           \Big)
.  \]
\end{prop}
\begin{proof}
It follows from \cite{EO6} Proposition~3.7 (in particular the first 
displayed formula on page 4240) and \cite{EO6} Proposition~3.8 that 
there are $c_1,c_2 > 0$, depending only on $\kappa$, $\mu$, $M$, and $K$,
such that 
\begin{equation}
\sum_{k=1}^d |||\partial_k u|||_{\cm,d + \kappa}
\leq c_1 \, \Big( \|f_0\|_{M,d + \kappa - 2}
   + \sum_{k=1}^d |||f_k|||_{\cm,d + \kappa}
   + \|\nabla u\|_{M, d-\kappa}
           \Big)
\label{epcom212;2}
\end{equation}
and 
\begin{equation}
\sum_{k=1}^d |||\partial_k u|||_{\cm,d + 2 \kappa}
\leq c_2 \, \Big( \|f_0\|_{M,d + 2 \kappa - 2}
   + \sum_{k=1}^d |||f_k|||_{\cm,d + 2 \kappa}
   + \|\nabla u\|_{M, d}
           \Big)
.
\label{epcom212;1}
\end{equation}
We may assume that $\|f_0\|_{M,d + 2 \kappa - 2} < \infty$ and
$|||f_k|||_{\cm,d + 2 \kappa} < \infty$ for all $k \in \{ 1,\ldots,d \} $.
By Proposition~\ref{pcom211} and \cite{ERe2} Lemma~3.1(a) it follows that 
$\|\nabla u\|_{M, d-\kappa} < \infty$ and then (\ref{epcom212;2}) implies that 
$\partial_k u \in \cm_{d + \kappa}(\Omega) \subset C^{\kappa/2}(\Omega)$ 
for all $k \in \{ 1,\ldots,d \} $ by \cite{ERe2} Lemma~3.1(c).
By \cite{ERe2} Lemma~3.1(b) there exists a suitable $c_3 > 0$
such that 
\[
|\langle v \rangle_{\Omega(x,\rho)} - v(x)|
\leq c_3 \, \rho^{\kappa / 2} \, |||v|||_{\cm, d + \kappa}
\]
for all $x \in \Omega$, $\rho \in (0,1]$ and $v \in C(\Omega)$.
Choosing $\rho = 1$ gives
$|v(x)| \leq \sqrt{\omega_d} \, \|v\|_{L_2(\Omega)} + c_3 \, |||v|||_{\cm, d + \kappa}$.
Hence if $k \in \{ 1,\ldots,d \} $ and $x \in \Omega$, then 
\begin{eqnarray*}
|(\partial_k u)(x)|
& \leq & \sqrt{\omega_d} \, \|\nabla u\|_{L_2(\Omega)} 
   + c_3 \, c_1 \, \Big( \|f_0\|_{M,d + \kappa - 2}
   + \sum_{m=1}^d |||f_m|||_{\cm,d + \kappa}
   + \|\nabla u\|_{M, d-\kappa}
           \Big)  \\
& \leq & c_4  \, \Big( \|f_0\|_{M,d + \kappa - 2}
   + \sum_{m=1}^d (|||f_m|||_{\cm,d + \kappa} + \|f_m\|_{L_2(\Omega)})
   + \|\nabla u\|_{L_2(\Omega)}
           \Big)
\end{eqnarray*}
for a suitable $c_4 > 0$, where we used Proposition~\ref{pcom211}
to estimate $\|\nabla u\|_{M, d-\kappa}$
and \cite{ERe2} Lemma~3.1(a) to estimate $\|f_m\|_{M, d-\kappa}$.
Now $\|\partial_k u\|_{M,d} \leq \sqrt{\omega_d} \, \|\partial_k u\|_{L_\infty(\Omega)}$.
Using (\ref{epcom212;1}) completes the proof.
\end{proof}

We next consider operators with lower-order terms.

\begin{prop} \label{pcom213}
Let $\kappa \in (0,1)$, $\mu,M,K > 0$, $\gamma \in [0,d)$ and $\delta \in [0,2]$ 
with $\gamma + \delta < d$.
Then there exists a $c > 0$ such that the following is valid.

Let $\Omega \subset \Ri^d$ be an open set of class $C^{1+\kappa}$ with parameter $K$.
For all $k,l \in \{ 1,\ldots,d \} $ let $c_{kl}\in C^\kappa(\Omega) \cap L_\infty(\Omega)$
and let $b_k, c_k, c_0 \in L_\infty(\Omega)$, with $\|c_{kl}\|_{C^\kappa(\Omega)} \leq M$,
$\|b_k\|_{L_\infty(\Omega)} \leq M$, $\|c_k\|_{L_\infty(\Omega)} \leq M$
and $\|c_0\|_{L_\infty(\Omega)} \leq M$.
Suppose 
$\RRe \sum_{k,l=1}^d c_{kl}(x) \, \xi_k \, \overline{\xi_l} \geq \mu \, |\xi|^2$
for all $\xi \in \Ci^d$ and $x \in \Omega$.
Define $\gota$ as in {\rm (\ref{eScom1;1})}.

Let $f_0,f_1,\ldots,f_d \in L_2(\Omega)$ and $u \in W^{1,2}_0(\Omega)$.
Suppose
\[
\gota(u,v) 
= (f_0,v)_{L_2(\Omega)} + \sum_{k=1}^d (f_k, \partial_k v)_{L_2(\Omega)}
\]
for all $v \in W^{1,2}_0(\Omega)$.
Then 
\begin{equation}
\|\nabla u\|_{M,\gamma + \delta}
\leq c \, \Big( \varepsilon^{2-\delta} \, \|f_0\|_{M,\gamma}
   + \sum_{k=1}^d \|f_k\|_{M,\gamma + \delta}
   + \varepsilon^{-(\gamma + \delta)} \, \|u\|_{W^{1,2}(\Omega)}
           \Big)
\label{epcom213;1}
\end{equation}
for all $\varepsilon \in (0,1]$.

If, in addition, $\Omega$ is bounded and
\[
\frac{\mu}{2} \, \|u\|_{W^{1,2}(\Omega)}^2
\leq \RRe \gota(u,u)
,  \]
then there exists a $\tilde c > 0$, depending only on 
$\kappa$, $\mu$, $M$, $K$, $\gamma$, $\delta$ and $\diam \Omega$, such that 
\[
\|\nabla u\|_{M,\gamma + \delta}
\leq \tilde c \, \Big( \|f_0\|_{M,\gamma} + \sum_{k=1}^d \|f_k\|_{M,\gamma + \delta} \Big)
.  \]
\end{prop}
\begin{proof}
Let $f_0,f_1,\ldots,f_d \in L_2(\Omega)$ and $u \in W^{1,2}_0(\Omega)$.
Suppose
\[
\gota(u,v) 
= (f_0,v)_{L_2(\Omega)} + \sum_{k=1}^d (f_k, \partial_k v)_{L_2(\Omega)}
\]
for all $v \in W^{1,2}_0(\Omega)$.
Then
\[
\sum_{k,l=1}^d \int_\Omega c_{kl} \, (\partial_l u) \,  \, \overline{\partial_k v}
= (\tilde f_0, v)_{L_2(\Omega)} + \sum_{k=1}^d (\tilde f_k, \partial_k v)_{L_2(\Omega)}
\]
for all $v \in W^{1,2}_0(\Omega)$, where
\[
\tilde f_0 = f_0 - \sum_{l=1}^d c_l \, \partial_l u - c_0 \, u
\quad \mbox{and} \quad
\tilde f_k = f_k - b_k \, u
\]
for all $k \in \{ 1,\ldots,d \} $.

For all $\gamma \in [0,d)$ let $P(\gamma)$ be the hypothesis
\begin{quote}
for all $\delta \in [0,2]$ with $\gamma + \delta < d$ there exists a $c > 0$, 
depending only on $\kappa$, $\mu$, $M$, $K$, $\gamma$ and $\delta$, 
such that for all $\varepsilon \in (0,1]$ one has the estimate
\[
\|\nabla u\|_{M,\gamma + \delta}
\leq c \, \Big( \varepsilon^{2-\delta} \, \|f_0\|_{M,\gamma}
   + \sum_{k=1}^d \|f_k\|_{M,\gamma + \delta}
   + \varepsilon^{-(\gamma + \delta)} \, \|u\|_{W^{1,2}(\Omega)}
           \Big).
\]
\end{quote}

We start with the base case.

\begin{lemma} \label{lcom219}
The hypothesis $P(0)$ is valid.
\end{lemma}
\begin{proof}
Let $R_e \in (0,1]$ and $c > 0$ be as in Lemma~\ref{lcom218}
with the choice $\gamma = 0$.
By \cite{ERe2} Lemma~3.1(a) there is a suitable $c_1 > 0$ such that 
\[
\|u\|_{M,\delta,R_e}
\leq c_1 \, (|||u|||_{\cm,\delta,R_e} + \|u\|_{L_2(\Omega)})
.  \]
Then 
\begin{eqnarray*}
\|u\|_{M,\delta,R_e}
& \leq & c \, c_1 \, \|\nabla u\|_{M,0,R_e} + c_1 \, \|u\|_{L_2(\Omega)}  \\
& \leq & (1+c) \, c_1 \, \|u\|_{W^{1,2}(\Omega)}
.
\end{eqnarray*}
Hence by Lemma~\ref{lcom217}\ref{lcom217-1} there is a suitable $c_2 > 0$
such that 
\[
\|u\|_{M,\delta}
\leq c_2 \, \|u\|_{W^{1,2}(\Omega)}
.  \]
By Proposition~\ref{pcom211} there exists a suitable $c_3 > 0$ such that 
\[
\|\nabla u\|_{M,\delta}
\leq c_3 \, \Big( \varepsilon^{2-\delta} \, \|\tilde f_0\|_{M,0}
   + \sum_{k=1}^d \|\tilde f_k\|_{M,\delta}
   + \varepsilon^{-\delta} \, \|\nabla u\|_{L_2(\Omega)}
           \Big)
\]
for all $\varepsilon \in (0,1]$.
Now 
\begin{eqnarray*}
\|\tilde f_0\|_{M,0}
& \leq & \|f_0\|_{M,0} + d \, M \, \|\nabla u\|_{M,0} + M \, \|u\|_{M,0}  \\
& \leq & \|f_0\|_{M,0} + (d \, M + M \, c_2) \,  \|u\|_{W^{1,2}(\Omega)}
\end{eqnarray*}
and if $k \in \{ 1,\ldots,d \} $, then 
\[
\|\tilde f_k\|_{M,\delta}
\leq \|f_k\|_{M,\delta} + M \, \|u\|_{M,\delta}
\leq \|f_k\|_{M,\delta} + c_2 \, M \, \|u\|_{W^{1,2}(\Omega)}
.  \]
Then the lemma follows.
\end{proof}

The induction step is in the next lemma.

\begin{lemma} \label{lcom220}
Let $\gamma \in [0,d)$ and suppose that $P(\gamma)$ is valid.
Let $\tilde \delta \in [0,2]$ and suppose that $\gamma + \tilde \delta < d$.
Then $P(\gamma + \tilde \delta)$ is valid.
\end{lemma}
\begin{proof}
Let $\delta \in [0,2]$ and suppose that $\gamma + \tilde \delta + \delta < d$.
By hypothesis $P(\gamma)$ there is a suitable $c_1 > 0$ such that 
\[
\|\nabla u\|_{M,\gamma + \delta}
\leq c_1 \, \Big( \varepsilon^{2-\delta} \, \|f_0\|_{M,\gamma}
   + \sum_{k=1}^d \|f_k\|_{M,\gamma + \delta}
   + \varepsilon^{-(\gamma + \delta)} \, \|u\|_{W^{1,2}(\Omega)}
           \Big)
\]
for all $\varepsilon \in (0,1]$.
By Lemma~\ref{lcom218} there are $c_2 > 0$ and $R_e \in (0,1]$ such that 
\[
|||v|||_{\cm,\gamma + \delta + 2,R_e}
\leq c_2 \, \|\nabla v\|_{M,\gamma + \delta, R_e}
\]
for all $v \in W^{1,2}(\Omega)$.
By \cite{ERe2} Lemma~3.1(a) there is a suitable $c_3 > 0$ such that 
\[
\|u\|_{M,\gamma + \delta + \tilde \delta,R_e}
\leq c_3 \, (|||u|||_{\cm,\gamma + \delta + \tilde \delta,R_e} + \|u\|_{L_2(\Omega)})
.  \]
Then 
\begin{eqnarray*}
\|u\|_{M,\gamma + \delta + \tilde \delta, R_e}
& \leq & c_3 \, (|||u|||_{\cm,\gamma + \delta + 2,R_e} + \|u\|_{L_2(\Omega)})  \\
& \leq & c_3 \, c_2 \, \|\nabla u\|_{M,\gamma + \delta, R_e}
   + c_3 \, \|u\|_{L_2(\Omega)}  \\
& \leq & c_3 \, (1 + c_2 \, c_1) \, \Big( \varepsilon^{2-\delta} \, \|f_0\|_{M,\gamma}
   + \sum_{k=1}^d \|f_k\|_{M,\gamma + \delta}
   + \varepsilon^{-(\gamma + \delta)} \, \|u\|_{W^{1,2}(\Omega)}
                              \Big)  
\end{eqnarray*}
for all $\varepsilon \in (0,1]$.
Hence by Lemma~\ref{lcom217}\ref{lcom217-1} there is a suitable $c_4 > 0$ such that 
\begin{equation}
\|u\|_{M,\gamma + \delta + \tilde \delta}
\leq c_4 \, \Big( \varepsilon^{2-\delta} \, \|f_0\|_{M,\gamma}
   + \sum_{k=1}^d \|f_k\|_{M,\gamma + \delta}
   + \varepsilon^{-(\gamma + \delta)} \, \|u\|_{W^{1,2}(\Omega)}
                              \Big)
\label{elcom220;5}
\end{equation}
for all $\varepsilon \in (0,1]$.
By Proposition~\ref{pcom211} there exists a suitable $c_5 > 0$ such that 
\begin{equation}
\|\nabla u\|_{M,\gamma + \tilde \delta + \delta}
\leq c_5 \, \Big( \varepsilon^{2-\delta} \, \|\tilde f_0\|_{M,\gamma + \tilde \delta}
   + \sum_{k=1}^d \|\tilde f_k\|_{M,\gamma + \tilde \delta + \delta}
   + \varepsilon^{-(\gamma + \tilde \delta + \delta)} \, \|\nabla u\|_{L_2(\Omega)}
           \Big)
\label{elcom220;1}
\end{equation}
for all $\varepsilon \in (0,1]$.
Moreover, the hypothesis $P(\gamma)$ applied with $\tilde \delta$ instead of $\delta$
and with $\varepsilon = 1$ gives that there is a suitable $c_6 > 0$ such that 
\[
\|\nabla u\|_{M, \gamma + \tilde \delta}
\leq c_6 \, \Big( \|f_0\|_{M,\gamma} + \sum_{k=1}^d \|f_k\|_{M,\gamma + \tilde \delta} 
     + \|u\|_{W^{1,2}(\Omega)}
            \Big)
.  \]
Now 
\begin{eqnarray}
\|\tilde f_0\|_{M,\gamma + \tilde \delta}
& \leq & \|f_0\|_{M,\gamma + \tilde \delta} 
     + d \, M \, \|\nabla u\|_{M,\gamma + \tilde \delta}
     + M \, \|u\|_{M,\gamma + \tilde \delta} \nonumber  \\
& \leq & \|f_0\|_{M,\gamma + \tilde \delta} 
   + d \, M \, c_6 \, \Big( \|f_0\|_{M,\gamma} + \sum_{k=1}^d \|f_k\|_{M,\gamma + \tilde \delta} 
                           + \|u\|_{W^{1,2}(\Omega)}
            \Big)  \nonumber  \\*
& & {}
   + c_4 \, M \, \Big( \varepsilon^{2-\delta} \, \|f_0\|_{M,\gamma}
   + \sum_{k=1}^d \|f_k\|_{M,\gamma + \delta}
   + \varepsilon^{-(\gamma + \delta)} \, \|u\|_{W^{1,2}(\Omega)}
                              \Big)  \nonumber   \\
& \leq & c_7 \, \Big( \|f_0\|_{M,\gamma + \tilde \delta} 
        + \sum_{k=1}^d \|f_k\|_{M,\gamma + \tilde \delta + \delta}
        + \|u\|_{W^{1,2}(\Omega)} \Big) ,
\label{elcom220;2}
\end{eqnarray}
where $c_7 = 1 + c_6 \, d \, M + c_4 \, M$.
If $k \in \{ 1,\ldots,d \} $, then 
\begin{eqnarray}
\|\tilde f_k\|_{M,\gamma + \tilde \delta + \delta}
& \leq & \|f_k\|_{M,\gamma + \tilde \delta + \delta}
   + M \, \|u\|_{M,\gamma + \tilde \delta + \delta}  \nonumber  \\
& \leq & \|f_k\|_{M,\gamma + \tilde \delta + \delta}  \nonumber \\*
& & \hspace*{10mm} {}
   + c_4 \, M \, \Big( \varepsilon^{2-\delta} \, \|f_0\|_{M,\gamma}
        + \sum_{m=1}^d \|f_m\|_{M,\gamma + \delta}
        + \varepsilon^{-(\gamma + \delta)} \, \|u\|_{W^{1,2}(\Omega)}
                              \Big) 
. 
\label{elcom220;3}
\end{eqnarray}
Inserting (\ref{elcom220;2}) and (\ref{elcom220;3}) into (\ref{elcom220;1})
gives the lemma.
\end{proof}

Now we complete the proof of Proposition~\ref{pcom213}.

\medskip

The estimate (\ref{epcom213;1}) follows from Lemma~\ref{lcom219}
and iteration of Lemma~\ref{lcom220}.

Now suppose in addition that $\Omega$ is bounded and 
\[
\frac{\mu}{2} \, \|u\|_{W^{1,2}(\Omega)}^2
\leq \RRe \gota(u,u)
.  \]
Then there exists an $N \in \Ni$,
depending only on $\diam \Omega$, such that $\overline \Omega$ can be covered by 
at most $N$ balls of radius 1 with centre in~$\Omega$.
Therefore 
\begin{eqnarray*}
\|u\|_{W^{1,2}(\Omega)}
& \leq & \frac{2}{\mu} \, ( \|f_0\|_{L_2(\Omega)} + \sum_{k=1}^d \|f_k\|_{L_2(\Omega)} )  \\
& \leq & \frac{2N}{\mu} \, ( \|f_0\|_{M,0} + \sum_{k=1}^d \|f_k\|_{M,0} )  \\
& \leq & \frac{2N}{\mu} \, ( \|f_0\|_{M,\gamma} + \sum_{k=1}^d \|f_k\|_{M,\gamma + \delta} )
.  
\end{eqnarray*}
The proof is complete.
\end{proof}

For a similar estimate in the Campanato region we need an explicit bound 
on products of two H\"older continuous functions.

\begin{lemma} \label{lcom221}
Let $\gamma \in (d,d+2]$, $\Omega \subset \Ri^d$ open,
$u \in C^{(\gamma - d)/2}(\Omega) \cap L_\infty(\Omega)$ and $v \in L_2(\Omega)$.
Then 
\begin{eqnarray*}
|||u \, v|||_{\cm,\gamma}^2
& \leq & 2 \|u\|_{L_\infty(\Omega)}^2 \, |||v|||_{\cm,\gamma}^2
   + 2^{1 + \gamma - d} \, \omega_d \, \|v\|_{L_\infty(\Omega)}^2 \, |||u|||_{C^{(\gamma - d)/2}(\Omega)}^2  \\*
& & \hspace*{10mm} {}
   + 2^\gamma \, \omega_d \, \|u\|_{L_\infty(\Omega)}^2 \, \|v\|_{L_\infty(\Omega)}^2
.  
\end{eqnarray*}
\end{lemma}
\begin{proof}
Let $x \in \Omega$ and $r \in (0,\frac{1}{2}]$.
Then 
\begin{eqnarray*}
\int_{\Omega(x,r)} |u \, v - \langle u \, v \rangle_{\Omega(x,r)}|^2
& \leq & 2 \int_{\Omega(x,r)} \Big| \frac{1}{|\Omega(x,r)|} \, 
       \int_{\Omega(x,r)} u(y) \, (v(y) - v(z)) \, dz \Big|^2 \, dy  \\*
& & \hspace*{10mm} {}
   + 2 \int_{\Omega(x,r)} \Big| \frac{1}{|\Omega(x,r)|} \, 
       \int_{\Omega(x,r)} v(z) \, (u(y) - u(z)) \, dz \Big|^2 \, dy  \\
& \leq & 2 \|u\|_{L_\infty(\Omega)}^2 \, r^\gamma \, |||v|||_{\cm,\gamma}^2\\*
& & \hspace*{10mm} {}
   + 2 |\Omega(x,r)| \, 
         \Big( \|v\|_{L_\infty(\Omega)} \, |||u|||_{C^{(\gamma - d)/2}(\Omega)} \, 
               (2r)^{(\gamma - d)/2} \Big)^2  \\
& \leq & \Big( 2 \|u\|_{L_\infty(\Omega)}^2 \, |||v|||_{\cm,\gamma}^2
   + 2^{1 + \gamma - d} \, \omega_d \, \|v\|_{L_\infty(\Omega)}^2 \, |||u|||_{C^{(\gamma - d)/2}(\Omega)}^2
     \Big) \, r^\gamma
.
\end{eqnarray*}
Then the lemma follows easily.
\end{proof}

\begin{prop} \label{pcom214}
Let $\kappa \in (0,1)$ and $\mu,M,K > 0$.
Then there exists a $c > 0$ such that the following is valid.

Let $\Omega \subset \Ri^d$ be an open set of class $C^{1+\kappa}$ with parameter $K$.
For all $k,l \in \{ 1,\ldots,d \} $ let $c_{kl}, b_k\in C^\kappa(\Omega) \cap L_\infty(\Omega)$,
and let $c_k, c_0 \in L_\infty(\Omega)$, with $\|c_{kl}\|_{C^\kappa(\Omega)} \leq M$,
$\|b_k\|_{C^\kappa(\Omega)} \leq M$, $\|c_k\|_{L_\infty(\Omega)} \leq M$
and $\|c_0\|_{L_\infty(\Omega)} \leq M$.
Suppose $\RRe \sum_{k,l=1}^d c_{kl}(x) \, \xi_k \, \overline{\xi_l} \geq \mu \, |\xi|^2$
for all $\xi \in \Ci^d$ and $x \in \Omega$.
Define $\gota$ as in {\rm (\ref{eScom1;1})}.
Let $f_0,f_1,\ldots,f_d \in L_2(\Omega)$ and $u \in W^{1,2}_0(\Omega)$.
Suppose
\[
\gota(u,v) 
= (f_0,v)_{L_2(\Omega)} + \sum_{k=1}^d (f_k, \partial_k v)_{L_2(\Omega)}
\]
for all $v \in W^{1,2}_0(\Omega)$.
Then 
\[
\sum_{k=1}^d |||\partial_k u|||_{\cm,d + 2 \kappa}
\leq c \, \Big( \|f_0\|_{M,d + 2 \kappa - 2}
   + \sum_{k=1}^d ( |||f_k|||_{\cm,d + 2 \kappa} + \|f_k\|_{L_2(\Omega)} )
   + \|u\|_{W^{1,2}(\Omega)}
           \Big)
.  \]

If, in addition, $\Omega$ is bounded and
\begin{equation}
\frac{\mu}{2} \, \|u\|_{W^{1,2}(\Omega)}^2
\leq \RRe \gota(u,u)
,  
\label{epcom214;1}
\end{equation}
then there exists a $\tilde c > 0$, depending only on 
$\kappa$, $\mu$, $M$, $K$ and $\diam \Omega$, such that 
\[
\sum_{k=1}^d \|\partial_k u\|_{C^\kappa(\Omega)}
\leq \tilde c \, \Big( \|f_0\|_{M,d + 2 \kappa - 2}
   + \sum_{k=1}^d \|f_k\|_{C^\kappa(\Omega)} 
                 \Big)
.  \]
\end{prop}
\begin{proof}
We proceed as in the proof of Proposition~\ref{pcom213}.
Again 
\[
\sum_{k,l=1}^d \int_\Omega c_{kl} \, (\partial_l u) \,  \, \overline{\partial_k v}
= (\tilde f_0, v)_{L_2(\Omega)} + \sum_{k=1}^d (\tilde f_k, \partial_k v)_{L_2(\Omega)}
\]
for all $v \in W^{1,2}_0(\Omega)$, where
\[
\tilde f_0 = f_0 - \sum_{l=1}^d c_l \, \partial_l u - c_0 \, u
\quad \mbox{and} \quad
\tilde f_k = f_k - b_k \, u
\]
for all $k \in \{ 1,\ldots,d \} $.
By Proposition~\ref{pcom212} there is a suitable $c_1 > 0$ such that 
\[
\sum_{k=1}^d |||\partial_k u|||_{\cm,d + 2 \kappa}
\leq c_1 \, \Big( \|\tilde f_0\|_{M,d + 2 \kappa - 2}
   + \sum_{k=1}^d ( |||\tilde f_k|||_{\cm,d + 2 \kappa} + \|\tilde f_k\|_{L_2(\Omega)} )
   + \|\nabla u\|_{L_2(\Omega)}
           \Big)
.  \]
Now 
\[
\|\tilde f_0\|_{M,d + 2 \kappa - 2}
\leq \|f_0\|_{M,d + 2 \kappa - 2} + d \, M \, \|\nabla u\|_{M,d + 2 \kappa - 2}
     + M \, \|u\|_{M,d + 2 \kappa - 2}
.  \]
Since $d + 2 \kappa - 2 < d$ it follows from Proposition~\ref{pcom213} that there 
is a suitable $c_2 > 0$ such that 
\[
\|\nabla u\|_{M,d + 2 \kappa - 2}
\leq c_2 \, (\|f_0\|_{M, d + 2 \kappa - 2} + \sum_{k=1}^d \|f_k\|_{M, d + 2 \kappa - 2}
     + \|u\|_{W^{1,2}(\Omega)} )
.  \]
Arguing as in (\ref{elcom220;5}) there is a suitable $c_3 > 0$ such that 
\[
\|u\|_{M,d + 2 \kappa  - 2}
\leq c_3 \, (\|f_0\|_{M, d + 2 \kappa - 2} + \sum_{k=1}^d \|f_k\|_{M, d + 2 \kappa - 2}
     + \|u\|_{W^{1,2}(\Omega)} )
.  \]
If $k \in \{ 1,\ldots,d \} $, then 
\[
|||\tilde f_k|||_{\cm,d + 2 \kappa}
\leq |||f_k|||_{\cm,d + 2 \kappa} + |||b_k \, u|||_{\cm,d + 2 \kappa}
.  \]
By Lemma~\ref{lcom221} and \cite{ERe2} Lemma~3.1(c) there is a suitable $c_4 > 0$ such that 
\[
|||b_k \, u|||_{\cm,d + 2 \kappa}
\leq c_4 \, \|b_k\|_{C^\kappa(\Omega)} \, 
     ( |||u|||_{\cm,d + 2 \kappa} + \|u\|_{L_\infty(\Omega)} )
.  \]
By Lemma~\ref{lcom218} there are suitable $c_5 > 0$ and $R_e \in (0,1]$
such that 
\begin{eqnarray*}
|||u|||_{\cm,d + 2 \kappa, R_e}
& \leq & c_5 \, \|\nabla u\|_{M, d + 2 \kappa - 2}  \\
& \leq & c_2 \, c_5 (\|f_0\|_{M, d + 2 \kappa - 2} + \sum_{k=1}^d \|f_k\|_{M, d + 2 \kappa - 2}
     + \|u\|_{W^{1,2}(\Omega)} )
.  
\end{eqnarray*}
By \cite{ERe2} Lemma~3.1(b) there is a suitable $c_6 > 0$ such that 
\[
|\langle u \rangle_{\Omega(x,\rho)} - u(x)|
\leq c_6 \, \rho^\kappa \, |||u|||_{\cm,d + 2 \kappa, R_e}
\]
for all $\rho \in (0,R_e]$ and $x \in \Omega$.
Therefore 
\begin{equation}
\|u\|_{L_\infty(\Omega)}
\leq \sqrt{\omega_d} \, R_e^{d/2} \, \|u\|_{L_2(\Omega)}
   + c_6 \, R_e^\kappa \, |||u|||_{\cm,d + 2 \kappa, R_e}
.  
\label{epcom214;2}
\end{equation}
Estimating as before gives the first part of the proposition.

If $\Omega$ is bounded and (\ref{epcom214;1}) is valid,
then one can bound $\|u\|_{W^{1,2}(\Omega)}$ as at the end of the 
proof of Proposition~\ref{pcom213}.
The $C^\kappa(\Omega)$-norm of $\partial_k u$ can be bounded by a suitable 
multiple of $|||\partial_k u|||_{\cm, d+2\kappa} + \|\partial_k u\|_{L_\infty(\Omega)}$
and then as in (\ref{epcom214;2}) by 
$|||\partial_k u|||_{\cm, d+2\kappa} + \|\partial_k u\|_{L_2(\Omega)}$.
Further one can bound the Campanato seminorm $|||f_k|||_{\cm, d+2\kappa}$
by a suitable multiple of $\|f_k\|_{C^\kappa(\Omega)}$.
\end{proof}

\section{H\"older-Gaussian bounds} \label{Scom4}

This section is devoted to bounds and regularity of the heat kernel of the 
elliptic operator $A_D$ subject to Dirichlet boundary conditions. 
We prove Gaussian bounds, H\"older regularity as well as $L_p$ estimates on  the boundary. 

The set $\Omega$ is allowed to be unbounded in the next theorem.

\begin{thm} \label{tcom201}
Let $\kappa,\tau' \in (0,1)$ and $\mu,M,\tau,K > 0$.
Then there exist $a,b > 0$ and $\omega \in \Ri$ such that 
the following is valid.

Let $\Omega \subset \Ri^d$ be an open set of class $C^{1+\kappa}$ with  parameter $K$.
For all $k,l \in \{ 1,\ldots,d \} $ let $c_{kl}, b_k, c_k \in C^\kappa(\Omega) \cap L_\infty(\Omega)$,
and let $c_0 \in L_\infty(\Omega)$, with $\|c_{kl}\|_{C^\kappa(\Omega)} \leq M$,
$\|b_k\|_{C^\kappa(\Omega)} \leq M$, $\|c_k\|_{C^\kappa(\Omega)} \leq M$
and $\|c_0\|_{L_\infty(\Omega)} \leq M$.
Suppose $\RRe \sum_{k,l=1}^d c_{kl}(x) \, \xi_k \, \overline{\xi_l} \geq \mu \, |\xi|^2$
for all $\xi \in \Ci^d$ and $x \in \Omega$.
Then there exists 
a function $(t,x,y) \mapsto H_t(x,y)$ from 
$(0,\infty) \times \Omega \times \Omega$ into $\Ci$
such that the following is valid.
\begin{tabel}
\item \label{tcom201-1}
The function $(t,x,y) \mapsto H_t(x,y)$ is continuous 
from $(0,\infty) \times \Omega \times \Omega$ into $\Ci$.
\item \label{tcom201-3}
For all $t \in (0,\infty)$ the
function $H_t$ is once differentiable in each variable and the 
derivative with respect to one variable is differentiable in the 
other variable.
Moreover, for every multi-index $\alpha,\beta$ with 
$0 \leq |\alpha|,|\beta| \leq 1$ one has 
\begin{equation}
|(\partial_x^\alpha \, \partial_y^\beta \, H_t)(x,y)|
\leq a \, t^{-d/2} \, t^{-(|\alpha| + |\beta|)/2} \, 
      e^{-b \, \frac{|x-y|^2}{t}} \, e^{\omega t}
\label{etcom201;1}
\end{equation}
and 
\begin{eqnarray*}
\lefteqn{
|(\partial_x^\alpha \, \partial_y^\beta \, H_t)(x+h,y+k) 
     - (\partial_x^\alpha \, \partial_y^\beta \, H_t)(x,y)|
} \hspace{20mm}  \\*
& \leq & a \, t^{-d/2} \, t^{-(|\alpha| + |\beta|)/2} \, 
    \left( \frac{|h| + |k|}{\sqrt{t} + |x-y|} \right)^\kappa \, 
      e^{-b \, \frac{|x-y|^2}{t}}  \, e^{\omega t}
\end{eqnarray*}
for all $x,y \in \Omega$ and $h,k \in \Ri^d$ with $x+h,y+k \in \Omega$ and 
$|h| + |k| \leq \tau \, \sqrt{t} + \tau' \, |x-y|$.
\item \label{tcom201-2}
For all $t \in (0,\infty)$ the function $H_t$ is the kernel of 
the operator $e^{-t A_D}$.
\item \label{tcom201-6}
$\|\partial^\alpha \,  e^{-t A_D}\|_{L_p(\Omega) \to L_q(\Omega)}
\leq a \, t^{- |\alpha| / 2} \, t^{- \frac{d}{2} (\frac{1}{p} - \frac{1}{q}) } \, e^{\omega t}$
for all $t > 0$, multi-index $\alpha$ and $p,q \in [1,\infty]$ with $p \leq q$
and $|\alpha| \leq 1$.
\item \label{tcom201-5}
$\|\partial^\alpha \,  e^{-t A_D}\|_{L_p(\Omega) \to C^\nu(\Omega)}
\leq a \, t^{-|\alpha| / 2} t^{- \frac{d}{2p}} \, t^{- \nu / 2} \, e^{\omega t}$
for all $t > 0$, $p \in [1,\infty]$, multi-index $\alpha$ 
and $\nu \in (0,\kappa]$ with $|\alpha| \leq 1$.
\item \label{tcom201-4}
If, in addition,
\[
\frac{\mu}{2} \|u\|_{W^{1,2}(\Omega)}^2
\leq \RRe \gota(u,u)
\]
for all $u \in W^{1,2}(\Omega)$, then 
one can choose $\omega = - \frac{\mu}{4} < 0$.
\end{tabel}
\end{thm}
\begin{proof}
`\ref{tcom201-1}--\ref{tcom201-3}'. 
If $b_k = c_k = 0$, then this is Theorem~3.1 in \cite{EO6}. 
It is stated there for bounded $\Omega$ but 
we proved in the previous section all the necessary results which allow us to adopt 
the proof of \cite{EO6} to the setting of this theorem.
Cf.\ Theorem~3.16 in \cite{EO6}.

`\ref{tcom201-4}'. 
Note that the coercivity bound implies that 
$\|e^{-t A_D}\|_{L_2(\Omega) \to L_2(\Omega)}^2 \leq e^{- \mu t}$ for all $t > 0$.
We can repeat the same arguments to obtain uniform bounds (with respect to $t$) 
for $A_D + \omega$ with $\omega = - \frac{\mu}{4}$.

`\ref{tcom201-6}'.
The Gaussian bounds (\ref{etcom201;1}) give
\[
\|\partial^\alpha \,  e^{-t A_D}\|_{L_p(\Omega) \to L_p(\Omega)}
\leq a \, t^{-|\alpha|/2} \, e^{\omega t} \, \int_{\Ri^d} 
    t^{-d/2} \,  e^{-b \, \frac{|x-y|^2}{t}} \, dx
= a \, (\pi/b)^{d/2} \, t^{-|\alpha|/2} \, e^{\omega t} 
\]
for all $p \in [1,\infty]$.
Using H\"older's estimate or Young's inequality one estimates similarly 
\[
|(\partial^\alpha \,  e^{-t A_D} u)(x)|
\leq a \, \Big( (\pi/b)^{d/2} \vee 1 \Big) \, t^{-|\alpha|/2} \, 
     t^{-\frac{d}{2p}} \, e^{\omega t} \, \|u\|_{L_p(\Omega)}
\]
for all $p \in [1,\infty]$, $u \in L_p(\Omega)$ and $x \in \Omega$.
Then 
\[
\|\partial^\alpha \,  e^{-t A_D}\|_{L_p(\Omega) \to L_q(\Omega)}
\leq a \, \, \Big( (\pi/b)^{d/2} \vee 1 \Big) \, t^{-|\alpha| / 2} \, 
          t^{-\frac{d}{2} (\frac{1}{p} - \frac{1}{q}) } \, e^{\omega t}
\]
by interpolation for all $p,q \in [1,\infty]$ with $p \leq q$
and $|\alpha| \leq 1$.
Note that the constants $a$, $b$ and in particular $\omega$ are the same 
as in Statement~\ref{tcom201-3}.

`\ref{tcom201-5}'.
Choosing $\tau = 1$, then the bounds in Statement~\ref{tcom201-3} are valid 
with certain values of $a$, $b$ and $\omega$ for all $h$ whenever $|h| \leq \sqrt{t}$. 
Hence if $x,x' \in \Omega$ with $|x-x'| \leq \sqrt{t}$, 
then as in the proof of Statement~\ref{tcom201-6} the H\"older Gaussian bounds give
\[
|(\partial^\alpha e^{-t A_D} u)(x) - (\partial^\alpha e^{-t A_D} u)(x')|
\leq a \, \Big( (\pi/b)^{d/2} \vee 1 \Big) \, t^{-|\alpha|/2} \, t^{-\frac{d}{2p}} \, 
    \Big( \frac{|x-x'|}{\sqrt{t}} \Big)^\kappa \, e^{\omega t} \, \|u\|_{L_p(\Omega)}
.  \]
Alternatively, if $|x-x'| \geq \sqrt{t}$, then 
\begin{eqnarray*}
|(\partial^\alpha e^{-t A_D} u)(x) - (\partial^\alpha e^{-t A_D} u)(x')|
& \leq & 2 \Big( \frac{|x-x'|}{\sqrt{t}} \Big)^\kappa \, \|\partial^\alpha e^{-t A_D} u\|_{L_\infty(\Omega)}  \\
& \leq & 2 a \, \Big( (\pi/b)^{d/2} \vee 1 \Big) \, t^{-|\alpha|/2} \, t^{-\frac{d}{2p}} \, 
    \Big( \frac{|x-x'|}{\sqrt{t}} \Big)^\kappa \, e^{\omega t} \, \|u\|_{L_p(\Omega)}
\end{eqnarray*}
and the bound from $L_p(\Omega)$ into $C^\kappa(\Omega)$ follows.
Finally, the bounds in Statement~\ref{tcom201-3} are obviously valid with $\nu$ instead of $\kappa$ 
for any $\nu \in (0, \kappa)$.  
Thus we obtain the bound from $L_p(\Omega)$ into $C^\nu(\Omega)$.
\end{proof}

For bounded $\Omega$ one also has estimates for the semigroup
from $L_p(\Omega)$ into $L_p(\Gamma)$.

\begin{prop} \label{pcom201.5}
Let $\kappa \in (0,1)$.
Let $\Omega \subset \Ri^d$ be a bounded open set of class~$C^{1+\kappa}$.
For all $k,l \in \{ 1,\ldots,d \} $ let $c_{kl}, b_k, c_k \in C^\kappa(\Omega)$,
and let $c_0 \in L_\infty(\Omega)$.
Suppose there is a $\mu > 0$ such that 
$\RRe \sum_{k,l=1}^d c_{kl}(x) \, \xi_k \, \overline{\xi_l} \geq \mu \, |\xi|^2$
for all $\xi \in \Ci^d$ and $x \in \Omega$.
Then there exist $a,\omega > 0$ such that 
\[
\|\partial_k \,  e^{-t A_D}\|_{L_p(\Omega) \to L_p(\Gamma)}
\leq a \, t^{- \frac{1}{2} - \frac{1}{2p}} \, e^{\omega t}
\]
for all $k \in \{ 1,\ldots,d \} $, $t > 0$ and $p \in [1,\infty]$.
\end{prop}
\begin{proof}
By Theorem~\ref{tcom201} the semigroup $(e^{-t A_D})_{t > 0}$ has a kernel
$(H_t)_{t > 0}$ which is differentiable in the first entry and there are 
$a,b,\omega > 0$ such that 
\[
|(\partial^{(1)}_k H_t)(x,y)|
\leq a \, t^{-d/2} \, t^{-1/2} \, 
      e^{-b \, \frac{|x-y|^2}{t}} \, e^{\omega t}
\]
for all $t > 0$, $x,y \in \Omega$ and $k \in \{ 1,\ldots,d \} $.

Let $k \in \{ 1,\ldots,d \} $.
We use that 
$\|\partial_k \,  e^{-t A_D}\|_{L_1(\Omega) \to L_1(\Gamma)}
\leq \sup_{y \in \Omega} \int_\Gamma |(\partial^{(1)}_k H_t)(z,y)| \, dz$.
Let $z_0 \in \Gamma$.
After a suitable rotation and translation
we may assume that there are $\delta > 0$ 
and a $C^1_b$-function $\tau \colon \Ri^{d-1} \to \Ri$ such that $z_0 = 0$
and $\Gamma \cap [-\delta,\delta]^d = \{ (\xi,\tau(\xi)) : \xi \in [-\delta,\delta]^{d-1} \} $.
Let $y = (\tilde y,y_d) \in \Omega$.
Then 
\begin{eqnarray*}
\lefteqn{
\int_{\Gamma \cap [-\delta,\delta]^d} |(\partial^{(1)}_k H_t)(z,y)| \, dz
} \hspace*{10mm} \\*
& \leq & \int_{[-\delta,\delta]^{d-1}} 
    a \, t^{-d/2} \, t^{-1/2} \, 
      e^{-b \, \frac{|\xi - \tilde y|^2}{t}} \, e^{-b \, \frac{|\tau(\xi) - y_d|^2}{t}} \, 
   e^{\omega t} \, \sqrt{1 + |(\nabla \tau)(\xi)|^2} \, d\xi  \\
& \leq & \int_{\Ri^{d-1}} 
    a \, t^{-d/2} \, t^{-1/2} \, 
      e^{-b \, \frac{|\xi - \tilde y|^2}{t}} \,  
   e^{\omega t} \, \sqrt{1 + \|\nabla \tau\|_\infty^2} \, d\xi  \\
& = & a \, t^{-1} \, (\pi / b)^{(d-1)/2} \, \sqrt{1 + \|\nabla \tau\|_\infty^2} \, e^{\omega t}
.
\end{eqnarray*}
Covering $\Gamma$ with a finite number of these type of sets, one deduces that 
there is an $a_1 > 0$ such that 
\[
\|\partial_k \,  e^{-t A_D}\|_{L_1(\Omega) \to L_1(\Gamma)}
\leq \sup_{y \in \Omega} \int_\Gamma |(\partial^{(1)}_k H_t)(z,y)| \, dz
\leq a_1 \, t^{-1} \, e^{\omega t}
\] 
for all $t > 0$.

Much easier one obtains
\[
\|\partial_k \,  e^{-t A_D}\|_{L_\infty(\Omega) \to L_\infty(\Gamma)}
\leq \sup_{z \in \Gamma} \, \int_\Omega |(\partial^{(1)}_k H_t)(z,y)| \, dy
\leq a \, (\pi/b)^{d/2} t^{-1/2} \, e^{\omega t}
\]
for all $t > 0$.
Then the proposition follows by interpolation.
\end{proof}

The idea of using Morrey and Campanato  spaces in order to obtain Gaussian upper bounds and  H\"older regularity for heat kernels of elliptic operators on $\Ri^d$  was used in  \cite{aus1}. 
We refer to \cite{EO6} for more details and other references on this subject.

\section{H\"older bounds for the Green function} \label{Scom5}

We consider now the  Green function associated with the elliptic operator. 
We prove upper bounds and H\"older continuity. 
We present two versions. 
The first is valid for possibly unbounded $\Omega$, but requires an
additional coercivity condition.
The second version does not require the additional  coercivity condition,
but then we use that $\Omega$ is bounded.

\begin{thm} \label{tcom202}
Let $\kappa \in (0,1)$ and $\mu,M,K  > 0$.
Then there exists a $c > 0$ such that the following is valid.

Let $\Omega \subset \Ri^d$ be an open set of class $C^{1+\kappa}$ with parameter $K$.
For all $k,l \in \{ 1,\ldots,d \} $ let $c_{kl}, b_k, c_k \in C^\kappa(\Omega) \cap L_\infty(\Omega)$,
and let $c_0 \in L_\infty(\Omega)$
with $\|c_{kl}\|_{C^\kappa(\Omega)} \leq M$,
$\|b_k\|_{C^\kappa(\Omega)} \leq M$, $\|c_k\|_{C^\kappa(\Omega)} \leq M$
and $\|c_0\|_{L_\infty(\Omega)} \leq M$.
Suppose 
$\RRe \sum_{k,l=1}^d c_{kl}(x) \, \xi_k \, \overline{\xi_l} \geq \mu \, |\xi|^2$
for all $\xi \in \Ci^d$ and $x \in \Omega$.
Further suppose that 
\begin{equation}
\frac{\mu}{2} \, \|u\|_{W^{1,2}(\Omega)}^2
\leq \RRe \gota(u,u)
\label{etcom202;1}
\end{equation}
for all $u \in W^{1,2}(\Omega)$.

Then the operator $A_D^{-1}$ has 
a kernel $G \colon \{ (x,y) \in \Omega \times \Omega : x \neq y \} \to \Ci$, 
which is differentiable in each variable and the derivative is 
differentiable in the other variable.
Moreover, for every multi-index $\alpha,\beta$ with 
$0 \leq |\alpha|,|\beta| \leq 1$ the function
$\partial_x^\alpha \, \partial_y^\beta \, G$ 
extends to a continuous function on 
$ \{ (x,y) \in \overline \Omega \times \overline \Omega : x \neq y \} $
with estimates
\[
|(\partial_x^\alpha \, \partial_y^\beta \, G)(x,y)| 
\leq \left\{ \begin{array}{ll}
c \, |x-y|^{-(d-2+|\alpha| + |\beta|)}  & \mbox{if } d-2+|\alpha| + |\beta| \neq 0 \\[5pt]
c \, \log( 1 + \frac{1}{|x-y|} ) & \mbox{if } d-2+|\alpha| + |\beta| = 0
                 \end{array} \right.
\]
and
\[
|(\partial_x^\alpha \, \partial_y^\beta \, G)(x',y') - (\partial_x^\alpha \, \partial_y^\beta \, G)(x,y)| 
\leq c \, \frac{(|x'-x| + |y'-y|)^\kappa}{|x-y|^{d-2+|\alpha| + |\beta| +\kappa}}  
\]
for all $x,x',y,y' \in \Omega$ with $x \neq y$, $x' \neq y'$ and
$|x-x'| + |y-y'| \leq \frac{1}{2} \, |x-y|$.
\end{thm}
\begin{proof}
The additional assumption (\ref{etcom202;1}) is the condition of Theorem~\ref{tcom201}\ref{tcom201-4}.
Choose $\tau = 1$ and $\tau' = \frac{1}{2}$ in Theorem~\ref{tcom201}.
Therefore one can choose $\omega < 0$ in Theorem~\ref{tcom201}\ref{tcom201-3}.
Then the estimates follow easily by the Laplace transform as in the proof of Theorem~4.1 in \cite{EO6}.
\end{proof}

If merely $0 \not\in \sigma(A_D)$ one  can also 
obtain fairly uniform Green kernel bounds.

\begin{thm} \label{tcom202.5}
Let $\kappa \in (0,1)$ and $\mu,M,K,\widetilde K,R > 0$.
Then there exists a $c > 0$ such that the following is valid.

Let $\Omega \subset \Ri^d$ be a open bounded set of class $C^{1+\kappa}$ with  parameter $K$
and $\diam \Omega \leq R$.
For all $k,l \in \{ 1,\ldots,d \} $ let $c_{kl}, b_k, c_k \in C^\kappa(\Omega) \cap L_\infty(\Omega)$,
and let $c_0 \in L_\infty(\Omega)$
with $\|c_{kl}\|_{C^\kappa(\Omega)} \leq M$,
$\|b_k\|_{C^\kappa(\Omega)} \leq M$, $\|c_k\|_{C^\kappa(\Omega)} \leq M$
and $\|c_0\|_{L_\infty(\Omega)} \leq M$.
Suppose 
\[
\RRe \sum_{k,l=1}^d c_{kl}(x) \, \xi_k \, \overline{\xi_l} \geq \mu \, |\xi|^2
\]
for all $\xi \in \Ci^d$ and $x \in \Omega$.
Further suppose $0 \notin  \sigma(A_D)$ and $\|A_D^{-1}\|_{2 \to 2} \leq \widetilde K$.

Then the operator $A_D^{-1}$ has 
a kernel $G \colon \{ (x,y) \in \Omega \times \Omega : x \neq y \} \to \Ci$, 
which is differentiable in each variable and the derivative is 
differentiable in the other variable.
Moreover, for every multi-index $\alpha,\beta$ with 
$0 \leq |\alpha|,|\beta| \leq 1$ the function
$\partial_x^\alpha \, \partial_y^\beta \, G$ 
extends to a continuous function on 
$ \{ (x,y) \in \overline \Omega \times \overline \Omega : x \neq y \} $
with estimates
\begin{equation}
|(\partial_x^\alpha \, \partial_y^\beta \, G)(x,y)| 
\leq \left\{ \begin{array}{ll}
c \, |x-y|^{-(d-2+|\alpha| + |\beta|)}  & \mbox{if } d-2+|\alpha| + |\beta| \neq 0 \\[5pt]
c \, \log( 1 + \frac{1}{|x-y|} ) & \mbox{if } d-2+|\alpha| + |\beta| = 0
                 \end{array} \right.
\label{etcom202.5;1}
\end{equation}
and
\[
|(\partial_x^\alpha \, \partial_y^\beta \, G)(x',y') - (\partial_x^\alpha \, \partial_y^\beta \, G)(x,y)| 
\leq c \, \frac{(|x'-x| + |y'-y|)^\kappa}{|x-y|^{d-2+|\alpha| + |\beta| +\kappa}}  
\]
for all $x,x',y,y' \in \Omega$ with $x \neq y$, $x' \neq y'$ and
$|x-x'| + |y-y'| \leq \frac{1}{2} \, |x-y|$.
\end{thm}
\begin{proof}
Let $a, b > 0$ and $\omega \in \Ri$ be as in Theorem~\ref{tcom201}
with the choice $\tau = 1$ and $\tau' = \frac{1}{2}$.
Choose $\lambda = \omega + 2$ and consider the operator $A_D + \lambda \, I$.

The resolvent equation gives 
\[
A_D^{-1} - (A_D + \lambda \, I)^{-1}
= \lambda \, (A_D + \lambda \, I)^{-1} \, A_D^{-1}
.  \]
Hence by iteration
\begin{eqnarray}
A_D^{-1}
& = & \sum_{k=1}^{2d-1} \lambda^{k-1} \, (A_D + \lambda \, I)^{-k}
   + \lambda^{2d-1} \, (A_D + \lambda \, I)^{-2d} \, A_D^{-1}  \nonumber  \\
& = & \sum_{k=1}^{2d-1} \lambda^{k-1} \, (A_D + \lambda \, I)^{-k}
   + \lambda^{2d-1} \, (A_D + \lambda \, I)^{-d} \, A_D^{-1} \, (A_D + \lambda \, I)^{-d}
. 
\label{etcom202;2}
\end{eqnarray}
If $m \in \Ni$, then 
\[
(A_D + \lambda \, I)^{-m} 
= \frac{1}{(m-1)!} \int_0^\infty t^{m-1} \, e^{-t (A_D + \lambda \, I)} \, dt
.  \]
Since $t^{m-1} \leq (m-1)! \, \varepsilon^{-(m-1)} \, e^{\varepsilon t}$
for all $\varepsilon > 0$, one can estimate the kernel of 
$(A_D + \lambda \, I)^{-m}$ as in the previous theorem.

It remains to bound the last term in (\ref{etcom202;2}).
Recall that 
\[
\|\partial^\alpha \, e^{-t A_D}\|_{L_2(\Omega) \to C^\kappa(\Omega)}
\leq a \, t^{-|\alpha|/2} \, t^{-d/4} \, t^{-\kappa/2} \, e^{\omega t}
\]
for all $t > 0$ and for every multi-index $\alpha$ with $|\alpha| \leq 1$
by Theorem~\ref{tcom201}\ref{tcom201-5}.
Hence 
\begin{eqnarray*}
\|\partial^\alpha \, (A_D + \lambda \, I)^{-d}\|_{L_2(\Omega) \to C^\kappa(\Omega)}
& = & \frac{1}{(d-1)!} \, 
   \Big\| \int_0^\infty t^{d-1} \, e^{-\lambda t} \, \partial^\alpha \, e^{-t A_D} \, dt
   \Big\|_{L_2(\Omega) \to C^\kappa(\Omega)}  \\
& \leq & c_\alpha 
= \frac{1}{(d-1)!} \, \int_0^\infty t^{d-1} \, \, 
    a \, t^{-|\alpha|/2} \, t^{-d/4} \, t^{-\kappa/2} \, e^{-2 t} \, dt
.
\end{eqnarray*}
Similarly, Theorem~\ref{tcom201}\ref{tcom201-6} gives a bound
\[
\|\partial^\alpha \, (A_D + \lambda \, I)^{-d}\|_{L_2(\Omega) \to C_b(\Omega)}
\leq \widetilde c_\alpha 
= \frac{1}{(d-1)!} \, \int_0^\infty t^{d-1} \, \, 
    a \, t^{-|\alpha|/2} \, t^{-d/4} \, e^{-2 t} \, dt
.  \]
Without loss of generality we may assume that 
\[
\|\partial^\alpha \, (A_D^* + \lambda \, I)^{-d}\|_{L_2(\Omega) \to C^\kappa(\Omega)}
\leq c_\alpha 
\]
and 
\[
\|\partial^\alpha \, (A_D^* + \lambda \, I)^{-d}\|_{L_2(\Omega) \to C_b(\Omega)}
\leq \tilde c_\alpha 
\]
for the two adjoint operators.

Next let $\alpha,\beta$ be two multi-indices with $|\alpha|, |\beta| \leq 1$.
By the Riesz representation theorem
there exists a function $K^\alpha \colon \Omega \times \Omega \to \Ci$ such that 
$K^\alpha(x, \cdot) \in L_2(\Omega)$ and 
\[
(\partial^\alpha \, (A_D + \lambda \, I)^{-d} \, A_D^{-1} u)(x)
= (u, \overline{K^\alpha(x, \cdot)})_{L_2(\Omega)}
\]
for all $u \in L_2(\Omega)$ and $x \in \Omega$.
Then 
\[
(\partial^\alpha \, (A_D + \lambda \, I)^{-d} \, A_D^{-1} u)(x)
= \int_\Omega K^\alpha(x,y) \, u(y) \, dy
\]
for all $u \in L_2(\Omega)$ and $x \in \Omega$.
Similarly, there exists a function $K_\beta \colon \Omega \times \Omega \to \Ci$ such that 
$K_\beta(x, \cdot) \in L_2(\Omega)$ and 
\[
(\partial^\beta \, (A_D^* + \lambda \, I)^{-d} u)(x)
= \int_\Omega K_\beta(x,y) \, u(y) \, dy
\]
for all $u \in L_2(\Omega)$ and $x \in \Omega$.
Then 
\[
((\partial^\beta \, (A_D^* + \lambda \, I)^{-d})^* u)(x)
= \int_\Omega \overline{K_\beta(y,x)} \, u(y) \, dy
\]
for all $u \in L_2(\Omega)$ and $x \in \Omega$.
Define $K^{(\alpha,\beta)} \colon \Omega \times \Omega \to \Ci$ by 
\[
K^{(\alpha,\beta)}(x,y)
= (-1)^{|\beta|} \, \int_\Omega K^\alpha(x,z) \, \overline{K_\beta(y,z)} \, dz
= (-1)^{|\beta|} \, (K^\alpha(x,\cdot) , K_\beta(y,\cdot))_{L_2(\Omega)}
.  \]
Then 
\[
\Big( \partial^\alpha \, (A_D + \lambda \, I)^{-d} \, A_D^{-1} \, 
                 (A_D + \lambda \, I)^{-d} \, \partial^\beta u \Big)(x)
= \int_\Omega K^{(\alpha,\beta)}(x,y) \, u(y) \, dy
\]
for all $u \in L_2(\Omega)$ (see also (1) in \cite{AE11}).
By the Cauchy--Schwarz inequality,   
\[
|K^{(\alpha,\beta)}(x,y)| \leq \tilde c_\alpha \, \tilde c_\beta \, \|A_D^{-1}\|_{2 \to 2}
\]
for all $x,y \in \Omega$, 
which gives immediately uniform bounds as in (\ref{etcom202.5;1}). 
Here the diameter of $\Omega$ is involved.
Next let $x,x',y, y' \in \Omega$ with $|x-x'| \leq 1$ and 
$|y - y'| \leq 1$.
Then 
\begin{eqnarray*}
|K^{(\alpha,\beta)}(x,y) - K^{(\alpha,\beta)}(x',y')|
& \leq & |(K^\alpha(x,\cdot), K_\beta(y,\cdot) - K_\beta(y',\cdot))_{L_2(\Omega)}|  \\
& & \hspace*{10mm} {}
   + |(K^\alpha(x,\cdot) - K^\alpha(x',\cdot), K_\beta(y',\cdot))_{L_2(\Omega)}|  \\
& \leq & c_\beta \, \tilde c_\alpha \, \|A_D^{-1}\|_{2 \to 2} \, |y - y'|^\kappa
   + c_\alpha \, \tilde c_\beta \, \|A_D^{-1}\|_{2 \to 2} \, |x - x'|^\kappa
.  
\end{eqnarray*}
Since $\Omega$ is bounded, there is a suitable $c_1 > 0$ such that 
\[
|K^{(\alpha,\beta)}(x,y) - K^{(\alpha,\beta)}(x',y')|
\leq c_1 \, \frac{(|x'-x| + |y'-y|)^\kappa}{|x-y|^{d-2+|\alpha| + |\beta| +\kappa}}  
\]
for all $x,x',y,y' \in \Omega$ with $x \neq y$, $x' \neq y'$ and
$|x-x'| + |y-y'| \leq \frac{1}{2} \, |x-y|$.
Also here the diameter of $\Omega$ is involved.
This proves the theorem.
\end{proof}

\section{The harmonic lifting on $L_p$} \label{Scom6}

The semigroup $(e^{-t A_D})_{t > 0}$ on $L_2(\Omega)$ 
extends consistently to a semigroup on $L_p(\Omega)$
for all $p \in [1,\infty]$, which is a $C_0$-semigroup if $p \in [1,\infty)$.
We denote by $-A_D^{(p)}$ the generator of this semigroup if confusion is possible.

\begin{prop} \label{pcom209}
Let $\kappa \in (0,1)$ and $K > 0$.
Let $\Omega \subset \Ri^d$ be an open set of class $C^{1+\kappa}$ with parameter $K$. 
For all $k,l \in \{ 1,\ldots,d \} $ let $c_{kl}, b_k, c_k \in C^\kappa(\Omega) \cap L_\infty(\Omega)$,
and let $c_0 \in L_\infty(\Omega)$.
Suppose there is a $\mu > 0$ such that 
$\RRe \sum_{k,l=1}^d c_{kl}(x) \, \xi_k \, \overline{\xi_l} \geq \mu \, |\xi|^2$
for all $\xi \in \Ci^d$ and $x \in \Omega$.
Further suppose that $0 \not\in \sigma(A_D)$.
Then one has the following.
\begin{tabel} 
\item \label{pcom209-1}
If $p \in [1,\infty)$, then $D(A_D^{(p)}) \subset W^{1,p}(\Omega)$ 
and the map $\partial^\alpha \, (A_D^{(p)})^{-1}$ is continuous from $L_p(\Omega)$ 
into $L_p(\Omega)$ for every multi-index $\alpha$ with $|\alpha| \leq 1$.
\item \label{pcom209-2}
If $\nu \in (0,\kappa]$, $p \in (\frac{d}{1-\nu},\infty]$
and $\alpha$ is a multi-index with $|\alpha| \leq 1$, then 
$\partial^\alpha \, A_D^{-1}$ maps $L_p(\Omega)$ continuously into $C^\nu(\Omega) \cap L_\infty(\Omega)$.
\item \label{pcom209-3}
If $p \in (1,\infty]$, $\Omega$ is bounded
and $\alpha$ is a multi-index with $|\alpha| \leq 1$, then the operator
$u \mapsto (\partial^\alpha \, A_D^{-1} u)|_\Gamma$ 
from $L_\infty(\Omega)$ into $C(\Gamma)$
extends consistently 
to a bounded operator from $L_p(\Omega)$ into $L_p(\Gamma)$.
\item \label{pcom209-4}
If $\nu \in (0,\kappa]$, $p \in (\frac{d}{1-\nu},\infty]$
and $\alpha$ is a multi-index with $|\alpha| \leq 1$, then 
$u \mapsto (\partial^\alpha \, A_D^{-1} u)|_\Gamma$ 
is continuous from $L_p(\Omega)$ into $C^\nu(\Gamma) \cap L_\infty(\Gamma)$.
\end{tabel}
\end{prop}
\begin{proof}
`\ref{pcom209-1}'.
There exists a $\lambda > 0$ such that 
\[
\frac{\mu}{2} \|u\|_{W^{1,2}(\Omega)}^2
\leq \RRe \gota(u,u) + \lambda \|u\|_{L_2(\Omega)}^2
\]
for all $u \in W^{1,2}(\Omega)$.
Then for all $p \in [1,\infty]$ the function 
$t \mapsto \partial^\alpha \, e^{-t (A_D + \lambda \, I)}$
from $(0,\infty)$ into $\cl(L_p(\Omega))$ is 
integrable by Theorem~\ref{tcom201}\ref{tcom201-6}, see also 
Statement~\ref{tcom201-4} of that theorem.
Hence $D(A_D^{(p)} + \lambda \, I) \subset W^{1,p}(\Omega)$ 
and the map $\partial^\alpha \, (A_D^{(p)} + \lambda \, I)^{-1}$ is continuous from $L_p(\Omega)$ 
into $L_p(\Omega)$ for all $k \in \{ 1,\ldots,d \} $.
Then also 
$\partial^\alpha \, (A_D^{(p)})^{-1} 
= \partial^\alpha \, (A_D^{(p)} + \lambda \, I)^{-1}  (A_D + \lambda \, I) \, A_D^{-1}$
is bounded.

The proof of Statement~\ref{pcom209-2} is similar using Theorem~\ref{tcom201}\ref{tcom201-5}
and for Statement~\ref{pcom209-3} use Proposition~\ref{pcom201.5}.
Statement~\ref{pcom209-4} follows from Statement~\ref{pcom209-2}.
\end{proof}

In the proof of the next proposition we use the uniform bounds that 
we derived in  Section~\ref{Scom4}.

\begin{prop} \label{pcom203}
Let $\kappa \in (0,1)$.
Let $\Omega \subset \Ri^d$ be an open bounded set of class~$C^{1+\kappa}$.
For all $k,l \in \{ 1,\ldots,d \} $ let $c_{kl}, b_k, c_k \in C^\kappa(\Omega)$,
and let $c_0 \in L_\infty(\Omega)$.
Suppose there is a $\mu > 0$ such that 
$\RRe \sum_{k,l=1}^d c_{kl}(x) \, \xi_k \, \overline{\xi_l} \geq \mu \, |\xi|^2$
for all $\xi \in \Ci^d$ and $x \in \Omega$.
Further suppose that $0 \not\in \sigma(A_D)$.
Let $\nu \in (0,\kappa]$ and $p \in (\frac{d}{1-\nu},\infty]$.
Let $u \in L_p(\Omega)$.
Then $A_D^{-1} u$ has a weak conormal derivative and 
\[
\partial_n^\gota \, A_D^{-1} u 
= \sum_{k,l=1}^d n_k \, ( c_{kl} \, \partial_l \, A_D^{-1} u)|_\Gamma
   + \sum_{k=1}^d n_k \, ( b_k \, A_D^{-1} u)|_\Gamma
.   \]
Moreover, $\partial_n^\gota \, A_D^{-1} u \in C^\nu(\Gamma)$
and the map $u \mapsto \partial_n^\gota \, A_D^{-1} u$ is 
continuous from $L_p(\Omega)$ into $C^\nu(\Gamma)$.
\end{prop}
\begin{proof}
Let $M > 0$ be such that 
$\|c_{kl}\|_{C^\kappa(\Omega)} \leq M$,
$\|b_k\|_{C^\kappa(\Omega)} \leq M$, $\|c_k\|_{C^\kappa(\Omega)} \leq M$
and $\|c_0\|_{L_\infty(\Omega)} \leq M$.
There exists a $\lambda > 0$, depending only on $\mu$ and $M$, such that 
\[
\frac{3\mu}{4} \, \|u\|_{W^{1,2}(\Omega)}^2
\leq \RRe \gota(u,u) + \lambda \|u\|_{L_2(\Omega)}^2
\]
for all $u \in W^{1,2}(\Omega)$.
One can regularise the coefficients.
First extend the $c_{kl}$ to functions $\tilde c_{kl} \in C^\kappa(\Ri^d)$
and similarly to the $b_k$ and $c_k$.
Then obtain 
coefficients $c^{(n)}_{kl}$, etc., and operators $A_D^{(n)}$.
Without loss of generality 
\[
\frac{\mu}{2} \, \|u\|_{W^{1,2}(\Omega)}^2
\leq \RRe \gota^{(n)}(u,u) + \lambda \|u\|_{L_2(\Omega)}^2
\]
for all $u \in W^{1,2}(\Omega)$ and $n \in \Ni$.
By Theorem~\ref{tcom201}\ref{tcom201-5} and \ref{tcom201-4}
for all $k \in \{ 1,\ldots,d \} $ the function 
$t \mapsto \partial_k \, e^{-t (A_D^{(n)} + \lambda \, I)}$
from $(0,\infty)$ into $\cl(L_p(\Omega),C^\nu(\Omega))$ is 
integrable with a uniform bound  in $n \in \Ni$.
Now  we prove the desired equality for the operator  $A_D^{(n)}$ with $C^\infty$ coefficients 
and then let $n \to \infty$. 
The arguments  are  exactly  the same  as in Step 2 of the proof 
of Proposition 5.3 in \cite{EO6}. 
We then   deduce that for all 
$u \in L_p(\Omega)$ the function $(A_D + \lambda \, I)^{-1} u$ has a weak conormal derivative and 
\[
\partial_n^\gota (A_D + \lambda \, I)^{-1} u 
= \sum_{k,l=1}^d n_k \, \Tr( c_{kl} \, \partial_l \, (A_D + \lambda \, I)^{-1} u)
   + \sum_{k=1}^d n_k \, \Tr( b_k \, (A_D + \lambda \, I)^{-1} u)
.   \]
Finally replace $u$ by $(A_D + \lambda \, I) \, A_D^{-1} u$ to obtain the proposition. 
\end{proof}

Every element of the domain $D(A_D)$ of the elliptic operator 
has a conormal derivative.
The next proposition is an extension of \cite{BinzE1} Proposition~4.2(a).

\begin{prop} \label{pcom222}
Let $\kappa \in (0,1)$.
Let $\Omega \subset \Ri^d$ be an open bounded set of class~$C^{1+\kappa}$.
For all $k,l \in \{ 1,\ldots,d \} $ let $c_{kl}, b_k, c_k \in C^\kappa(\Omega)$,
and let $c_0 \in L_\infty(\Omega)$.
Suppose there is a $\mu > 0$ such that 
$\RRe \sum_{k,l=1}^d c_{kl}(x) \, \xi_k \, \overline{\xi_l} \geq \mu \, |\xi|^2$
for all $\xi \in \Ci^d$ and $x \in \Omega$.
Further suppose that $0 \not\in \sigma(A_D)$.
Then $D(A_D) \subset D(\partial_n^\gota)$.
\end{prop}
\begin{proof}
Define $R \colon L_\infty(\Omega) \to C(\Gamma)$ by 
\[
R u 
= \sum_{k,l=1}^d n_k \, ( c_{kl} \, \partial_l \, A_D^{-1} u)|_\Gamma
   + \sum_{k=1}^d n_k \, ( b_k \, A_D^{-1} u)|_\Gamma
.   \]
By Proposition~\ref{pcom209}\ref{pcom209-3} the operator $R$ extends 
to a bounded operator $R_2$ from $L_2(\Omega)$ into $L_2(\Gamma)$.
Let $v \in W^{1,2}(\Omega)$.
Next let $u \in L_\infty(\Omega)$.
It follows from Proposition~\ref{pcom203} that $A_D^{-1} u \in D(\partial_n^\gota)$
and $R u = \partial_n^\gota \, A_D^{-1} u$.
Hence 
\[
(R u, \Tr v)_{L_2(\Gamma)}
= (\partial_n^\gota \, A_D^{-1} u, \Tr v)_{L_2(\Gamma)}
= \gota(A_D^{-1} u, v) - (u, v)_{L_2(\Omega)}
.  \]
Since $A_D^{-1}$ is continuous from $L_2(\Omega)$ into $W^{1,2}(\Omega)$ 
by the closed graph theorem, and $L_\infty(\Omega)$ is dense in $L_2(\Omega)$, 
one deduces that 
\[
(R_2 u, \Tr v)_{L_2(\Gamma)}
= \gota(A_D^{-1} u, v) - (u, v)_{L_2(\Omega)}
\]
for all $u \in L_2(\Omega)$.
Hence $D(A_D) \subset D(\partial_n^\gota)$.
\end{proof}

There is a remarkable relation between the harmonic lifting, 
conormal derivative $\partial_n^{\gota^*}$ with respect of the 
dual form $\gota^*$ and the resolvent of the Dirichlet operator.

\begin{lemma} \label{lcom206}
Let $\kappa \in (0,1)$.
Let $\Omega \subset \Ri^d$ be an open bounded set of class~$C^{1+\kappa}$.
For all $k,l \in \{ 1,\ldots,d \} $ let $c_{kl}, b_k, c_k \in C^\kappa(\Omega)$,
and let $c_0 \in L_\infty(\Omega)$.
Suppose there is a $\mu > 0$ such that 
$\RRe \sum_{k,l=1}^d c_{kl}(x) \, \xi_k \, \overline{\xi_l} \geq \mu \, |\xi|^2$
for all $\xi \in \Ci^d$ and $x \in \Omega$.
Further suppose that $0 \not\in \sigma(A_D)$.
Let $v \in L_2(\Omega)$ and $\varphi \in H^{1/2}(\Gamma)$.
Then 
\[
(\gamma(\varphi), v)_{L_2(\Omega)}
= - (\varphi, \partial_n^{\gota^*} \, (A_D^*)^{-1} v)_{L_2(\Gamma)}
.  \]
\end{lemma}
\begin{proof}
Since  $(A_D^*)^{-1} v \in W_0^{1,2}(\Omega)$, we have by the definition of $\gamma(\varphi)$ 
\begin{eqnarray*}
0 
&=& \overline{\gota(\gamma(\varphi), (A_D^*)^{-1} v)}   \\
&=& \gota^*((A_D^*)^{-1} v, \gamma(\varphi))  \\
&=& \int_\Omega (A_D^* \, (A_D^*)^{-1} v) \, \overline{\gamma(\varphi)} 
   + \int_\Gamma \Big( \partial_n^{\gota^*} \, (A_D^*)^{-1} v \Big) \, \overline \varphi 
,
\end{eqnarray*}
where we used the fact that the conormal derivative 
$\partial_n^{\gota^*} \, (A_D^*)^{-1} v$ exists by Proposition~\ref{pcom222}.
\end{proof}

If one considers $\gamma$ as a densely defined operator from $L_2(\Gamma)$
into $L_2(\Omega)$, then we just proved that 
$\gamma^* = - \partial_n^{\gota^*} \, (A_D^*)^{-1}$.
Note that Theorem~\ref{tcom102}\ref{tcom102-2} is a consequence of 
Statement~\ref{tcom204-4} in the next theorem.

\begin{thm} \label{tcom204}
Let $\kappa \in (0,1)$.
Let $\Omega \subset \Ri^d$ be an open bounded set with a $C^{1+\kappa}$-boundary.
For all $k,l \in \{ 1,\ldots,d \} $ let $c_{kl}, b_k, c_k \in C^\kappa(\Omega)$,
and let $c_0 \in L_\infty(\Omega)$.
Suppose there is a $\mu > 0$ such that 
$\RRe \sum_{k,l=1}^d c_{kl}(x) \, \xi_k \, \overline{\xi_l} \geq \mu \, |\xi|^2$
for all $\xi \in \Ci^d$ and $x \in \Omega$.
Further suppose that $0 \not\in \sigma(A_D)$.
Let $G \colon \{ (x,y) \in \Omega \times \Omega : x \neq y \} \to \Ci$
be as in Theorem~\ref{tcom202}.
Then $G$ is differentiable on $ \{ (x,y) \in \Omega \times \Omega : x \neq y \} $ 
by Theorem~\ref{tcom202} and the derivative extends to a continuous
function on $ \{ (x,y) \in \overline \Omega \times \overline \Omega : x \neq y \} $.
Define the function $K_\gamma \colon \Omega \times \Gamma \to \Ci$ by 
\[
K_\gamma(x,z) 
= - \sum_{k,l=1}^d n_k(z) \, c_{lk}(z) \, (\partial^{(2)}_l G)(x,z)
   - \sum_{k=1}^d n_k(z) \, c_k(z) \, G(x,z)
.  \]
Then one has the following.
\begin{tabel}
\item \label{tcom204-1}
The map $K_\gamma$ is continuous.
\end{tabel}
Define $T \colon L_1(\Gamma) \to C(\Omega)$ by 
\[
(T \varphi)(x) 
= \int_\Gamma K_\gamma(x,z) \, \varphi(z) \, dz
.  \]
\begin{tabel}
\setcounter{teller}{1}
\item \label{tcom204-2}
If $\varphi \in H^{1/2}(\Gamma)$, then 
$\gamma (\varphi) = T \varphi$ a.e.
In particular, $T$ is an extension of the harmonic map~$\gamma$.
\item \label{tcom204-3}
There exists a $\tilde c > 0$ such that 
\[
|K_\gamma(x,z)| \leq \frac{\tilde c}{|x-z|^{d-1}}
\quad \mbox{and} \quad
|K_\gamma(x',z') - K_\gamma(x,z)| 
\leq \tilde c \, \frac{(|x'-x| + |z'-z|)^\kappa}{|x-z|^{d-1+\kappa}}
\]
for all $x,x' \in \Omega$ and $z,z' \in \Gamma$ with $|x'-x| + |z'-z| \leq \frac{1}{2} \, |x-z|$.
\item \label{tcom204-4}
If $p \in [1,\infty)$ and $\varphi \in L_p(\Gamma)$, then $T \varphi \in L_p(\Omega)$
and the map $T \colon L_p(\Gamma) \to L_p(\Omega)$ is bounded.
In particular $\gamma$ extends to a bounded map from $L_2(\Gamma)$ into $L_2(\Omega)$.
\item \label{tcom204-5}
Let $T_2 = T|_{L_2(\Gamma)} \colon L_2(\Gamma) \to L_2(\Omega)$.
Then for all $p \in (\frac{d}{1-\kappa}, \infty]$ there exists a $c > 0$ such that 
$T_2^* u \in C^\kappa(\Gamma)$ and 
$\|T_2^* u\|_{C^\kappa(\Gamma)} \leq c \, \|u\|_{L_p(\Omega)}$
for all $u \in L_p(\Omega)$.
\end{tabel}
\end{thm}
\begin{proof}
Statements~\ref{tcom204-1}--\ref{tcom204-3} and also the $p=1$ version 
of Statement~\ref{tcom204-4} follow as in the proof of Proposition~5.5 in \cite{EO6}, 
where we now use Lemma~\ref{lcom206} and Proposition~\ref{pcom203}.

`\ref{tcom204-4}'.
Let $p \in (1,\infty)$ and let $q \in (1,\infty)$ be the dual exponent.
Define $R \colon L_\infty(\Omega) \to C(\Gamma)$ by 
\[
R u 
= \sum_{k,l=1}^d n_k \, ( \overline{c_{lk}} \, \partial_l \, (A_D^*)^{-1} u)|_\Gamma
   + \sum_{k=1}^d n_k \, ( \overline{c_k} \, (A_D^*)^{-1} u)|_\Gamma
.  \]
Then by Proposition~\ref{pcom209}\ref{pcom209-3} the map $R$ extends to a 
bounded operator $R_q$ from $L_q(\Omega)$ into $L_q(\Gamma)$.
Let $\varphi \in H^{1/2}(\Gamma)$ and $u \in L_\infty(\Omega)$.
Then Statement~\ref{tcom204-2}, Lemma~\ref{lcom206} and Proposition~\ref{pcom203} give
\[
(T \varphi, u)_{L_2(\Omega)}
= (\gamma (\varphi), u)_{L_2(\Omega)}
= - (\varphi, \partial_n^{\gota^*} \, (A_D^*)^{-1} u)_{L_2(\Gamma)}
= - (\varphi, R u)_{L_2(\Gamma)}
.  \]
Hence 
\begin{equation}
\int_\Omega (T \varphi) \, \overline u
= - \int \varphi \, \overline{R u}
\label{etcom204;1}
\end{equation}
first for all $\varphi \in H^{1/2}(\Gamma)$ and then by continuity and 
density for all $\varphi \in L_1(\Gamma)$.\\
Now let $\varphi \in L_p(\Gamma)$.
Then 
\[
\Big| \int_\Omega (T \varphi) \, \overline u \Big|
= |\int \varphi \, \overline{R u}| 
\leq \|\varphi\|_{L_p(\Gamma)} \, \|R u\|_{L_q(\Gamma)}
\leq \|\varphi\|_{L_p(\Gamma)} \, \|R_q\|_{L_q(\Omega) \to L_q(\Gamma)} \, \|u\|_{L_q(\Omega)}
\]
for all $u \in L_\infty(\Omega)$.
Hence $T \varphi \in L_p(\Omega)$ and 
$\|T \varphi\|_{L_p(\Omega)} \leq \|R_q\|_{L_q(\Omega) \to L_q(\Gamma)} \, \|\varphi\|_{L_p(\Gamma)}$.

`\ref{tcom204-5}'.
It follows from (\ref{etcom204;1}) that 
$(T_2 \varphi, u)_{L_2(\Omega)} = - (\varphi, R_2 u)_{L_2(\Gamma)}$
first for all $\varphi \in L_2(\Gamma)$ and $u \in L_\infty(\Omega)$, 
and then by density for all $\varphi \in L_2(\Gamma)$ and $u \in L_2(\Omega)$.
Hence $T_2^* = - R_2$. 
If $u \in L_p(\Omega)$, then $T_2^* u = - R_2 u = R_p u$. 
Now the statement follows from Proposition~\ref{pcom209}\ref{pcom209-4}. 
This completes the proof of the proposition.
\end{proof}

If no confusion is possible, then we will identify 
$\gamma$ with $T|_{L_2(\Gamma)} \colon L_2(\Gamma) \to L_2(\Omega)$.

\section{The harmonic lifting on $C^\nu$}\label{Scom7}

Throughout this section let $\kappa \in (0,1)$ and $\Omega \subset \Ri^d$
a bounded open set of class~$C^{1+\kappa}$.
For all $k,l \in \{ 1,\ldots,d \} $ let $c_{kl}, b_k, c_k \in C^\kappa(\Omega)$ 
and $c_0 \in L_\infty(\Omega)$.
Let the form $\gota$ and operator $A_D$ be as in the introduction.
We assume that $0 \not\in \sigma(A_D)$.
Let $\cn$ be the associated Dirichlet-to-Neumann operator.
In this section we prove H\"older continuity estimates for the 
harmonic lifting as stated in Theorem~\ref{tcom102}\ref{tcom102-1}.

Let $M > 0$ be such that 
$\|c_{kl}\|_{C^\kappa(\Omega)} \leq M$,
$\|b_k\|_{C^\kappa(\Omega)} \leq M$, $\|c_k\|_{C^\kappa(\Omega)} \leq M$
and $\|c_0\|_{L_\infty(\Omega)} \leq M$.
There exists a $\lambda > 0$, depending only on $\mu$ and $M$, such that 
\begin{equation}
\frac{\mu}{2} \|u\|_{W^{1,2}(\Omega)}^2
\leq \RRe \gota(u,u) + \lambda \|u\|_{L_2(\Omega)}^2
\label{elcom308;10}
\end{equation}
for all $u \in W^{1,2}(\Omega)$.
Let $\gamma_\lambda$ and $\cn_\lambda$ be the harmonic lifting and 
Dirichlet-to-Neumann operator with respect to the operator with 
coefficient $c_0 + \lambda \, \one_\Omega$ instead of $c_0$.
These correspond to the form~$\gota_\lambda$.

First we need a family of suitable cut-off functions.

\begin{lemma} \label{lcom301}
There exists a $\hat c \geq 1$ and for all $r \in (0,\infty)$ there exists a 
function $\chi_r \in C_c^\infty(\Ri^d)$ such that 
$0 \leq \chi_r \leq 1$, 
$\chi_r(x) = 1$ for all $x \in B(0,2r)$,
$\supp \chi_r \subset B(0,3r)$ and 
$\|\partial^\alpha \chi_r\|_\infty \leq \frac{\hat c}{r^{|\alpha|}}$
for every multi-index with $1 \leq |\alpha| \leq 2$.
\end{lemma}
\begin{proof}
Fix $\chi \in C_c^\infty(\Ri^d)$ such that $0 \leq \chi \leq 1$, 
$\chi(x) = 1$ for all $x \in B(0,2)$ and 
$\supp \chi \subset B(0,3)$.
Let $r \in (0,\infty)$.
Define $\chi_r \in C_c^\infty(\Ri^d)$ by $\chi_r(x) = \chi(r^{-1} \, x)$.
\end{proof}

The proof of Theorem~\ref{tcom102}\ref{tcom102-1} heavily depends on the 
following estimate.

\begin{lemma} \label{lcom307}
Let $\hat c \geq 1$ and the $\chi_r$ be as in Lemma~\ref{lcom301}.
Then there exists a $c' > 0$ such that the following is valid.

Let $r \in (0,\diam(\Omega)]$ and $y \in \Omega$.
Define $\Psi \colon \overline \Omega \to \Ri$ by 
$\Psi(x) = \chi_r(x-y)$.
Then $|(\gamma_\lambda(\Psi|_\Gamma))(y)| \leq c'$.
\end{lemma}

At the moment, for the proof of Lemma~\ref{lcom307},
we need only the first three parts of the next technical lemma.
In order to avoid too much repetition in the proof of the commutator 
estimates, we already include some results for the Dirichlet-to-Neumann operator.
Recall the convention that we identify $C^\nu(\Omega)$ with $C^\nu(\overline \Omega)$.

\begin{lemma} \label{lcom308}
There exists an $\widetilde M > 0$ such that the following is valid.

Let $r \in (0,\diam \Omega]$, $\Psi \in C^{1+\kappa}(\Omega)$,
$y \in \overline \Omega$ and $c \geq 0$.
Suppose that $\supp \Psi \subset B(y,3r)$,
\[
\|\Psi\|_{L_\infty(\Omega)} \leq c \, r, 
 \quad
\|\partial_k \Psi\|_{L_\infty(\Omega)} \leq c
\quad \mbox{and} \quad
|||\partial_k \Psi|||_{C^\kappa(\Omega)} \leq c \, r^{-\kappa}
\]
for all $k \in \{ 1,\ldots,d \} $.
Then $\gamma_\lambda(\Psi|_\Gamma) \in C^{1+\kappa}(\Omega)$,
$\Psi|_\Gamma \in D(\cn_\lambda)$ and $\cn_\lambda(\Psi|_\Gamma) \in C^\kappa(\Gamma)$.
Moreover,
\begin{tabel}
\item \label{lcom308-1}
$\|\nabla \gamma_\lambda(\Psi|_\Gamma)\|_{L_\infty(\Omega)} \leq \widetilde M \, c$, 
\item \label{lcom308-2}
$|||\nabla \gamma_\lambda(\Psi|_\Gamma)|||_{C^\kappa(\Omega)} \leq \widetilde M \, c \, r^{-\kappa}$, 
\item \label{lcom308-3}
$|\Psi(x) - (\gamma_\lambda(\Psi|_\Gamma))(x)| \leq \widetilde M \, c \, d(x,\Gamma)$
for all $x \in \Omega$,
\item \label{lcom308-4}
$\|\cn_\lambda(\Psi|_\Gamma)\|_{C(\Gamma)} \leq \widetilde M \, c$, and 
\item \label{lcom308-5}
$|||\cn_\lambda(\Psi|_\Gamma)|||_{C^\kappa(\Gamma)} \leq \widetilde M \, c \, r^{-\kappa}$.
\end{tabel}
\end{lemma}
\begin{proof}
There is a $\hat c > 0$ such that $|\Omega(x,\rho)| \geq \hat c \, \rho^d$
for all $x \in \overline \Omega$ and $\rho \in (0,\diam(\Omega)]$.

Define $\tau \in W^{1,2}_0(\Omega)$ by 
$\tau = \Psi - \gamma_\lambda(\Psi|_\Gamma)$.
Recall that $\lambda > 0$ is introduced in (\ref{elcom308;10}).
Then 
\[
\gota_\lambda(\tau, v)
= \gota_\lambda(\Psi, v) 
= \sum_{k=1}^d (f_k, \partial_k v)_{L_2(\Omega)}
   + (f_0,v)_{L_2(\Omega)}
\]
for all $v \in W^{1,2}_0(\Omega)$, where 
\[
f_k = b_k \, \Psi + \sum_{l=1}^d c_{kl} \, \partial_l \Psi
\quad \mbox{and} \quad
f_0 = (c_0 + \lambda) \, \Psi + \sum_{l=1}^d c_l \, \partial_l \Psi
\]
for all $k \in \{ 1,\ldots,d \} $.
Note that $\supp f_k \subset \Omega(y,3r)$
for all $k \in \{ 0,1,\ldots,d \} $.
Then there is an $M_1 > 0$, depending only on $M$, $\mu$ and $\Omega$, but
independent of $\Psi$, $y$, $r$ and $c$, such that 
\[
\|f_0\|_{L_\infty(\Omega)} \leq M_1 \, c
 , \quad
\|f_k\|_{L_\infty(\Omega)} \leq M_1 \, c
\quad \mbox{and} \quad
|||f_k|||_{C^\kappa(\Omega)} \leq M_1 \, c \, r^{-\kappa}
\]
for all $k \in \{ 1,\ldots,d \} $.
Hence 
$\|f_0\|_{M, d+2\kappa - 2} 
\leq \|f_0\|_{M,d} 
\leq \sqrt{\omega_d} \, \|f_0\|_{L_\infty(\Omega)}
\leq \sqrt{\omega_d} \, M_1 \, c$
and $\|f_k\|_{C^\kappa(\Omega)} \leq 2 M_1 \, c \, r^{-\kappa}$
for all $k \in \{ 1,\ldots,d \} $.
Let $\tilde c > 0$ be as in Proposition~\ref{pcom214}.
Then 
\begin{equation}
\sum_{k=1}^d \|\partial_k \tau\|_{C^\kappa(\Omega)}
\leq c \, \tilde c \, M_1 \, (1 + 2d) \, r^{-\kappa}
.  
\label{elcom308;20}
\end{equation}
On the other hand, 
\[
\|f_k\|_{L_2(\Omega)}
\leq \|f_k\|_{L_\infty(\Omega)} \, |\Omega(y,3r)|^{1/2}
\leq \sqrt{\omega_d} \, 3^{d/2} \, M_1 \, c \, r^{d/2}
\]
for all $k \in \{ 0,1,\ldots,d \} $.
Then coercivity (\ref{elcom308;10}) implies that 
\[
\|\nabla \tau\|_{L_2(\Omega)}
\leq \frac{2}{\mu} \, \sum_{k=0}^d \|f_k\|_{L_2(\Omega)}
\leq \frac{2}{\mu} \, (d+1) \, \sqrt{\omega_d} \, 3^{d/2} \, M_1 \, c \, r^{d/2}
.  \]
Let $x \in \Omega$ and $k \in \{ 1,\ldots,d \} $.
If $\tilde x \in \Omega(x,r \wedge 1)$, then 
$|(\partial_k \tau)(x)| 
\leq |(\partial_k \tau)(\tilde x)| + |||\partial_k \tau|||_{C^\kappa(\Omega)} \, (r \wedge 1)^\kappa$.
Integrating over $\tilde x \in \Omega(x,r \wedge 1)$ and dividing by $|\Omega(x,r \wedge 1)|$ gives
\begin{eqnarray*}
|(\partial_k \tau)(x)|
& \leq & |||\partial_k \tau|||_{C^\kappa(\Omega)} \, (r \wedge 1)^\kappa
   + \langle |\partial_k \tau| \rangle_{\Omega(x,r \wedge 1)}  \\
& \leq & |||\partial_k \tau|||_{C^\kappa(\Omega)} \, (r \wedge 1)^\kappa
   + |\Omega(x,r \wedge 1)|^{-1/2} \, \|\partial_k \tau\|_{L_2(\Omega)}
\leq M_2 \, c
,  
\end{eqnarray*}
where $M_2 = \tilde c \, M_1 \, (1 + 2d) 
   + \hat c^{-1/2} \, \frac{2}{\mu} \, (d+1) \, \sqrt{\omega_d} \, 
             3^{d/2} \, M_1 (1 \vee (\diam \Omega)^{d/2})$.
This proves Statements~\ref{lcom308-1} and \ref{lcom308-2}, 
since $\gamma_\lambda(\Psi)|_\Gamma) = \Psi - \tau$.

Before we turn to the proof of Statement~\ref{lcom308-3}, we derive a bound
for $\|\tau\|_{L_\infty(\Omega)}$ which is not optimal in $r$, but suffices for 
the proof of Statement~\ref{lcom308-3}.
Coercivity (\ref{elcom308;10}) gives as before
\[
\|\tau\|_{L_2(\Omega)}
\leq \frac{2}{\mu} \, (d+1) \, \sqrt{\omega_d} \, 3^{d/2} \, M_1 \, c \, r^{d/2}
.  \]
Using Proposition~\ref{pcom213} there exists an $M_2 > 0$, independent of 
$\Psi$, $y$, $r$ and $c$, such that 
\[
\|\nabla \tau\|_{M, d-1}
\leq M_2 \sum_{k=0}^d \|f_k\|_{L_\infty(\Omega)}
\leq (d+1) \, M_1 \, M_2 \, c
.  \]
Then Lemma~\ref{lcom218} gives
a bound for $|||\tau|||_{\cm, d+1}$.
Next use 
\cite{ERe2} Lemma~3.1(c) to bound the $C^{1/2}(\Omega)$-seminorm 
by the Campanato norm.
Hence there is an $M_3 > 0$, independent of 
$\Psi$, $r$ and $c$, such that $|||\tau|||_{C^{1/2}(\Omega)} \leq M_3 \, c$.
Let $x \in \Omega$.
Then 
\[
|\tau(x)|
\leq |||\tau|||_{C^{1/2}(\Omega)} \, (r \wedge 1)^{1/2} + \langle |\tau| \rangle_{\Omega(x,r \wedge 1)}
\leq M_4 \, c
,
\]
where 
$M_4 = M_3 + \frac{2}{\mu} \, (d+1) \, \sqrt{\omega_d} \, 3^{d/2} \, M_1 \, (1 \vee (diam \Omega)^{d/2})$.

Now we prove Statement~\ref{lcom308-3}.
Let $x \in \Omega$. 
Then there is a $z \in \Gamma$
such that $d(x,\Gamma) = |x-z|$.
Obviously the straight line segment from 
$x$ to $z$ is in $\Omega$, except of course for the point $z$.
Since $\tau(z) = 0$, the estimate in Statement~\ref{lcom308-3}
follows from the mean value theorem and the bound of Statement~\ref{lcom308-1}.

Finally we turn to the Dirichlet-to-Neumann operator.
First assume that $c_{kl}, b_k, c_k,c_0 \in C^\infty_b(\Omega)$ 
for all $k,l \in \{ 1,\ldots,d \} $
and there is a $\Phi \in C^\infty(\Ri^d)$ such that 
$\Psi = \Phi|_{\overline \Omega}$.
Then $\tau \in C^\infty(\Omega)$ by interior regularity.
It follows as in the proof of Lemma~5.2 in \cite{EO6} 
that $\Psi - \tau$ has a weak conormal derivative
and 
\begin{eqnarray*}
\partial_n^{\gota_\lambda}(\Psi - \tau)
& = & \sum_{k,l=1}^d n_k \, ( c_{kl} \, \partial_l (\Psi - \tau))|_\Gamma
   + \sum_{k=1}^d n_k \, ( b_k \, (\Psi - \tau))|_\Gamma  \\
& = & \sum_{k,l=1}^d n_k \, ( c_{kl} \, \partial_l (\Psi - \tau))|_\Gamma
   + \sum_{k=1}^d n_k \, ( b_k \, \Psi)|_\Gamma
,  
\end{eqnarray*}
where we used in the last step that $\tau|_\Gamma = 0$ pointwise
because $\tau \in C(\overline \Omega) \cap H^1_0(\Omega)$ 
and $\Omega$ is Lipschitz.

So $\gamma_\lambda(\Psi|_\Gamma) = \Psi - \tau \in D(\partial_n^{\gota_\lambda})$
and $\Psi|_\Gamma \in D(\cn_\lambda)$.
Moreover, 
\begin{eqnarray*}
\cn_\lambda(\Psi|_\Gamma)
& = & \partial_n^{\gota_\lambda} \, \gamma_\lambda(\Psi)|_\Gamma)
= \partial_n^{\gota_\lambda} \, (\Psi - \tau)  
= \sum_{k,l=1}^d n_k \, \Big( c_{kl} \, \partial_l(\Psi - \tau) \Big)|_\Gamma
   + \sum_{k=1}^d n_k \, ( b_k \, \Psi)|_\Gamma
.
\end{eqnarray*}
Hence 
\begin{equation}
\gota_\lambda(\gamma_\lambda(\Psi|_\Gamma), v)
= \int_\Gamma \sum_{k,l=1}^d n_k \, \Big( c_{kl} \, \partial_l(\Psi - \tau) \Big)|_\Gamma \, \overline{\Tr v}
   + \sum_{k=1}^d n_k \, ( b_k \, \Psi)|_\Gamma \, \overline{\Tr v}
\label{elcom308;11}
\end{equation}
for all $v \in W^{1,2}(\Omega)$.

Next we consider H\"older continuous coefficients. 
We proceed again by approximation. 
First extend the $c_{kl}$ to functions $\tilde c_{kl} \in C^\kappa(\Ri^d)$
and similarly to the $b_k$ and $c_k$.
Then by regularisation and restriction one  obtains 
coefficients $c^{(n)}_{kl}$, etc.\ on $\Omega$, 
which satisfy the ellipticity on $\Omega$ for large $n$.
For large $n$ one obtains operator $A_D^{(n)}$,
form $\gota^{(n)}$, harmonic lifting $\gamma_\lambda^{(n)}$
and Dirichlet-to-Neumann operator $\cn_\lambda^{(n)}$ for all $n \in \Ni$.
Note that (\ref{elcom308;10}) remains valid for $\gota^{(n)}$ instead of 
$\gota$ with the same $\lambda$ if $n \in \Ni$ is large enough and 
without loss of generality for all $n \in \Ni$.
We still assume that there is a $\Phi \in C^\infty(\Ri^d)$ such that 
$\Psi = \Phi|_\Omega$.
For all $n \in \Ni$ define $\tau_n = \Psi - \gamma_\lambda^{(n)}(\Psi|_\Gamma)$.
Then $\sup_{n \in \Ni} \|\partial_k \tau_n\|_{C^\kappa(\Omega)} < \infty$
by (\ref{elcom308;20}).
By the Arzel\`a--Ascoli theorem there exists a $\tilde \tau \in C^{1+\kappa}(\Omega)$
such that, after passing to a subsequence if necessary, 
$\lim \tau_n = \tilde \tau$ in $C^{1+\kappa}(\Omega)$.
Then $\tilde \tau \in W^{1,2}_0(\Omega)$.
If $v \in C_c^\infty(\Omega)$, then 
\[
\gota_\lambda^{(n)}(\tau_n,v)
= \gota_\lambda^{(n)}(\Psi,v)
\]
for all $n \in \Ni$.
Taking the limit $n \to \infty$ gives
\begin{equation}
\gota_\lambda(\tilde \tau,v)
= \gota_\lambda(\Psi,v)
.
\label{elcom308;12}
\end{equation}
Then by density,  (\ref{elcom308;12}) is valid for all $v \in W^{1,2}_0(\Omega)$.
Hence $\tilde \tau = \tau$.
Moreover, 
\[
\gamma_\lambda(\Psi|_\Gamma)
= \Psi - \tau
= \lim (\Psi - \tau_n)
= \lim \gamma_\lambda^{(n)}(\Psi|_\Gamma)
\]
in $W^{1,2}(\Omega)$.
Now $\lim c^{(n)}_{kl} = c_{kl}$ weakly$^*$ in $L_\infty(\Omega)$
and similarly for the other coefficients (including $c_0$).
Let $v \in W^{1,2}(\Omega)$.
Then (\ref{elcom308;11}) is applicable for $\gota_\lambda^{(n)}$ and gives
\[
\gota_\lambda^{(n)}(\gamma_\lambda^{(n)}(\Psi|_\Gamma), v)
= \int_\Gamma \sum_{k,l=1}^d n_k \, \Big( c_{kl}^{(n)} \, \partial_l(\Psi - \tau_n) \Big)|_\Gamma \, \overline{\Tr v}
   + \sum_{k=1}^d n_k \, ( b_k^{(n)} \, \Psi)|_\Gamma \, \overline{\Tr v}
\]
for all $n \in \Ni$.
Taking the limits $n \to \infty$ implies that (\ref{elcom308;11}) is valid
for $\gota_\lambda$.
We deduce that (\ref{elcom308;11}) is valid whenever there is a 
$\Phi \in C^\infty(\Ri^d)$ such that 
$\Psi = \Phi|_\Omega$, but without any additional smoothness 
assumptions on the coefficients.

We now also drop the additional smoothness assumption on $\Psi$.
Extend $\Psi$ to an element $\widetilde \Psi \in C^{1+\kappa}(\Ri^d)$.
Regularise to obtain $\widetilde \Psi_n \in C^\infty(\Ri^d)$ 
for all $n \in \Ni$ with the same bounds as for $\Psi$
and $\lim \widetilde \Psi_n = \widetilde \Psi$ in $C^{1+\kappa}(\Ri^d)$.
Let $\Psi_n = \widetilde \Psi_n|_\Omega$ for all $n \in \Ni$.
For all $n \in \Ni$ define 
$\tau_n = \Psi_n - \gamma_\lambda(\Psi_n)|_\Gamma)$.
Then $\sup_{n \in \Ni} \|\partial_k \tau_n\|_{C^\kappa(\Omega)} < \infty$
by (\ref{elcom308;20}) again.
One can repeat almost the same argument and  
deduce that (\ref{elcom308;11}) is valid without any 
additional smoothness assumption.

Hence in the general case one obtains that
$\Psi|_\Gamma \in D(\cn_\lambda)$ and 
\[
\cn_\lambda(\Psi|_\Gamma)
= \sum_{k,l=1}^d n_k \, \Big( c_{kl} \, \partial_l(\Psi - \tau) \Big)|_\Gamma
   + \sum_{k=1}^d n_k \, ( b_k \, \Psi)|_\Gamma
.  \]
Therefore $\cn_\lambda(\Psi|_\Gamma) \in C^\kappa(\Gamma)$ 
and the bounds of Statements~\ref{lcom308-4} and~\ref{lcom308-5} 
follow by the estimates at the beginning of the proof.
\end{proof}

\begin{proof}[{\bf Proof of Lemma~\ref{lcom307}.}]
Let $\widetilde M > 0$ be as in Lemma~\ref{lcom308}.
It follows from the bounds of Lemma~\ref{lcom301}
that there is a $c > 0$, independent of $r$ and $y$, such that 
\[
\|r \, \Psi\|_{L_\infty(\Omega)} \leq c \, r
 , \quad
\|\partial_k (r \, \Psi)\|_{L_\infty(\Omega)} \leq c
\quad \mbox{and} \quad
|||\partial_k (r \, \Psi)|||_{C^\kappa(\Omega)} \leq c \, r^{-\kappa}
\]
for all $k \in \{ 1,\ldots,d \} $.
Write $\tau = \Psi - \gamma_\lambda(\Psi|_\Gamma)$.
We distinguish two cases.

\noindent
{\bf Case 1.} Suppose $3r < d(y,\Gamma)$.
Then $\Psi|_\Gamma = 0$ and hence $\gamma_\lambda(\Psi|_\Gamma) = 0$
and trivially $|(\gamma_\lambda(\Psi|_\Gamma))(y)| = 0$.

\noindent
{\bf Case 2.} Suppose $3r \geq d(y,\Gamma)$.
It follows from Lemma~\ref{lcom308}\ref{lcom308-3} that 
\[
|r \, \tau(y)| 
\leq \widetilde M \, c \, d(y,\Gamma)
\leq 3 \widetilde M \, c \, r
.  \]
So $|\tau(y)| \leq 3 \widetilde M \, c$ and 
$|(\gamma_\lambda(\Psi|_\Gamma))(y)| \leq 1 + 3 \widetilde M \, c$.
\end{proof}

We will need some  integral estimates. 
They are rather elementary and we state  them in the following lemma.

\begin{lemma} \label{lcom304}
\mbox{}
\begin{tabel}
\item \label{lcom304-1}
There exists a $c_1 > 0$ such that 
\[
\int_{ \{ w \in \Gamma : |w-z| \geq r \} }
   \frac{1}{|w-z|^{d-1+\beta}} \, dw
\leq \frac{c_1}{1 - 2^{-\beta}} \, \frac{1}{r^\beta}
\]
for all $\beta \in (0,\infty)$, $r \in (0,\infty)$ and $z \in \Gamma$.
\item \label{lcom304-2}
There exists a $c_2 > 0$ such that 
\[
\int_{ \{ w \in \Gamma : |w-z| \leq r \} }
   \frac{1}{|w-z|^{d-1-\beta}} \, dw
\leq \frac{c_2}{1 - 2^{-\beta}} \, r^\beta
\]
for all $\beta \in (0,d-1]$, $r \in (0,\infty)$ and $z \in \Gamma$.
\item \label{lcom304-3}
There exists a $c_3 > 0$ such that 
\[
\int_{ \{ w \in \Gamma : |w-x| \geq r \} }
   \frac{1}{|w-x|^{d-1+\beta}} \, dw
\leq \frac{4^\beta \, c_3}{1 - 2^{-\beta}} \, \frac{1}{r^\beta}
\]
for all $\beta \in (0,\infty)$, $r \in (0,\infty)$ and $x \in \Omega$.
\item \label{lcom304-4}
There exists a $c_4 > 0$ such that 
\[
\int_{ \{ w \in \Gamma : |w-x| \leq r \} }
   \frac{1}{|w-x|^{d-1-\beta}} \, dw
\leq \frac{c_4}{1 - 2^{-\beta}} \, r^\beta
\]
for all $\beta \in (0,d-1]$, $r \in (0,\infty)$ and $x \in \Omega$.
\end{tabel}
\end{lemma}
\begin{proof}
There exists a $c > 0$ such that the surface measure satisfies 
$\sigma(\Gamma \cap B(z,\rho)) \leq c \, \rho^{d-1}$ for all 
$z \in \Gamma$ and $\rho \in (0,\infty)$.

`\ref{lcom304-1}'.
Let $\beta \in (0,\infty)$, $r \in (0,\infty)$ and $z \in \Gamma$.
For all $n \in \Ni$ define the annulus
$I_n = \{ w \in \Gamma : 2^{n-1} \, r \leq |w-z| \leq 2^n \, r \} $.
Then 
\[
\int_{ \{ w \in \Gamma : |w-z| \geq r \} }
   \frac{1}{|w-z|^{d-1+\beta}} \, dw
\leq \sum_{n=1}^\infty \frac{1}{(2^{n-1} \, r)^{d-1+\beta}}
   \cdot c \, (2^n \, r)^{d-1}
= \frac{2^{d-1} c}{1 - 2^{-\beta}} \, \frac{1}{r^\beta} 
\]
as required.

`\ref{lcom304-2}'.
Let $\beta \in (0,\infty)$, $r \in (0,\infty)$ and $z \in \Gamma$.
For all $n \in \Ni_0$ define the annulus
$I_n = \{ w \in \Gamma : 2^{-n-1} \, r \leq |w-z| \leq 2^{-n} \, r \} $.
Then 
\[
\int_{ \{ w \in \Gamma : |w-z| \leq r \} }
   \frac{1}{|w-z|^{d-1-\beta}} \, dw
\leq \sum_{n=0}^\infty \frac{1}{(2^{-n-1} \, r)^{d-1-\beta}}
   \cdot c \, (2^{-n} \, r)^{d-1}
= \frac{2^{d-1} \, c}{2^\beta - 1} \, r^\beta
,   \]
where we used that $d-1-\beta \geq 0$ in the inequality.

`\ref{lcom304-3}'.
Let $\beta \in (0,\infty)$, $r \in (0,\infty)$ and $x \in \Omega$.
Let $z \in \Gamma$ be such that $d(x,\Gamma) = |x-z|$.
If $w \in \Gamma$, then $|x-z| = d(x,\Gamma) \leq |x-w|$ and hence
$|w-z| \leq |w-x| + |x-z| \leq 2 |w-x|$.
We consider two cases.
Let $c_1 > 0$ be as in Statement~\ref{lcom304-1}.

\noindent
{\bf Case 1. } Suppose $r \geq 2 d(x,\Gamma)$.
Let $w \in \Gamma$ with $|w-x| \geq r$.
Then $|w-z| \geq |w-x| - |x-z| \geq r - \frac{1}{2} \, r = \frac{1}{2} \, r$.
Therefore
\begin{eqnarray*}
\int_{ \{ w \in \Gamma : |w-x| \geq r \} }
   \frac{1}{|w-x|^{d-1+\beta}} \, dw
& \leq & \int_{ \{ w \in \Gamma : |w-z| \geq r/2 \} }
  \Big( \frac{2}{|w-z|} \Big)^{d-1+\beta} \, dw  \\
& \leq & 2^{d-1+\beta} \, \frac{c_1}{1 - 2^{-\beta}} \, \frac{1}{(r/2)^\beta}  
= \frac{2^{d-1} \, 4^\beta \, c_1}{1 - 2^{-\beta}} \, \frac{1}{r^\beta} 
\end{eqnarray*}
as required.

\noindent
{\bf Case 2. } Suppose $r \leq 2 d(x,\Gamma)$.
We split the integral in two parts: the first is over
$ \{ w \in \Gamma : r \leq |w-x| < 2 d(x,\Gamma) \} $
and the second is over $ \{ w \in \Gamma : |w-x| \geq 2 d(x,\Gamma) \} $.
The second is covered in Case~1, so we have to deal with the first part.
Let $w \in \Gamma$ and suppose that $|w-x| < 2 d(x,\Gamma)$.
Then $|w-z| \leq |w-x| + |x-z| \leq 3 |x-z|$.
So 
\begin{eqnarray*}
\int_{ \{ w \in \Gamma : |w-x| < 2 d(x,\Gamma) \} }
   \frac{1}{|w-x|^{d-1+\beta}} \, dw
& \leq & \int_{ \{ w \in \Gamma : |w-z| < 3 |x-z| \} }
   \frac{1}{|z-x|^{d-1+\beta}} \, dw  \\
& = & \frac{\sigma(\Gamma \cap B(z,3 |x-z|))}{|z-x|^{d-1+\beta}}
\leq \frac{c \, (3 |x-z|)^{d-1}}{|z-x|^{d-1+\beta}}  \\
& = & \frac{3^{d-1} c}{|z-x|^\beta}
\leq \frac{2^\beta \, 3^{d-1} c}{r^\beta}
.
\end{eqnarray*}
Therefore both contributions together give
\begin{eqnarray*}
\int_{ \{ w \in \Gamma : |w-x| \geq r \} }
   \frac{1}{|w-x|^{d-1+\beta}} \, dw
& \leq & \frac{2^\beta \, 3^{d-1} c}{r^\beta}
   + \frac{2^{d-1} \, 4^\beta \, c_1}{1 - 2^{-\beta}} \, \frac{1}{(2 d(x,\Gamma))^\beta}  \\
& \leq & \frac{2^\beta \, 3^{d-1} c}{r^\beta}
   + \frac{2^{d-1} \, 4^\beta \, c_1}{1 - 2^{-\beta}} \, \frac{1}{r^\beta} 
\end{eqnarray*}
as required.

`\ref{lcom304-4}'.
Let $\beta \in (0,d-1]$, $r \in (0,\infty)$ and $x \in \Omega$.
If $d(x,\Gamma) > r$, then the inequality is trivial.
So we may assume that $d(x,\Gamma) \leq r$.
Let $z \in \Gamma$ be such that $d(x,\Gamma) = |x-z|$.
If $w \in \Gamma$ and $|w-x| \leq r$, then 
$|w-z| \leq |w-x| + |x-z| \leq r + d(x,\Gamma) \leq 2 r$.
So $ \{ w \in \Gamma : |w-x| \leq r \} \subset \{ w \in \Gamma : |w-z| \leq 2r \} $.
We distinguish two cases.
Let $c_2 > 0$ be as in Statement~\ref{lcom304-2}.

\noindent
{\bf Case 1. } Suppose $r \leq 2 d(x,\Gamma)$.
Let $w \in \Gamma$ and suppose $|w-x| \leq r$.
Then $|w-z| \leq 2 r \leq 4 d(x,\Gamma) \leq 4 |x-w|$.
Hence 
\begin{eqnarray*}
\int_{ \{ w \in \Gamma : |w-x| \leq r \} }
   \frac{1}{|w-x|^{d-1-\beta}} \, dw
& \leq & \int_{ \{ w \in \Gamma : |w-z| \leq 2r \} }
  \Big( \frac{4}{|w-z|} \Big)^{d-1-\beta} \, dw  \\
& \leq & 4^{d-1-\beta} \, \frac{c_2}{1 - 2^{-\beta}} \, (2r)^\beta 
= \frac{4^{d-1} \, c_2 \, 2^{-\beta}}{1 - 2^{-\beta}} \, r^\beta
\end{eqnarray*}
as required.

\noindent
{\bf Case 2. } Suppose $r \geq 2 d(x,\Gamma)$.
We split the integral in two parts: the first is over
$ \{ w \in \Gamma : |w-x| \leq 2 d(x,\Gamma) \} $
and the second is over $ \{ w \in \Gamma : 2 d(x,\Gamma) < |w-x| \leq r \} $.
The first is covered in Case~1, so we have to deal with the second part.
If $w \in \Gamma$ and $2 d(x,\Gamma) < |w-x|$,
then $|w-z| \leq |w-x| + |x-z| \leq |w-x| + \frac{1}{2} \, |w-x| \leq 2 |w-x|$.
Hence 
\begin{eqnarray*}
\int_{ \{ w \in \Gamma : 2 d(x,\Gamma) < |w-x| \leq r \} }
   \frac{1}{|w-x|^{d-1-\beta}} \, dw
& \leq & \int_{ \{ w \in \Gamma : |w-z| \leq 2r \} }
  \Big( \frac{2}{|w-z|} \Big)^{d-1-\beta} \, dw  \\
& \leq & 2^{d-1-\beta} \, \frac{c_2}{1 - 2^{-\beta}} \, (2r)^\beta   \\
& = & \frac{2^{d-1} \, c_2}{1 - 2^{-\beta}} \, r^\beta
.
\end{eqnarray*}
Therefore
\begin{eqnarray*}
\int_{ \{ w \in \Gamma : |w-x| \leq r \} }
   \frac{1}{|w-x|^{d-1-\beta}} \, dw
& \leq & \frac{4^{d-1} \, c_2 \, 2^{-\beta}}{1 - 2^{-\beta}} \, (2 d(x,\Gamma))^\beta
   + \frac{2^{d-1} \, c_2}{1 - 2^{-\beta}} \, r^\beta  \\
& \leq & \frac{4^{d-1} \, c_2}{1 - 2^{-\beta}} \, r^\beta
   + \frac{2^{d-1} \, c_2}{1 - 2^{-\beta}} \, r^\beta
\end{eqnarray*}
as required.
\end{proof}

We need the following extension for H\"older continuous functions.

\begin{lemma} \label{lcom502}
Let $\nu \in (0,1)$ and $\varphi \in C^\nu(\Gamma)$.
Then there exists a $\Phi \in C^\nu(\Ri^d)$
such that $\Phi|_\Gamma = \varphi$, 
\[
|\Phi(x) - \Phi(y)| \leq 2 \sqrt{2} \, \|\varphi\|_{C^\nu(\Gamma)} \, |x-y|^\nu
\]
for all $x,y \in \Ri^d$ and 
\[
\|\, \Phi|_\Omega \, \|_{C^\nu(\Omega)}
\leq \Big( 2 + 2 \sqrt{2} (\diam \Omega)^\nu \Big) \, \|\varphi\|_{C^\nu(\Gamma)}
.  \]
\end{lemma}
\begin{proof}
First note that $|\varphi(w) - \varphi(z)| \leq 2 \|\varphi\|_{C^\nu(\Gamma)} \, |w-z|^\nu$
for all $w,z \in \Gamma$, also if $|w-z| > 1$.
Then the existence and H\"older estimate follows from \cite{McShane} Theorem~1.
The factor $\sqrt{2}$ arises since we work over the complex numbers.

Fix $z \in \Gamma$.
If $x \in \overline \Omega$, then 
\begin{eqnarray*}
|\Phi(x)|
& \leq & |\Phi(x) - \Phi(z)| + |\varphi(z)|  \\
& \leq & \sqrt{2} \, \|\varphi\|_{C^\nu(\Gamma)} \, |x-z|^\nu + \|\varphi\|_{L_\infty(\Gamma)}  
\leq \Big( 2 \sqrt{2} (\diam \Omega)^\nu + 1 \Big) \, \|\varphi\|_{C^\nu(\Gamma)}
\end{eqnarray*}
and the norm estimate follows.
\end{proof}

\begin{proof}[{\bf Proof of Theorem~\ref{tcom102}\ref{tcom102-1}.}] 
Let $T_\lambda \colon L_1(\Gamma) \to C(\Omega)$ be the extension of $\gamma_\lambda$
as in Theorem~\ref{tcom204}.
For simplicity we first drop the subscript $\lambda$.

Recall that $T|_{H^{1/2}(\Gamma)} = \gamma$.
Let $\tilde c$ be as in Theorem~\ref{tcom204}\ref{tcom204-3}.
Let $c_3$ be as in Lemma~\ref{lcom304}\ref{lcom304-3}.
Let $c_4$ be as in Lemma~\ref{lcom304}\ref{lcom304-4}.
Let $\hat c$ be as in Lemma~\ref{lcom301}.
Let $c' > 0$ be as in Lemma~\ref{lcom307}.

Fix $\varphi \in C^\nu(\Gamma)$.
Let $\Phi \in C^\nu(\overline{\Omega})$ be an extension of $\varphi$ 
as in Lemma~\ref{lcom502}. 
Let $x, x_0 \in \Omega$ with $|x - x_0| \leq 1$.
Then
\begin{eqnarray*}
(T \varphi)(x) - (T \varphi)(x_0) 
&=& \int_\Gamma \Big( K_\gamma(x,z) - K_\gamma(x_0,z) \Big) \varphi(z) \, dz\\
&=& \int_\Gamma \Big( K_\gamma(x,z) - K_\gamma(x_0,z) \Big) ( \varphi(z) - \Phi(x_0)) \, dz  \\*
& & {} 
   + \Phi(x_0) \Big( (\gamma(\one_\Gamma))(x) - (\gamma(\one_\Gamma))(x_0) \Big)  \\
&=& \int_\Gamma \Big( K_\gamma(x,z) - K_\gamma(x_0,z) \Big) \chi_r(x-z)  \, 
         ( \varphi(z) - \Phi(x_0)) \, dz  \\*
&& {} + \int_\Gamma \Big( K_\gamma(x,z) - K_\gamma(x_0,z) \Big) (1- \chi_r(x-z)) 
         \, ( \varphi(z) - \Phi(x_0)) \, dz  \\*
& & {} 
   + \Phi(x_0) \Big( (\gamma(\one_\Gamma))(x) - (\gamma(\one_\Gamma))(x_0) \Big)  \\
&=& I_1 + I_2 + I_3,
\end{eqnarray*}
where $\chi_r$ is the  cut-off function of  Lemma~\ref{lcom301} 
for which we take  $r = |x-x_0|$. 

We first deal with the third term $I_3$.
It follows from Lemma~\ref{lcom308} with $r = \diam \Omega$ that 
$\gamma(\one_\Gamma) = \gamma(\one|_\Gamma) \in C^{1+\kappa}(\Omega)$
Hence $\gamma(\one_\Gamma) \in C^\nu(\Omega)$ and so 
\[
|\Phi(x_0) \Big( (\gamma(\one_\Gamma))(x) - (\gamma(\one_\Gamma))(x_0) \Big)|
\leq ( 2 + 2 \sqrt{2} (\diam \Omega)^\nu ) \, \|\varphi\|_{C^\nu(\Gamma)} \, 
    |||\gamma(\one_\Gamma)|||_{C^\nu(\Omega)} \, |x-x_0|^\nu
\]
by the bounds of Lemma~\ref{lcom502}.

Next we consider $I_2$.
By the definition of $\chi_r$, the integral in $I_2$ concerns the region $ |z-x| > 2 |x-x_0|$.
Hence by the H\"older bounds of $K_\gamma$ in Theorem~\ref{tcom204}\ref{tcom204-3}
one estimates
\begin{eqnarray*}
| I_2 | 
&=&  | \int_{ \{ z \in \Gamma : |x-z| > 2 |x-x_0| \} }  \Big( K_\gamma(x,z) - K_\gamma(x_0,z) \Big)
   (1- \chi_r(x-z)) ( \varphi(z) - \Phi(x_0)) \, dz| \\
&\le& \sqrt{2} \, \tilde c \, ||| \varphi |||_{C^\nu(\Gamma)} 
       \int_{ \{ z \in \Gamma : |x-z| > 2 |x-x_0| \} }  
    \frac{|x-x_0|^\kappa}{|x-z|^{d-1+\kappa}} \, |x_0-z|^\nu \, dz\\
&\le& 4 \tilde c \, ||| \varphi |||_{C^\nu(\Gamma)} 
       \int_{ \{ z \in \Gamma : |x-z| > 2 |x-x_0| \} } 
    \frac{|x-x_0|^\kappa}{|x-z|^{d-1+\kappa}} \, |x-z|^\nu \, dz    \\
& \leq & \frac{16 \tilde c \, c_3}{2^{\kappa - \nu} - 1} \, 
    ||| \varphi |||_{C^\nu(\Gamma)} \, |x-x_0|^\nu 
,
\end{eqnarray*}
where we used the triangle inequality $|x_0 - z| \leq |x_0-x| + |x-z| \le 2 |x-z|$
and Lemma~\ref{lcom304}\ref{lcom304-3} with $\beta = \kappa - \nu$.
 
Now we deal with $I_1$. 
We decompose it as follows
\begin{eqnarray*}
I_1 
&=& \int_\Gamma - K_\gamma(x_0,z) \, \chi_r(x-z) \, ( \varphi(z) - \Phi(x_0)) \, dz 
   + \int_\Gamma K_\gamma(x,z) \, \chi_r(x-z) \, ( \varphi(z) - \Phi(x)) \, dz\\
&& {}
   + \int_\Gamma  K_\gamma(x,z) \, \chi_r(x-z) \, ( \Phi(x) - \Phi(x_0)) \, dz   \\
& = & I_{1a} + I_{1b} + I_{1c}.
\end{eqnarray*}
We use the kernel bound of Theorem~\ref{tcom204}\ref{tcom204-3} for $I_{1a}$.
If $z \in \Gamma$ and $\chi_r(x-z) \neq 0$, then $|x-z| \leq 3 r = 3 |x-x_0|$
and $|z-x_0| \leq 4 |x-x_0|$.
Therefore
\begin{eqnarray*}
| I_{1a} | 
& \leq & \sqrt{2} \, \tilde c \, ||| \varphi |||_{C^\nu(\Gamma)}  
   \int_{ \{ z \in \Gamma : |z-x_0| \le 4 |x-x_0| \} }  
      \frac{1}{| x_0 -z|^{d-1}} \, |x_0-z|^\nu \, dz  \\
& \leq & \sqrt{2} \, \tilde c \, ||| \varphi |||_{C^\nu(\Gamma)}  
   \frac{c_4}{1 - 2^{-\nu}} \, (4 |x-x_0|)^\nu
. 
\end{eqnarray*}
The term $I_{1b}$ is estimated exactly in the same way.

It remains to estimate $I_{1c}$.
Clearly,
\[
| I_{1c} | 
\leq \sqrt{2} \, ||| \varphi |||_{C^\nu(\Gamma)} \, |x-x_0|^\nu \, | \gamma(\Psi|_\Gamma) (x)|
\leq \sqrt{2} \, c' \, ||| \varphi |||_{C^\nu(\Gamma)} \, |x-x_0|^\nu 
,  \]
where $\Psi \colon \Ri^d \to \Ri$ is defined by $\Psi(y) = \chi_r(x-y)$
and we used Lemma~\ref{lcom307}.

Together we proved that there is a $c_5 > 0$, independent of $\varphi$, $x$ and $x_0$,
such that 
\[
|(T \varphi)(x) - (T \varphi)(x_0)|
\leq c_5 \, \|\varphi\|_{C^\nu(\Gamma)} \, |x-x_0|^\nu
.  \]
Next we provide an $L_\infty$-bound for $T \varphi$.
Let $x \in \Omega$.
Then 
\[
(T \varphi)(x)
= \int_\Gamma K_\gamma(x,z) \, \varphi(z) \, dz
= \int_\Gamma K_\gamma(x,z) \, (\varphi(z) - \Phi(x)) \, dz
   + \Phi(x) \, (\gamma(\one_\Gamma))(x)
.  \]
Using the kernel bound of Theorem~\ref{tcom204}\ref{tcom204-3}, 
Lemma~\ref{lcom502} and Lemma~\ref{lcom304}\ref{lcom304-4} one estimates
\begin{eqnarray*}
|\int_\Gamma K_\gamma(x,z) \, (\varphi(z) - \Phi(x)) \, dz|
& \leq & \int_\Gamma \frac{\tilde c}{|x-z|^{d-1}} 
     \cdot 2 \sqrt{2} \, \|\varphi\|_{C^\nu(\Omega)} \, |z-x|^\nu \, dz  \\
& \leq & 2 \sqrt{2} \, \tilde c \, \frac{c_4}{1 - 2^{-\nu}} \, (\diam \Omega))^\nu \, 
      \|\varphi\|_{C^\nu(\Gamma)}
,  
\end{eqnarray*}
and 
\[
|\Phi(x) \, (\gamma(\one_\Gamma))(x)|
\leq (2 + 2 \sqrt{2} \, (\diam \Omega)^\nu) \, 
    \|\gamma(\one_\Gamma)\|_{L_\infty(\Omega)} \, 
     \|\varphi\|_{C^\nu(\Gamma)}
.  \]
Restoring the notation with $\lambda$, we proved that $T_\lambda$ maps
$C^\nu(\Gamma)$ into $C^\nu(\Omega)$ and there exists a $c_6 > 0$
such that 
$\|T_\lambda \varphi\|_{C^\nu(\Omega)} 
\leq c_6 \, \|\varphi\|_{C^\nu(\Gamma)}$.

Finally, $\gamma = \gamma_\lambda + \lambda \, A_D^{-1} \, \gamma_\lambda$
and hence by continuity and density
$T = T_\lambda + \lambda \, A_D^{-1} \, T_\lambda$.
Now $T_\lambda$ maps $(C^\nu(\Gamma), \|\cdot\|_{C^\nu(\Gamma)})$ continuously into 
$(C^\nu(\Omega), \|\cdot\|_{C^\nu(\Omega)})$ and $L_\infty(\Omega)$,
and $A_D^{-1}$ maps $L_\infty(\Omega)$ continuously into 
$(C^\nu(\Omega), \|\cdot\|_{C^\nu(\Omega)})$
by Proposition~\ref{pcom209}\ref{pcom209-2}.
This completes the proof of Theorem~\ref{tcom102}\ref{tcom102-1}.
\end{proof}

\section{The Schwartz kernel of the Dirichlet-to-Neumann operator} \label{Scom8}

We next turn to the Dirichlet-to-Neumann operator.
Throughout this short section let $\kappa \in (0,1)$.
Let $\Omega \subset \Ri^d$ be an open bounded set of class~$C^{1+\kappa}$.
For all $k,l \in \{ 1,\ldots,d \} $ let $c_{kl}, b_k, c_k \in C^\kappa(\Omega)$,
and let $c_0 \in L_\infty(\Omega)$.
Suppose there is a $\mu > 0$ such that 
$\RRe \sum_{k,l=1}^d c_{kl}(x) \, \xi_k \, \overline{\xi_l} \geq \mu \, |\xi|^2$
for all $\xi \in \Ci^d$ and $x \in \Omega$.
Further suppose that $0 \not\in \sigma(A_D)$.

Let $G \colon \{ (x,y) \in \Omega \times \Omega : x \neq y \} \to \Ci$
be as in Theorem~\ref{tcom202}.
Then $G$ is differentiable on $ \{ (x,y) \in \Omega \times \Omega : x \neq y \} $ 
by Theorem~\ref{tcom202} and the derivatives extend to  continuous
functions on $ \{ (x,y) \in \overline \Omega \times \overline \Omega : x \neq y \} $.
Define $\widetilde K_\cn \colon \{ (z,w) \in  \Gamma \times \Gamma : z \neq w \} \to \Ci$ by
\begin{eqnarray*}
\widetilde K_\cn(z,w)
& = & - \sum_{k,l,k',l'=1}^d n_k(z) \, c_{kl}(z) \, n_{k'}(w) \, c_{l'k'}(w) \, (\partial^{(1)}_l \partial^{(2)}_{l'} G)(z,w)  \\*
& & {}
   - \sum_{k,k',l'=1}^d n_k(z) \, b_k(z) \, n_{k'}(w) \, c_{l'k'}(w) \, (\partial^{(2)}_{l'} G)(z,w)  \\*
& & {}
   - \sum_{k,k',l=1}^d n_k(z) \, c_{kl}(z) \, n_{k'}(w) \, c_{k'}(w) \, (\partial^{(1)}_l G)(z,w)  \\*
& & {}
   - \sum_{k,k'=1}^d n_k(z) \, b_k(z) \, n_{k'}(w) \, c_{k'}(w) \, G(z,w)
.
\end{eqnarray*}

For a possibly unbounded operator $Z$ on a given space
$L_2(X,\cb,\nu)$, where   $(X,\cb,\nu)$ is a $\sigma$-finite measure space,
and a measurable function $K \colon X \times X \to \Ci$, 
we say that $K$ is the {\bf Schwartz kernel} of $Z$ if 
\[
(Z u)(x)
= \int_X K(x,y) \, u(y) \, d\nu(y) 
\]
for all $u \in D(Z)$ and a.e.\ $x \in X$ with 
$x \not\in \supp u$.

\begin{prop} \label{pcom205.5}
Suppose that  $\frac{\mu}{2} \, \|u\|_{W^{1,2}(\Omega)}^2
\leq \RRe \gota(u,u)$
for all $u \in W^{1,2}(\Omega)$.
Then $\widetilde K_\cn$ is the Schwartz kernel of $\cn$.
\end{prop}
\begin{proof}
This follows as for the proof of Corollary~6.3 in \cite{EO6}.
Because of the coercivity condition 
$\frac{\mu}{2} \, \|u\|_{W^{1,2}(\Omega)}^2 \leq \RRe \gota(u,u)$
for all $u \in W^{1,2}(\Omega)$, we can regularise the coefficients
and the uniform estimates as in \cite{EO6} Proposition~6.6 are also valid in 
the current situation because of Theorem~\ref{tcom202}.
\end{proof}

For the general setting, that is without the coercivity condition, we need the following  lemma.

\begin{lemma} \label{lcom205.7}
Let $\lambda \in \Ri$.
Suppose that $0 \not\in \sigma(A_D + \lambda \, I)$.
Let $\gamma \colon H^{1/2}(\Gamma) \to H^1(\Omega)$ be the harmonic lifting and 
let $\gamma_\lambda^\# \colon H^{1/2}(\Gamma) \to H^1(\Omega)$ be the harmonic lifting 
with respect to the form $\gota_\lambda^*$.
We identify $\gamma_\lambda^\#$ with the bounded extension from $L_2(\Gamma)$ into 
$L_2(\Omega)$ (cf.\ Theorem~\ref{tcom204}\ref{tcom204-4}).
Then 
\[
\cn = \cn_\lambda - \lambda \, (\gamma_\lambda^\#)^* \, \gamma
.  \]
\end{lemma}
\begin{proof}
Let $\varphi \in D(\cn)$ and $\psi \in D(\cn_\lambda^*)$.
Write $u = \gamma (\varphi)$ and $v = \gamma_\lambda^\# (\psi)$.
Then 
$(\cn_\lambda^* \psi, \varphi)_{L_2(\Gamma)}
= \gota_\lambda^*(v,u)
= \overline{ \gota_\lambda(u,v) } $
by \cite{AE2} Remark~2.8.
So 
\begin{eqnarray*}
(\cn \varphi, \psi)_{L_2(\Gamma)} - (\varphi, \cn_\lambda^* \psi)_{L_2(\Gamma)}
& = & \gota(u,v) - \gota_\lambda(u,v)  \\
& = & - \lambda \, (u,v)_{L_2(\Omega)}
= - \lambda \, ((\gamma_\lambda^\#)^* \, \gamma (\varphi), \psi)_{L_2(\Gamma)}
.  
\end{eqnarray*}
By the boundedness of $(\gamma_\lambda^\#)^* \, \gamma$ from $L_2(\Gamma)$ into $L_2(\Gamma)$
one deduces that $\varphi \in D(\cn_\lambda^{**}) = D(\cn_\lambda)$.
Hence 
\[
((\cn - \cn_\lambda) \varphi, \psi)_{L_2(\Gamma)}
= - \lambda \, ((\gamma_\lambda^\#)^* \, \gamma (\varphi), \psi)_{L_2(\Gamma)}
.  \]
Since $D(\cn_\lambda^*)$ is dense in $L_2(\Gamma)$
and similarly $D(\cn_\lambda) \subset D(\cn)$, the lemma follows.
\end{proof}

In the next theorem we show that $\cn$ has a Schwartz kernel
and this kernel satisfies  appropriate  bounds.

\begin{thm} \label{tcom205}
The Dirichlet-to-Neumann operator $\cn$ has a Schwartz kernel $K_\cn$.
Moreover, there exists a $c > 0$ 
such that 
\[ 
| K_{\cn}(z,w) | \le \frac{c}{| z-w|^d}
\]
and 
\[
| K_{\cn}(z,w) - K_{\cn}(z',w') | 
\leq c \, \frac{ (| z-z'| + | w-w'|)^\kappa}{ | z-w|^{d+ \kappa}}
\]
for all $z, z', w, w' \in \Gamma$ with $z \neq w$, $z' \neq w'$ and $|z-z'| + |w-w'| \leq \frac{1}{2} \, |z-w|$.
\end{thm}
\begin{proof}
There exists a $\lambda > 0$ such that 
$\frac{\mu}{2} \, \|u\|_{W^{1,2}(\Omega)}^2
\leq \RRe \gota(u,u) + \lambda \, \|u\|_{L_2(\Omega)}^2$
for all $u \in W^{1,2}(\Omega)$.
Let $\gamma_\lambda^\# \colon H^{1/2}(\Gamma) \to H^1(\Omega)$ be the harmonic lifting 
with respect to the form~$\gota_\lambda^*$.
Then $\cn = \cn_\lambda - \lambda \, (\gamma_\lambda^\#)^* \, \gamma$ by 
Lemma~\ref{lcom205.7}.
Now $(\gamma_\lambda^\#)^* \, \gamma$ has a Schwartz kernel $K_Q$ given by 
\[
K_Q(z,w)
= \int_\Omega \overline{ K_{\gamma_\lambda^\#}(x,z) } \, K_\gamma(x,w) \, dx
\]
for all $z,w \in \Gamma$ with $z \neq w$, where we use the notation of 
Theorem~\ref{tcom204}.
Hence $\widetilde K_{\cn_\lambda} - \lambda \, K_Q$
is the Schwartz kernel for $\cn$.

It remains  to show the kernel bounds and it suffices to show that 
$\widetilde K_{\cn_\lambda}$ and $K_Q$ satisfy the appropriate kernel bounds.
Using the bounds  in Theorem~\ref{tcom202.5} for the Green function we obtain the desired bounds for  the kernel  $\widetilde K_{\cn_\lambda}$ as in the proof 
of \cite{EO6} Proposition~6.4. Using the bounds  in Theorem~\ref{tcom204} for the kernel of the harmonic lifting we obtain the bounds for  $K_Q$ as in 
the proof of \cite{EO6} Proposition~6.5.
\end{proof}

\section{Commutator estimates} \label{Scom9}

Throughout this section let $\kappa \in (0,1)$ and $\Omega \subset \Ri^d$
a bounded open set of class $C^{1+\kappa}$.
For all $k,l \in \{ 1,\ldots,d \} $ let $c_{kl}, b_k,c_k \in C^\kappa(\Omega)$ 
and $c_0 \in L_\infty(\Omega)$.
Let the form $\gota$ and operator $A_D$ be as in the introduction.
We assume that $0 \not\in \sigma(A_D)$.
Let $\cn$ be the associated Dirichlet-to-Neumann operator.
In this section we prove the commutator estimates of Theorem~\ref{tcom101}.

We start with removing the shift in the last two statements of Lemma~\ref{lcom308}.

\begin{lemma} \label{lcom901}
There exists an $\widetilde M > 0$ such that the following is valid.

Let $r \in (0,\diam \Omega]$, $\Psi \in C^{1+\kappa}(\Omega)$,
$y \in \overline \Omega$ and $c \geq 0$.
Suppose that $\supp \Psi \subset B(y,3r)$,
\[
\|\Psi\|_{L_\infty(\Omega)} \leq c \, r
 , \quad
\|\partial_k \Psi\|_{L_\infty(\Omega)} \leq c
\quad \mbox{and} \quad
|||\partial_k \Psi|||_{C^\kappa(\Omega)} \leq c \, r^{-\kappa}
\]
for all $k \in \{ 1,\ldots,d \} $.
Then $\Psi|_\Gamma \in D(\cn)$ and $\cn(\Psi|_\Gamma) \in C^\kappa(\Gamma)$.
Moreover,
\begin{tabel}
\item \label{lcom901-4}
$\|\cn(\Psi|_\Gamma)\|_{C(\Gamma)} \leq \widetilde M \, c$, and
\item \label{lcom901-5}
$|||\cn(\Psi|_\Gamma)|||_{C^\kappa(\Gamma)} \leq \widetilde M \, c \, r^{-\kappa}$.
\end{tabel}
\end{lemma}
\begin{proof}
Let $M > 0$ be such that 
$\|c_{kl}\|_{C^\kappa(\Omega)} \leq M$,
$\|b_k\|_{C^\kappa(\Omega)} \leq M$, $\|c_k\|_{C^\kappa(\Omega)} \leq M$
and $\|c_0\|_{L_\infty(\Omega)} \leq M$.
There exists a $\lambda > 0$, depending only on $\mu$ and $M$, such that 
\[
\frac{\mu}{2} \, \|u\|_{W^{1,2}(\Omega)}^2
\leq \RRe \gota(u,u) + \lambda \, \|u\|_{L_2(\Omega)}^2
\]
for all $u \in W^{1,2}(\Omega)$.
Let $\gamma_\lambda$ and $\cn_\lambda$ be the harmonic lifting and 
Dirichlet-to-Neumann operator with respect to the operator with 
coefficient $c_0 + \lambda \, \one_\Omega$ instead of $c_0$.
Let $\widetilde M$ be as in Lemma~\ref{lcom308}.
Then $\|\cn_\lambda(\Psi|_\Gamma)\|_{C(\Gamma)} \leq \widetilde M \, c$
and $|||\cn_\lambda(\Psi|_\Gamma)|||_{C^\kappa(\Gamma)} \leq \widetilde M \, c \, r^{-\kappa}$.
Let $\gamma_\lambda^\#$ be as in Lemma~\ref{lcom205.7}.
Then $\cn = \cn_\lambda - \lambda \, (\gamma_\lambda^\#)^* \, \gamma$ by 
Lemma~\ref{lcom205.7}.
It remains to show that $(\gamma_\lambda^\#)^* \, \gamma$ can be considered 
as a perturbation on $L_\infty(\Gamma)$ and $C^\kappa(\Gamma)$.

Fix $p \in (\frac{d}{1-\kappa}, \infty)$.
The inclusion $L_\infty(\Gamma) \to L_p(\Gamma)$ is bounded.
The map $\gamma$ is bounded from $L_p(\Gamma)$ into $L_p(\Omega)$ by 
Theorem~\ref{tcom204}\ref{tcom204-4}.
The map $(\gamma_\lambda^\#)^*$ is bounded from $L_p(\Omega)$
into $(C^\kappa(\Gamma), \|\cdot\|_{C^\kappa(\Gamma)})$ by Theorem~\ref{tcom204}\ref{tcom204-5}.
The inclusion from $(C^\kappa(\Gamma), \|\cdot\|_{C^\kappa(\Gamma)})$
into $L_\infty(\Gamma)$ is clearly bounded.
Hence by composition the map 
$\lambda \, (\gamma_\lambda^\#)^* \, \gamma$
is bounded from $L_\infty(\Gamma)$ into $L_\infty(\Gamma)$.
Let $c_1 = \|\lambda \, (\gamma_\lambda^\#)^* \, \gamma\|_{L_\infty(\Gamma) \to L_\infty(\Gamma)}$.
Then 
\[
\|\lambda \, (\gamma_\lambda^\#)^* \, \gamma(\Psi|_\Gamma)\|_{L_\infty(\Gamma)}
\leq c_1 \, \|\Psi|_\Gamma\|_{L_\infty(\Gamma)}
\leq c_1 \, c \, r
\leq c_1 \, c \, \diam(\Omega)
\]
and Statement~\ref{lcom901-4} follows.

The proof of Statement~\ref{lcom901-5} is similar. 
The inclusion $C^\kappa(\Gamma)  \to L_p(\Gamma)$ is bounded, 
and again $\gamma$ is bounded from $L_p(\Gamma)$ 
into $L_p(\Omega)$ and $(\gamma_\lambda^\#)^*$ is bounded from $L_p(\Omega)$
into $(C^\kappa(\Gamma), \|\cdot\|_{C^\kappa(\Gamma)})$. 
\end{proof}

A principle part of the proof of Theorem~\ref{tcom101} is given in the  next lemma.

\begin{lemma} \label{lcom302}
Let $\hat c \geq 1$ and the $\chi_r$ be as in Lemma~\ref{lcom301}.
Then there exists a $c' > 0$ such that the following is valid.

Let $r \in (0,\diam(\Omega)]$ and $z \in \Gamma$.
Define $\Psi \colon \Ri^d \to \Ri$ by 
$\Psi(x) = \chi_r(x-z)$.
Then $(g \, \Psi)|_\Gamma \in D(\cn)$, $\cn((g \, \Psi)|_\Gamma) \in C(\Gamma)$ and 
$\|\cn((g \, \Psi)|_\Gamma)\|_{L_\infty(\Gamma)} \leq c'$ for all $g \in W_2$
with $g(z) = 0$.
\end{lemma}
\begin{proof}
Note that $\supp(g \, \Psi) \subset B(z,3r)$.
If $x \in B(z,3r)$ and $k,l \in \{ 1,\ldots,d \} $, then 
$|g(x)| = |g(x) - g(z)| \leq |x-z| \leq 3r$, 
$|(g \, \Psi)(x)| \leq 3r$,
\[
|(\partial_l(g \, \Psi))(x)|
\leq |(\partial_l g)(x)| \, \Psi(x)
   + |g(x)| \, |(\partial_l \Psi)(x)|
\leq 1 \cdot 1 + 3r \cdot \frac{\hat c}{r}
\leq 4 \hat c
\]
and 
\begin{eqnarray*}
|(\partial_k \, \partial_l(g \, \Psi))(x)|
& \leq & |(\partial_k \, \partial_l g)(x)| \, \Psi(x)
   + |(\partial_l g)(x)| \, |(\partial_k \Psi)(x)|  \\*
& & \hspace*{20mm} {}
   + |(\partial_k g)(x)| \, |(\partial_l \Psi)(x)| 
   + |g(x)| \, |(\partial_k \, \partial_l \Psi)(x)|  \\
& \leq & 1 \cdot 1 + 2 \cdot 1 \cdot \frac{\hat c}{r} + 3r \cdot \frac{\hat c}{r^2}
\leq \frac{5\hat c+ \diam(\Omega)}{r}
.  
\end{eqnarray*}
Here we used a second derivative of $g$.
Obviously these estimates are also valid if $x \in \Ri^d \setminus B(z,3r)$.
Hence if $x,y \in \Ri^d$, then 
\[
|(\partial_l(g \, \Psi))(x) - (\partial_l(g \, \Psi))(y)|
\leq d^{1/2} \, \frac{5\hat c + \diam(\Omega)}{r} \, |x - y|
\]
and trivially 
$|(\partial_l(g \, \Psi))(x) - (\partial_l(g \, \Psi))(y)| \leq 8\hat c$,
so 
\[
|(\partial_l(g \, \Psi))(x) - (\partial_l(g \, \Psi))(y)|
\leq (8\hat c)^{1-\kappa} \, d^{\kappa / 2} \, 
    (5\hat c+\diam(\Omega))^\kappa \, \frac{|x - y|^\kappa}{r^\kappa}
\]
for all $l \in \{ 1,\ldots,d \} $.
Now the lemma follows from  Lemma~\ref{lcom901}, 
applied with $(g \, \Psi)|_\Omega$. 
\end{proof}

\begin{prop} \label{pcom305}
There exists a $c'' > 0$ such that $g|_\Gamma \in D(\cn)$ and 
$\|\cn (g|_\Gamma)\|_{C^\kappa(\Gamma)} \leq c''$
for all $g \in W_2$.
\end{prop}
\begin{proof}
This follows from Lemma~\ref{lcom901}\ref{lcom901-4} and~\ref{lcom901-5}
with the choice $r = \diam \Omega$ and $\Psi = g$.
\end{proof}

Since $\one_{\Ri^d} \in W_2$ one obtains the next corollary.

\begin{cor} \label{ccom306}
$\one_\Gamma \in D(\cn)$ and $\cn \one_\Gamma \in C^\kappa(\Gamma)$.
\end{cor}

Even more general one deduces that many functions are in the domain of 
$\cn$.

\begin{cor} \label{ccom306.5}
Let $\Psi \in W^{2,\infty}(\Ri^d)$.
Then $\Psi|_\Gamma \in D(\cn)$ and $\cn (\Psi|_\Gamma) \in C^\kappa(\Gamma)$.
\end{cor}

Next we show that the Schwartz kernel of the commutator $[\cn, M_g]$ 
satisfies Calder\'on--Zygmund estimates.

\begin{lemma} \label{lcom303}
There exists a $c > 0$ such that for all $g \in W_1$ the 
operator $[\cn, M_g]$ has a Schwartz kernel $K_g$
such that 
\[ 
|K_g(z,w) | \leq \frac{c}{|z-w|^{d-1}}
\]
and 
\[
| K_g(z,w) - K_g(z',w') | 
\leq c \, \frac{ (|z-z'| + |w-w'|)^\kappa}{ |z-w|^{d-1 + \kappa}}
\]
for all $z, z', w, w' \in \Gamma$ with 
$z \neq w$, $z' \neq w'$ and $|z-z'| + |w-w'| \leq \frac{1}{2} \, |z-w|$.
\end{lemma}
\begin{proof} 
By Theorem~\ref{tcom205} the Dirichlet-to-Neumann operator $\cn$
has a Schwartz kernel $K_{\cn}$ and there exists a $c > 0$ such that 
\[ 
| K_{\cn}(z,w) | \leq \frac{c}{| z-w|^d}
\]
and 
\[
|K_{\cn}(z,w) - K_{\cn}(z',w') | 
\leq c \, \frac{ (|z-z'| + |w-w'|)^\kappa}{ |z-w|^{d+ \kappa}}
\]
for all $z, z', w, w' \in \Gamma$ with 
$z \neq w$, $z' \neq w'$ and $|z-z'| + |w-w'| \leq \frac{1}{2} \, |z-w|$.

Now let $g \in W_1$.
Then $[\cn, M_g]$ has a Schwartz kernel $K_g$ given 
by $K_g(z,w) = (g(w) - g(z)) \, K_{\cn}(z,w)$.
Then 
$|K_g(z,w)| \leq |w-z| \, \frac{c}{|z-w|^d} = \frac{c}{|z-w|^{d-1}}$
for all $z,w \in \Gamma$ with $z \neq w$.
Next let $z, z', w, w' \in \Gamma$ with 
$z \neq w$, $z' \neq w'$ and $|z-z'| + |w-w'| \leq \frac{1}{2} \, |z-w|$.
Then $|w' - z'| \leq |w'-w| + |w-z| + |z-z'| \leq 2 |z-w|$.
Therefore
\begin{eqnarray*}
\lefteqn{
|K_g(z,w) - K_g(z',w')|
} \hspace*{2mm} \\*
& = & |(g(w) - g(z)) \, K_{\cn}(z,w) - (g(w') - g(z')) \, K_{\cn}(z',w')|  \\
& = & |(g(w) - g(w') + g(z') - g(z)) \, K_{\cn}(z,w) 
   + (g(w') - g(z')) \, (K_{\cn}(z,w) - K_{\cn}(z',w'))|  \\
& \leq & (|z - z'| + |w - w'|) \, \frac{c}{|z-w|^d}
   + |w'-z'| \, \frac{c \, (|z-z'| + |w-w'|)^\kappa}{|z-w|^{d+\kappa}}  \\
& \leq & (|z - z'| + |w - w'|)^\kappa \, |z-w|^{1-\kappa} \, \frac{c}{|z-w|^d}
   + 2 |w-z| \, \frac{c \, (|z - z'| + |w - w'|)^\kappa}{|z-w|^{d+\kappa}}  \\
& = & 3 c \, \frac{(|z - z'| + |w - w'|)^\kappa}{|z-w|^{d-1+\kappa}}  
\end{eqnarray*}
as required.
\end{proof}

For all $g \in W_1$ let $K_g$ be as in Lemma~\ref{lcom303}.
By the definition of  Schwartz kernel, 
\[
([ \cn , M_g] \varphi)(z) 
= \int_\Gamma K_g(z,w) \, \varphi(w)\, dw
\]
for all $z \in \Gamma \setminus \supp \varphi$. 
For the sequel we need to surpass this support condition. 
To do so, we introduce the following operator. 

For all $g \in W_2$ define $T_g \colon \{ \Phi|_\Gamma : \Phi \in C^\infty(\Ri^d) \} \to C(\Gamma)$
by 
\[
(T_g \varphi)(z)
= \int_\Gamma K_g(z,w) \, \Big( \varphi(w) - \varphi(z) \Big) \, dw
   + \varphi(z) \, ( [\cn,M_g] \one_\Gamma )(z)
.  \]
Note that the integral converges and that 
$\one_\Gamma \in D([\cn,M_g])$ by Corollary~\ref{ccom306.5}.
Moreover, if $\lambda \in \Ri$ and $\tilde g = g - \lambda \, \one_{\Ri^d}$, then 
$T_g = T_{\tilde g}$.

\begin{prop} \label{pcom309}
Let $g \in W_2$ and $\varphi \in \{ \Phi|_\Gamma : \Phi \in C^\infty(\Ri^d) \} $.
Then 
\[
T_g \varphi = [\cn,M_g] \varphi.
\]
\end{prop}
\begin{proof}
Fix $z \in \Gamma$. 
We prove that 
\begin{equation}
(T_g \varphi)(z) = ([\cn,M_g] \varphi)(z).
\label{epcom309;1}
\end{equation}
The proof is given in several steps.

\noindent\firststep \label{pcom309-1}
Suppose $z \not\in \supp \varphi$. 
Then (\ref{epcom309;1}) is valid.   \\ 
This follows from $\varphi(z) = 0$ and the definition of the Schwartz kernel $K_g$ of $[\cn, M_g]$.

\noindent\nextstep \label{pcom309-2}
Suppose $\varphi(z) = g(z) = 0$. 
Then (\ref{epcom309;1}) is valid.  \\
Let $\hat c \geq 1$ and the $\chi_r$ be as in Lemma~\ref{lcom301}.
For all $n \in \Ni$ define $\tau_n \colon \Ri^d \to \Ri$ by 
$\tau_n(y) = \chi_{1/n}(y - z)$.
Then 
\[
(T_g ( ((\one - \tau_n) \, \varphi)|_\Gamma ))(z) 
= ([\cn,M_g] ( (\one - \tau_n) \, \varphi)|_\Gamma )(z)
\]
for all $n \in \Ni$ by Step~\ref{pcom309-1}.
Let $c > 0$ be as in Lemma~\ref{lcom303}.
There is a $\tilde c > 0$ such that $|\varphi(w) - \varphi(z)| \leq \tilde c \, |w-z|$
for all $w \in \Gamma$.
Then
\[
|(T_g( ( \tau_n \, \varphi)|_\Gamma ))(z)|
\leq \int_{\{w \in \Gamma : |w-z| \leq 3/n\}} 
     \frac{c}{|w-z|^{d-1}} \, \tilde c \, |w-z| \, dw
\]
for all $n \in \Ni$ and hence $\lim_{n \to \infty} (T_g( ( \tau_n \, \varphi)|_\Gamma ))(z) = 0$.

It remains to show that 
$\lim_{n \to \infty} ([\cn,M_g] ( (\tau_n \, \varphi)|_\Gamma ))(z) = 0$.
Let $\Phi \in C^\infty(\Ri^d)$ be such that $\Phi|_\Gamma = \varphi$.
Then $\Phi(z) = 0$.
Note that $([\cn,M_g] ( (\tau_n \, \varphi)|_\Gamma ) )(z) = ( \cn( (g \, \tau_n \, \Phi)|_\Gamma) )(z)$
for all $n \in \Ni$.
It follows as in the proof of Lemma~\ref{lcom302} that there is a $\check c > 0$ such that 
\[
\|(g \, \tau_n \, \Phi)|_\Omega\|_{L_\infty(\Omega)} \leq \frac{\check c}{n} \, \frac{1}{n}
 , \quad 
\|\partial_k (g \, \tau_n \, \Phi)|_\Omega\|_{L_\infty(\Omega)} \leq \frac{\check c}{n} 
\quad \mbox{and} \quad
\|\partial_k (g \, \tau_n \, \Phi)|_\Omega\|_{C^\kappa(\Omega)} \leq \frac{\check c}{n} \, n^\kappa
\]
for all $k \in \{ 1,\ldots,d \} $ and $n \in \Ni$.
Let $\widetilde M > 0$ be as in Lemma~\ref{lcom901}.
Then Lemma~\ref{lcom901}\ref{lcom901-4} gives
\[
\|\cn( (g \, \tau_n \, \Phi)|_\Gamma)\|_{C(\Gamma)}
\leq \widetilde M \, \frac{\check c}{n}
\]
for all large $n \in \Ni$.
In particular $\lim_{n \to \infty} ( \cn( (g \, \tau_n \, \Phi)|_\Gamma) )(z) = 0$
and the proof of Step~\ref{pcom309-2} is complete.

\noindent\nextstep \label{pcom309-3}
Suppose $\varphi(z) = 0$. 
Then (\ref{epcom309;1}) is valid.  \\
Define $\tilde g \colon \Ri^d \to \Ri$ by $\tilde g = g - g(z) \, \one_{\Ri^d}$.
Then Step~\ref{pcom309-2} gives
\[
(T_g \varphi)(z) 
= (T_{\tilde g} \varphi)(z) 
= ([\cn,M_{\tilde g}] \varphi)(z)
= ([\cn,M_g] \varphi)(z)
\]
as required.

\noindent\nextstep \label{pcom309-4}
Now we prove the proposition.  \\
Define $\tilde \varphi \colon \Gamma \to \Ci$ by 
$\tilde \varphi = \varphi - \varphi(z)$.
Then Step~\ref{pcom309-3} gives
\begin{eqnarray*}
(T_g \varphi)(z) - \varphi(z) \, ([\cn,M_g] \one_\Gamma)(z)
& = & (T_g \varphi)(z) - \varphi(z) \, (T_g \one_\Gamma)(z)  \\
& = & (T_g \tilde \varphi)(z)  
= ([\cn,M_g] \tilde \varphi)(z)  \\
& = & ([\cn,M_g] \varphi)(z) - \varphi(z) \, ([\cn,M_g] \one_\Gamma)(z)
.
\end{eqnarray*}
This completes the proof.
\end{proof}

Now we are able to prove the commutator estimates of Theorem~\ref{tcom101}.

\begin{proof}[{\bf Proof of Theorem~\ref{tcom101}}]
We give the proof in four steps.

\noindent\firststep \label{tcom101-step1} 
Let $\nu \in (0,\kappa)$.
We prove that there is a $c > 0$ such that 
\[
\|T_g \varphi\|_{C^\nu(\Gamma)}
\leq c \, \|\varphi\|_{C^\nu(\Gamma)}
\]
for all $\varphi \in \{ \Phi|_\Gamma : \Phi \in C^\infty(\Ri^d) \} $
and $g \in W_2$.

\noindent
Let $c > 0$ be as in Lemma~\ref{lcom303}.
Further, let $\hat c \geq 1$ and the $\chi_r$ be as in Lemma~\ref{lcom301}.
Moreover, let $c_1$ and $c_2$ be as in Lemma~\ref{lcom304},
let $c'$ be as in Lemma~\ref{lcom302}
and let $c''$ be as in Proposition~\ref{pcom305}.

Let $g \in W_2$ and let $K_g$ be the Schwartz kernel of $[\cn,M_g]$.
Let $\varphi \in \{ \Phi|_\Gamma : \Phi \in C^\infty(\Ri^d) \} $.
Let $z, z_0 \in \Gamma$ with $|z-z_0| \leq 1$.
We may assume that $z \neq z_0$.
Further without loss of generality we may assume that $g(z) = 0$.
Choose $r = |z - z_0|$.
Then
\begin{eqnarray*}
\lefteqn{
(T_g \varphi)(z) - (T_g \varphi)(z_0)
}  \hspace*{10mm} \\*
& = & \int_\Gamma K_g(z,w) \, \Big( \varphi(w) - \varphi(z) \Big) \, dw
   - \int_\Gamma K_g(z_0,w) \, \Big( \varphi(w) - \varphi(z_0) \Big) \, dw  \\*
& & \hspace*{10mm} {}
   + \varphi(z) \, ( [\cn,M_g] \one_\Gamma )(z)
   - \varphi(z_0) \, ( [\cn,M_g] \one_\Gamma )(z_0)  \\
& = & \int_\Gamma K_g(z,w) \, \chi_r(w-z) \, \Big( \varphi(w) - \varphi(z) \Big) \, dw  \\
& & \hspace*{10mm} {}
   - \int_\Gamma K_g(z_0,w) \, \chi_r(w-z) \, \Big( \varphi(w) - \varphi(z_0) \Big) \, dw  \\
& & \hspace*{10mm} {}
   + \int_\Gamma (K_g(z,w) - K_g(z_0,w) ) \, (1 - \chi_r(w-z)) \, (\varphi(w) - \varphi(z_0)) \, dw    \\
& & \hspace*{10mm} {}
   + \int_\Gamma K_g(z,w) \, (1 - \chi_r(w-z)) \, \Big( \varphi(z_0) - \varphi(z) \Big) \, dw  \\
& & \hspace*{10mm} {}
   + \Big( \varphi(z) \, ( [\cn,M_g] \one_\Gamma )(z)
       - \varphi(z_0) \, ( [\cn,M_g] \one_\Gamma )(z_0) \Big)  \\
& = & I_1 + I_2 + I_3 + I_4 + I_5
.
\end{eqnarray*}
We estimate the five terms separately.

First using Lemma~\ref{lcom304}\ref{lcom304-1} with $\beta = \nu$ we estimate
\begin{eqnarray*}
|I_1|
& \leq & \int_{ \{ w \in \Gamma : |w-z| \leq 3 |z-z_0| \} }
  c \, \frac{1}{|w-z|^{d-1}} \, 2 \|\varphi\|_{C^\nu(\Gamma)} \, |w-z|^\nu  \\
& \leq & 2 c \, \|\varphi\|_{C^\nu(\Gamma)} \, c_2 \, \frac{1}{1 - 2^{-\nu}} \, 
         (3 |z-z_0|)^\nu
= 2 \frac{3^\nu \, c \, c_2}{1 - 2^{-\nu}}  \, \|\varphi\|_{C^\nu(\Gamma)} \, 
           |z-z_0|^\nu
.
\end{eqnarray*}

For the second term, note that if $\chi_r(w-z) \neq 0$, then 
$|w-z| \leq 3 |z-z_0|$ and hence
$|w-z_0| \leq |w-z| + |z-z_0| \leq 4 |z-z_0|$.
So
\[
|I_2|
\leq 2 \frac{4^\nu \, c \, c_2}{1 - 2^{-\nu}}  \, \|\varphi\|_{C^\nu(\Gamma)} \, 
           |z-z_0|^\nu
.  \]

For the third term, if $(1 - \chi_r(w-z)) \neq 0$, then $|w-z| \geq 2r$.
Hence $|w-z_0| \leq |w-z| + |z-z_0| \leq \frac{3}{2} \, |w-z|$.
Therefore 
$|\varphi(w) - \varphi(z_0)| 
\leq 2 \|\varphi\|_{C^\nu(\Gamma)} \, |w-z_0|^\nu
\leq 4 \|\varphi\|_{C^\nu(\Gamma)} \, |w-z|^\nu$.
Then Lemma~\ref{lcom304}\ref{lcom304-1} with $\beta = \kappa - \nu$ gives
\begin{eqnarray*}
|I_3|
& \leq & 4 \|\varphi\|_{C^\nu(\Gamma)}
   \int_{ \{ w \in \Gamma : |w-z| \geq 2 |z-z_0| \} }
   c \, \frac{ |z-z_0|^\kappa}{ |z-w|^{d-1 + \kappa}} \, |w-z|^\nu \, dw \\
& \leq & 4 \, c \, c_1 \, \|\varphi\|_{C^\nu(\Gamma)} \, |z-z_0|^\kappa \, 
    \frac{1}{1 - 2^{-(\kappa - \nu)}}
    \frac{1}{(2 |z - z_0|)^{\kappa - \nu}}
= \frac{4 \, c \, c_1}{2^{\kappa - \nu} - 1} \, \|\varphi\|_{C^\nu(\Gamma)} \, 
    |z - z_0|^\nu
    .
\end{eqnarray*}

For the term $I_4$ define $\Psi \colon \Ri^d \to \Ri$ by $\Psi(x) = \chi_r(x-z)$.
Then $z \not\in \supp (\one_\Gamma - \Psi|_\Gamma)$.
Hence by definition of Schwartz kernel 
\begin{eqnarray*}
I_4 
& = & \int_\Gamma K_g(z,w) \, (1 - \chi_r(w-z)) \, \Big( \varphi(z_0) - \varphi(z) \Big) \, dw  \\
& = & \Big( \varphi(z_0) - \varphi(z) \Big) \, ([\cn, M_g](\one_\Gamma - \Psi|_\Gamma))(z)
.  
\end{eqnarray*}
Moreover, 
\[
([\cn, M_g] (\Psi|_\Gamma))(z)
= (\cn ((g \, \Psi)|_\Gamma) )(z)
\]
and using Lemma~\ref{lcom302}, one estimates
$|(\cn ((g \, \Psi)|_\Gamma ))(z)| \leq c'$.
Furthermore, 
\[
|([\cn, M_g] \one_\Gamma)(z)| = |(\cn( g|_\Gamma ))(z)| \leq c''
\]
by Proposition~\ref{pcom305}.
Therefore
\[
|I_4|
= |\varphi(z) - \varphi(z_0)| \, |([\cn, M_g](\one_\Gamma - \Psi|_\Gamma))(z)|
\leq (c' + c'') \, |||\varphi|||_{C^\nu(\Gamma)} \, |z-z_0|^\nu
.  \]

For the last term $I_5$ note that 
$\varphi \, [\cn, M_g] \one_\Gamma = \varphi \, \cn (g|_\Gamma) - \varphi \, g|_\Gamma \cdot \cn \one_\Gamma$.
Now Proposition~\ref{pcom305} gives 
$\|\cn(g|_\Gamma)\|_{C^\kappa(\Gamma)} \leq c''$.
Since $g$ has a zero on $\Gamma$, it follows from the mean value theorem 
that $\|g|_\Gamma\|_{C(\Gamma)} \leq \diam \Omega$
and $||| \, g|_\Gamma \, |||_{C^\nu(\Gamma)} \leq 1 + \diam \Omega$.
Therefore $\| \, g|_\Gamma \, \|_{C^\nu(\Gamma)} \leq 1 + 2 \diam \Omega$.
Then 
\begin{eqnarray*}
\|\varphi \, [\cn, M_g] \one_\Gamma\|_{C^\nu(\Gamma)}
& \leq & 2 \|\varphi\|_{C^\nu(\Gamma)} \, \|\cn(g|_\Gamma)\|_{C^\kappa(\Gamma)}
   + 3 \|\varphi\|_{C^\nu(\Gamma)} \, \| \, g|_\Gamma \, \|_{C^\nu(\Gamma)}
             \, \|\cn \one_\Gamma\|_{C^\nu(\Gamma)}  \\
& \leq & \Big( 2 c'' + 3 (1 + 2 \diam \Omega) \, \|\cn \one_\Gamma\|_{C^\nu(\Gamma)} \Big) \, 
      \|\varphi\|_{C^\nu(\Gamma)}
.
\end{eqnarray*}
Hence 
\[
|I_5|
\leq \Big( 2 c'' + 3 (1 + 2 \diam \Omega) \, \|\cn \one_\Gamma\|_{C^\nu(\Gamma)} \Big) 
    \, \|\varphi\|_{C^\nu(\Gamma)} \, |z-z_0|^\nu
\]
as required.

Combining the contributions one obtains that there exists a $c_3 > 0$, independent 
of $\varphi$, $g$, $z$ and $z_0$, such that 
\[
| (T_g \varphi)(z) - (T_g \varphi)(z_0)|
\leq c_3 \, \|\varphi\|_{C^\nu(\Gamma)} \, |z-z_0|^\nu.
\]
This gives the required semi-norm estimate.

It remains to obtain a suitable $L_\infty$ estimate for $T_g \varphi$.
Let $z \in \Gamma$ and set  $\tilde g = g - g(z) \, \one_{\Ri^d} \in W_2$. 
Then 
\[
(T_g \varphi)(z) 
= (T_{\tilde g} \varphi)(z) 
= \int_\Gamma K_{\tilde g}(z,w) \, (\varphi(w) - \varphi(z)) \, dw 
   + \varphi(z) \, (\cn(\tilde g|_\Gamma))(z),
   \]
where we used that $\tilde g(z) = 0$.
Now Lemma~\ref{lcom304}\ref{lcom304-2} gives
\begin{eqnarray*}
|\int_\Gamma K_{\tilde g}(z,w) \, (\varphi(w) - \varphi(z)) \, dw|
& \leq & \int_\Gamma \frac{c}{|w-z|^{d-1}} \, 2 \|\varphi\|_{C^\nu(\Gamma)} \, |w-z|^\nu  \\
& \leq & 2 c \, \frac{c_2}{1 - 2^{-\nu}} \, (\diam(\Omega))^\nu \, \|\varphi\|_{C^\nu(\Gamma)}
.
\end{eqnarray*}
Moreover,
\[
|\varphi(z) \, (\cn (\tilde g|_\Gamma))(z)|
\leq c'' \, \|\varphi\|_{C^\nu(\Gamma)} 
\]
by Proposition~\ref{pcom305}. 
This completes the proof of Step~\ref{tcom101-step1}.

\medskip

Since $ \{ \Phi|_\Gamma : \Phi \in C^\infty(\Ri^d) \} $  is dense in $C^\nu(\Gamma)$, 
for all $g \in W_2$ there exists 
a unique bounded extension $\widetilde T_g \colon C^\nu(\Gamma) \to C^\nu(\Gamma)$
of $T_g$ and 
\[
\sup_{g \in W_2} \|\widetilde T_g\|_{C^\nu(\Gamma) \to C^\nu(\Gamma)} < \infty
.  \]

\noindent\nextstep \label{tcom101-step2} 
We prove that the commutator $[\cn, M_g]$ is defined on $D(\cn)$ and extends to a bounded operator 
on $L_2(\Gamma)$. 
This is Statement~\ref{tcom101-1} of the theorem for $p=2$.  \\ 
Choose $\nu = \kappa / 2$ in Step~\ref{tcom101-step1}.
Let $\cd = \{ \Phi|_\Gamma : \Phi \in C^\infty(\Ri^d) \} $.
Then there is a $c > 0$ such that 
\[
\|\widetilde T_g \varphi\|_{C^\nu(\Gamma)} 
\leq c \, \|\varphi\|_{C^\nu(\Gamma)}
\]
for all $\varphi \in \cd$ and all $g \in W_2$.
Similarly,  by duality there exists a $c' > 0$ and for all $g \in W_2$ there 
exists a bounded operator $\widecheck T_g \colon C^\nu(\Gamma) \to C^\nu(\Gamma)$
such that $\|\widecheck T_g\|_{C^\nu(\Gamma) \to C^\nu(\Gamma)} \leq c'$
and $\widecheck T_g \varphi = [\cn^*, M_g] \varphi$ for all $\varphi \in \cd$.
Then 
\[
( \widetilde T_g \varphi_1, \varphi_2)_{L_2(\Gamma)}
= ( [\cn, M_g] \varphi_1, \varphi_2)_{L_2(\Gamma)}
= - ( \varphi_1, [\cn^*, M_g] \varphi_2)_{L_2(\Gamma)}
= - ( \varphi_1, \widecheck T_g \varphi_2)_{L_2(\Gamma)}
\]
for all $\varphi_1,\varphi_2 \in \cd$.
Hence by density 
\[
( \widetilde T_g \varphi_1, \varphi_2)_{L_2(\Gamma)}
= - ( \varphi_1, \widecheck T_g \varphi_2)_{L_2(\Gamma)}
\]
for all $\varphi_1,\varphi_2 \in C^\nu(\Gamma)$.
At this stage we use Krein's lemma, which we state for the convenience 
of the reader.

\begin{thm} \label{tcom310}
Let $H$ be a Hilbert space and $Y$ a normed space which is 
continuously and densely embedded into $H$.
Let $T,S \in \cl(Y)$ be two bounded operators on $Y$
and suppose that $(T y_1, y_2)_H = (y_1, S y_2)_H$ 
for all $y_1,y_2 \in Y$.
Then $T$ extends to a bounded operator $\widehat T \in \cl(H)$ on~$H$.
Moreover, $\|\widehat T\|_{H \to H} \leq \|T\|_{Y \to Y}^{1/2} \, \|S\|_{Y \to Y}^{1/2}$.
\end{thm}
\begin{proof} 
This theorem was proved by Krein \cite{Kre3} Theorem~1 in case $S=T$.
It was independently proved by Lax \cite{Lax} Theorem~1 for the case $S=T$.
A short proof without the restriction $S = T$ is given in \cite{Bramanti} Theorem~10.
Although Bramanti writes that his statement is for the field `real, for simplicity',
his proof is also valid for the complex field, without changing one letter. 
\end{proof}

By Theorem~\ref{tcom310}
for all $g \in W_2$ there exists a unique bounded operator 
$\widehat T_g \colon L_2(\Gamma) \to L_2(\Gamma)$
such that $\widehat T_g$ is an extension of $\widetilde T_g$ 
and $\|\widehat T_g\|_{L_2(\Gamma) \to L_2(\Gamma)} \leq \sqrt{c \, c'}$.
Then $\widehat T_g$ is an extension of $T_g$ for all $g \in W_2$.

Let $g \in W_2$.
If $\varphi \in \cd$, then $\varphi \in D(\cn)$ and $g|_\Gamma \, \varphi \in D(\cn)$
by Corollary~\ref{ccom306.5}.
Moreover, $\widehat T_g \varphi = [\cn,M_g] \varphi = \cn(g|_\Gamma \, \varphi) - g|_\Gamma \, \cn \varphi$.
Let $\Psi \in C^\infty(\Ri^d)$ and set $v = \Psi|_\Omega$.
Then $v \in H^1(\Omega)$.
Now 
\[
(\cn(g|_\Gamma \, \varphi), \Tr v)_{L_2(\Gamma)}
   - (g|_\Gamma \, \cn \varphi, \Tr v)_{L_2(\Gamma)} 
= (\widehat T_g \varphi, \Tr v)_{L_2(\Gamma)}
.  \]
So 
\begin{equation}
\gota(\gamma( g|_\Gamma \, \varphi), v)
   - \gota(\gamma(\varphi), g|_\Omega \, v)
= (\widehat T_g \varphi, \Tr v)_{L_2(\Gamma)}
.  
\label{etcom101;10}
\end{equation}
By density and continuity, (\ref{etcom101;10}) is valid for all $v \in H^1(\Omega)$.
The space $\cd$ is dense in $H^{1/2}(\Gamma)$ and the map 
$\varphi \mapsto g|_\Gamma \, \varphi$ is bounded from $H^{1/2}(\Gamma)$ 
into $H^{1/2}(\Gamma)$.
Hence by density, (\ref{etcom101;10}) is valid for all $\varphi \in H^{1/2}(\Gamma)$
and $v \in H^1(\Omega)$.
Now let $\varphi \in D(\cn)$.
Then 
\begin{eqnarray*}
\gota(\gamma( g|_\Gamma \, \varphi), v)
& = & \gota(\gamma(\varphi), g|_\Omega \, v) + (\widehat T_g \varphi, \Tr v)_{L_2(\Gamma)}  \\
& = & (\cn \varphi, g|_\Gamma \, \Tr v)_{L_2(\Gamma)} + (\widehat T_g \varphi, \Tr v)_{L_2(\Gamma)}
= (\psi, \Tr v)_{L_2(\Gamma)} 
\end{eqnarray*}
for all $v \in H^1(\Omega)$, 
where $\psi = g|_\Gamma \, \cn \varphi + \widehat T_g \varphi \in L_2(\Gamma)$.
Hence $g|_\Gamma \, \varphi \in D(\cn)$ and 
$\cn(g|_\Gamma \, \varphi) = g|_\Gamma \, \cn \varphi + \widehat T_g \varphi$.
In particular, $\varphi \in D([\cn, M_g])$ and $[\cn, M_g] \varphi = \widehat T_g \varphi$.
We proved that $\widehat T_g$ is an extension of $[\cn, M_g]$ and 
$\|\widehat T_g\|_{L_2(\Gamma) \to L_2(\Gamma)} \leq \sqrt{c \, c'}$
for all $g \in W_2$.
This completes the proof  of Statement~\ref{tcom101-1}.

\noindent\nextstep \label{tcom101-step3} 
We prove Statement~\ref{tcom101-2}. \\
Let $\varphi \in D(\cn) \cap C^\nu(\Gamma)$.
Then $g|_\Gamma \, \varphi \in D(\cn)$ by Statement~\ref{tcom101-1} and 
obviously $g|_\Gamma \, \varphi \in C^\nu(\Gamma)$.
So $g|_\Gamma \, \varphi \in D(\cn) \cap C^\nu(\Gamma)$.
Moreover, $[\cn, M_g] \varphi = \widehat T_g \varphi = \widetilde T_g \varphi$.
Since $\widetilde T_g$ is bounded from $C^\nu(\Gamma)$ into $C^\nu(\Gamma)$, 
with norm bound uniformly in $g \in W_2$, Statement~\ref{tcom101-2} follows.

\noindent\nextstep \label{tcom101-step4} 
Finally we prove Statement~\ref{tcom101-1} for all $p \in (1,\infty) \setminus \{ 2 \} $
and Statement~\ref{tcom101-3}. \\
By Lemma~\ref{lcom303}, the operator $\widehat T_g$ is a Calder\'on--Zygmund operator 
and hence it is of weak type $(1,1)$.
Using interpolation and duality  it extends to a bounded operator on $L_p(\Gamma)$ for all 
$p \in (1,\infty)$ and the bounds of Statement~\ref{tcom101-1} 
hold on $L_p(\Gamma)$. 
This completes the proof of Theorem~\ref{tcom101}.
\end{proof}

\section{Poisson bounds} \label{Scom10}

In this section we prove the Poisson bound stated in Theorem~\ref{tcom103}. 

We start by proving that the semigroup is bounded  on  $L_\infty(\Gamma)$.
This result does not need much regularity and it  holds on a Lipschitz domain 
and for bounded measurable real coefficients. 
We use the next general observation which allows us to restrict attention to 
connected components of $\Omega$.

\begin{lemma} \label{lcom9802}
Let $\Omega$ be open, bounded and Lipschitz.
Let $\widehat \Omega$ be a connected component of~$\Omega$.
For all $k,l \in \{ 1,\ldots,d \} $ let 
$c_{kl}, b_k, c_k, c_0 \in L_\infty(\Omega)$.
Suppose there is a $\mu > 0$ such that 
\[
\RRe \sum_{k,l=1}^d c_{kl}(x) \, \xi_k \, \overline{\xi_l} \geq \mu \, |\xi|^2
\]
for all $\xi \in \Ci^d$ and $x \in \Omega$.
Suppose the corresponding elliptic operator with Dirichlet boundary 
conditions is invertible.
Let $\cn$ be the corresponding Dirichlet-to-Neumann operator on $L_2(\partial \Omega)$.
Restricting the coefficients to $\widehat \Omega$ one also obtains a 
Dirichlet-to-Neumann operator $\widehat \cn$ on $L_2(\partial \widehat \Omega)$.
Let $t > 0$.
Then $e^{-t \cn}$ leaves $L_2(\partial \widehat \Omega)$ invariant and 
$e^{-t \cn} \varphi = e^{-t \widehat \cn} \varphi$ for all 
$\varphi \in L_2(\partial \widehat \Omega)$, where as usual we identify 
$L_2(\partial \widehat \Omega)$ with the elements of $L_2(\partial \Omega)$ 
which vanish on $(\partial \Omega) \setminus \partial \widehat \Omega$.
\end{lemma}
\begin{proof}
Let $\gamma$ and $\hat \gamma$ denote the harmonic liftings on 
$\Omega$ and $\widehat \Omega$.
If $\varphi \in H^{1/2}(\partial \widehat \Omega)$, then 
$\hat \gamma(\varphi)$ is harmonic on $\Omega$, 
hence $\gamma(\varphi) = \hat \gamma(\varphi)$.
Therefore if $\varphi \in D(\widehat \cn)$, then 
$\varphi \in D(\cn)$ and $\cn \varphi = \widehat \cn \varphi$.

If $\lambda \in \Ri$ is large enough and $\varphi \in L_2(\partial \widehat \Omega)$, 
then $(\lambda \, I + \widehat \cn)^{-1} \varphi \in D(\widehat \cn) \subset D(\cn)$
and 
\[
(\lambda \, I + \cn) \, (\lambda \, I + \widehat \cn)^{-1} \varphi
= (\lambda \, I + \widehat \cn) \, (\lambda \, I + \widehat \cn)^{-1} \varphi
= \varphi
.  \]
So $(\lambda \, I + \cn)^{-1} \varphi 
= (\lambda \, I + \widehat \cn)^{-1} \varphi 
\in L_2(\partial \widehat \Omega)$
and $(\lambda \, I + \cn)^{-1}$ leaves $L_2(\partial \widehat \Omega)$ invariant.
If $t > 0$ and $\varphi \in L_2(\partial \widehat \Omega)$, then the Euler formula
gives
\[
e^{-t \cn} \varphi
= \lim_{n \to \infty} (I + \tfrac{t}{n} \, \cn)^{-n} \varphi
= \lim_{n \to \infty} (I + \tfrac{t}{n} \, \widehat \cn)^{-n} \varphi
= e^{-t \widehat \cn} \varphi
\in L_2(\partial \widehat \Omega)
\]
and the proof is complete.
\end{proof}

\begin{prop} \label{pcom9801}
Suppose $\Omega$ is open, bounded and Lipschitz.
For all $k,l \in \{ 1,\ldots,d \} $ let 
$c_{kl}, b_k, c_k, c_0$ be bounded measurable and real valued.
Suppose there is a $\mu > 0$ such that 
\[
\RRe \sum_{k,l=1}^d c_{kl}(x) \, \xi_k \, \overline{\xi_l} \geq \mu \, |\xi|^2
\]
for all $\xi \in \Ci^d$ and $x \in \Omega$.
Suppose the corresponding elliptic operator with Dirichlet boundary 
conditions is invertible.

Then there are $\lambda \ge 0$ and  a constant $M \geq 1$ such that 
$A_D+\lambda\,I$ is invertible and 
\[
\|e^{-t \cn_\lambda} \varphi\|_{L_\infty(\Gamma)}
\leq M \, \|\varphi\|_{L_\infty(\Gamma)}
\]
for all $t > 0$ and $\varphi \in L_\infty(\Gamma)$, where
$\cn_\lambda$ is the corresponding Dirichlet-to-Neumann operator 
with $c_0$ replaced by $c_0 + \lambda \, \one_{\Omega}$. 
\end{prop}
\begin{proof}
There exists a $\lambda > 0$ such that the form 
$\gota_\lambda \colon H^1(\Omega) \times H^1(\Omega) \to \Ci$ is coercive, where
$\gota_\lambda(u,v) = \gota(u,v) + \lambda \, (u,v)_{L_2(\Omega)}$
for all $u,v \in H^1(\Omega)$ as before.
Let $A_\lambda$ be the operator in $L_2(\Omega)$ 
associated with $\gota_\lambda$.
Then $A_\lambda$ is the operator with Neumann boundary conditions.
Let $\lambda_0 = \inf \{ \RRe z : z \in \sigma(A_\lambda) \} $.
Then $\lambda_0 \geq 0$.
First suppose that $\Omega$ is connected.
By \cite{AEG} Theorem~4.5(b) there exists a unique $f_0 \in D(A_\lambda) \cap C(\overline \Omega)$
such that $A_\lambda f_0 = \lambda_0 \, f_0$, 
$\|f_0\|_{L_2(\Omega)} = 1$ and $f_0(x) > 0$ for all $x \in \overline \Omega$.
Let $\delta = \min \{ f_0(x) : x \in \overline \Omega \} $
and $M = \max \{ f_0(x) : x \in \overline \Omega \} $.
Set $\varphi_0 = f_0|_\Gamma \in C(\Gamma)$.

Define 
\[
\widetilde \cc = \{ u \in L_2(\Omega,\Ri) : 0\le u \leq f_0 \} 
\quad \mbox{and} \quad
\cc = \{ \varphi \in L_2(\Gamma,\Ri) : 0 \le \varphi \leq \varphi_0 \} 
.  \]
Then $\widetilde \cc$ and $\cc$ are closed convex sets in $L_2(\Omega)$
and $L_2(\Gamma)$.
The projections $\widetilde P \colon L_2(\Omega) \to \widetilde \cc$ 
and $P \colon L_2(\Gamma) \to \cc$ are given by
\[
\widetilde P u = f_0 \wedge (\RRe u)^+ 
\quad \mbox{and} \quad
P \varphi = \varphi_0 \wedge (\RRe \varphi)^+
.  \]

We next show that the semigroup $(e^{-t A_\lambda})_{t > 0}$ 
leaves invariant $\widetilde \cc$. 
Indeed, let $u \in \widetilde \cc$ and $t > 0$.
Since $e^{-t A_\lambda}$ is a positive operator by \cite{Ouh5} Corollary~4.3,
one deduces that
\[
0 \leq e^{-t A_\lambda} u
\leq e^{-t A_\lambda} f_0
= e^{-t \lambda_0} \, f_0
\leq f_0
\]
and $e^{-t A_\lambda} u \in \widetilde \cc$.

By \cite{Ouh5} Theorem~2.2 1)${}\Rightarrow{}$2) it follows that 
$\widetilde P u \in H^1(\Omega)$ for all $u \in H^1(\Omega)$.
On the other hand, it follows easily from the identity  
$f_0 \wedge (\RRe u)^+ = \frac{1}{2}( f_0 + (\RRe u)^+ - | f_0 -(\RRe u)^+ |)$ 
and the continuity on $H^1(\Omega)$ of the maps $u \mapsto |u|$ and  $ u \mapsto (\RRe u)^+$ that 
the function $\widetilde P|_{H^1(\Omega)} \colon H^1(\Omega) \to H^1(\Omega)$
is continuous.
If $u \in H^1(\Omega) \cap C(\overline \Omega)$, then $\Tr u = u|_\Gamma$ and 
therefore
\[
(\Tr \circ \widetilde P)(u)
= \Tr (f_0 \wedge (\RRe u)^+)
= \varphi_0 \wedge (\RRe u)^+|_\Gamma) 
= (P \circ \Tr)(u)
.  \]
Since $H^1(\Omega) \cap C(\overline \Omega)$ is dense in $H^1(\Omega)$
one deduces by continuity that 
$\Tr \circ \widetilde P = P \circ \Tr$ on $H^1(\Omega)$.
Using  \cite{EO8} Proposition~3.3 we conclude  that the 
semigroup $(e^{- t \cn_\lambda})_{t > 0}$ leaves invariant 
the closed convex set $\cc$.
In particular, this implies that the semigroup $(e^{- t \cn_\lambda})_{t > 0}$
is positive.
This positivity was already proved in  \cite{EO8} Corollary~4.1. 
Now let $t > 0$ and $\varphi \in L_\infty(\Gamma)$ with $\|\varphi\|_{L_\infty(\Gamma)} \leq 1$.
Then $\varphi_0 \, |\varphi| \in \cc$ and by invariance also 
$e^{- t \cn_\lambda} (\varphi_0 \, |\varphi|) \in \cc$, 
hence
$e^{- t \cn_\lambda} (\varphi_0 \, |\varphi|) \leq \varphi_0$.
Moreover, $|\delta \, \varphi| \leq \varphi_0 \, |\varphi|$.
Therefore 
\[
|\delta \, e^{- t \cn_\lambda} \varphi|
= |e^{- t \cn_\lambda} (\delta \, \varphi)|
\leq e^{- t \cn_\lambda} |\delta \, \varphi|
\leq e^{- t \cn_\lambda} (\varphi_0 \, |\varphi|)
\leq \varphi_0
\leq M
.  \]
Hence $|e^{- t \cn_\lambda} \varphi| \leq \frac{M}{\delta}$.
Consequently 
$\|e^{- t \cn_\lambda} \varphi\|_{L_\infty(\Gamma)} 
\leq \frac{M}{\delta} \, \|\varphi\|_{L_\infty(\Gamma)}$
for all $\varphi \in L_\infty(\Gamma)$.

Now we remove the connectedness assumption.
Since  $\Omega$ is bounded and Lipschitz, 
it has a finite number of connected components.
So $\Omega$ is 
a disjoint union of a finite number of 
connected open subsets $\Omega_k$.
Using Lemma~\ref{lcom9802} we can apply the above arguments to the restriction of the 
semigroup $(e^{-t \cn})_{t > 0}$ to 
$L_2(\partial \Omega_k)$ for each $k$.
One obtains  $\|e^{- t \cn_\lambda} \varphi\|_{L_\infty(\partial \Omega_k)} 
\leq \frac{M_k}{\delta_k} \, \|\varphi\|_{L_\infty(\partial \Omega_k)}$
for all $\varphi \in L_\infty(\Gamma)$ and $t > 0$.
Here $M_k$ and $\delta_k$ 
are the same quantities as $M$ and $\delta$ 
but now taken on $\overline{\Omega_k}$. 
This proves the proposition on $L_\infty(\Gamma)$. 
\end{proof}

Let us mention that we can take $c_0$ to be complex.
Indeed, by the method of  \cite{EO8} Theorem~6.1 the semigroup 
$(e^{-t \cn_\lambda})$ is dominated by the Dirichlet-to-Neumann 
semigroup with $c_0$ replaced by $\RRe c_0$. 

For smooth domain and smooth coefficients we can remove the constant $\lambda$ in the proposition.
As in the previous sections, we denote by $A_D$ the elliptic operator with Dirichlet boundary conditions. 

\begin{prop} \label{pcom9802}
Suppose $\Omega$ is bounded and of class $C^{1+\kappa}$ for some $\kappa > 0$. 
Suppose $c_{kl}, b_k, c_k, c_0$ are H\"older  continuous on $\Omega$ and real valued.
Suppose also that $0 \notin \sigma(A_D)$. 
Then there is   a constant $M \geq 1$ such that 
\[
\|e^{-t \cn} \varphi\|_{L_\infty(\Gamma)}
\leq M \, \|\varphi\|_{L_\infty(\Gamma)}
\]
for all $t \in (0, 1]$ and $\varphi \in L_\infty(\Gamma)$.
\end{prop}
\begin{proof} 
Take $\lambda > 0$ large enough as in the previous proposition.
Let $\gamma_\lambda^\#$ and $K_{\gamma_\lambda^\#}$ be as in Lemma~\ref{lcom205.7}
and Theorem~\ref{tcom205}.
By Lemma~\ref{lcom308}~\ref{lcom308-3}, 
there exists a $\tilde c > 0$ such that 
\[
|K_\gamma(x,z)| \leq \frac{\tilde{c}}{|x-z|^{d-\frac{3}{4}}}
\quad \mbox{and} \quad 
 |K_{\gamma_\lambda^\#}(x,z)| \leq \frac{\tilde{c}}{|x-z|^{d-\frac{3}{4}}}
\]
for all $x \in \Omega$ and $z \in \Gamma$.
Actually,  the estimates hold with $1$ at the place of $\frac{3}{4}$ but then the 
next argument breaks down if $d=2$. 
By applying \cite{Fri} (Lemma~2, Section~4, Chapter~1) 
we obtain an upper bound for the kernel $K_Q$ of 
$Q = (\gamma_\lambda^\#)^* \, \gamma$.
Specifically, there is a $c > 0$ such that 
\[
| K_Q(z,w) | \le \frac{c}{| z-w|^{d-\frac{3}{2}}}
\]
for all $ z, w \in \Gamma$.
Therefore, this kernel is integrable on $\Gamma$ with respect to each variable with a bound independent of the other variable.
This implies that $Q$ extends to a bounded operator on $L_\infty(\Gamma)$.
The formula 
\[
\cn = \cn_\lambda - \lambda \, Q
  \]
from Lemma~\ref{lcom205.7} together with  
Proposition~\ref{pcom9801} imply the $L_\infty$-estimate for the semigroup $(e^{-t \cn})$. 
\end{proof} 

\begin{lemma}\label{lemma-poi1}
Under the assumptions of Proposition~\ref{pcom9802}, the semigroup 
$(e^{-t\cn})_{t > 0}$ is given by a kernel $K$ and there exist  
constants $c > 0$ and $\omega_0 \in \Ri$ such that 
\[
| K_t(z,w) | \le c \, t^{-(d-1)} \, e^{\omega_0 t} 
\]
for all $ t > 0$ and $z, w \in \Gamma$.
\end{lemma}
\begin{proof}
We argue as in \cite{EO4}, Theorem~2.6.
Let $d \ge 3$.
By the Sobolev embedding it follows that $e^{-t\cn}$ maps 
$L_2(\Gamma)$ into $L_{\frac{2(d-1)}{d-2}}(\Gamma)$  with a norm estimate 
$c \, t^{-1/2} \, e^{\omega' t}$ for all $t > 0$.
Since the semigroup is bounded on $L_\infty(\Gamma)$ by 
Proposition~\ref{pcom9802}, this estimate extrapolates to give
\[
\| e^{-t \cn} \|_{2 \to \infty} 
\le \tilde{c} \, t^{-\frac{d-1}{2}} \, e^{\omega'' t}
\]
for all $t > 0$ and some constants $\tilde{c}$ and $\omega''$.
For this  extrapolation argument we refer to  
\cite{Cou6} and  \cite{Ouh5}, Lemma~6.1.
The same argument applies to the adjoint semigroup $(e^{-t\cn^*})_{t > 0}$, since 
$\cn^*$ is the Dirichlet-to-Neumann operator associated with the 
adjoint form $\gota^*$.
We conclude by the help of the semigroup property that $e^{-t\cn}$ maps 
$L_1(\Gamma)$ into $L_\infty(\Gamma)$ with the desired estimates.  

The case $d = 2$ is treated in a similar way, see \cite{EO4}, Theorem~2.6.
\end{proof}

\begin{proof}[{\bf Proof of Theorem~\ref{tcom103}.}]
Note that by interpolation and the fact that the semigroup $(e^{-t\cn})_{t > 0}$ 
is bounded on $L_p(\Gamma)$ for all $p \in [1, \infty]$ 
the corresponding  $L_p(\Gamma) \to L_q(\Gamma)$ estimates for $e^{-t\cn}$ 
hold for all $1 \le p \le q \le \infty$.  

To simplify notation, set $S_t = e^{-t \cn}$ and consider $t \in (0, 1]$.
Let $g \in W_2$ and denote by $M_g$ the multiplication operator by $g$ as before.
Iterating the formula
\[
[M_g, S_t] = - \int_0^t S_{t-s} \, [M_g, \cn] \, S_s\, ds
\]
for the commutator of $S_t$ with $M_g$ one obtains
by induction for the commutator 
$\delta_g^d(S_t) = [M_g, [\ldots ,[M_g, S_t] \ldots ]]$
of order $d$ the integral 
\begin{eqnarray*}
\delta_g^d(S_t)
& = & \sum_{k=1}^d (-t)^k
   \sum_{\scriptstyle j_1,\ldots,j_k \in \Ni \atop
         \scriptstyle j_1 + \ldots + j_k = d}
   \int_{H_k}
     S_{t_{k+1} \, t} \, \delta^{j_k}(\cn) \, S_{t_k \, t} 
       \circ \ldots \circ  \\*
& & \hspace*{50mm} {}
   \circ
     S_{t_2 \, t} \, \delta^{j_1}(\cn) \, S_{t_1 \, t} \, d\lambda_k(t_1,\ldots,t_{k+1}),
     \end{eqnarray*}
where 
\[
H_k = \{ (t_1,\ldots,t_{k+1}) \in (0,\infty)^{k+1} : t_1 + \ldots + t_{k+1} = 1 \}
,  \]
 $d\lambda_k$ denotes Lebesgue measure of the $k$-dimensional surface $H_k$
and $\delta^j(\cn)$ is the commutator of order $j$ for all $j \in \Ni$.

The idea is to estimate the $L_1(\Gamma) \to L_\infty(\Gamma)$ norm of this commutator of order $d$.
By Theorem~\ref{tcom205}, the kernel $K_j$ of $\delta_g^j(\cn)$ can be estimated and there is a
$c > 0$ such that   $|K_j(z,w)| \le \frac{c}{|z-w|^{d-j}}$ for all $z,w \in \Gamma$ with $z \neq w$.
Therefore, using Riesz potentials, it follows that  
for all $1 < p \leq q <\infty$ with  $(d-1)(\frac{1}{p} - \frac{1}{q}) \in \{ 0,1,\ldots,d-1 \}$ one has 
$\|\delta_g^j(\cn)\|_{p \to q}
\leq c_{p,q} $
for some constant $c_{p,q} > 0$, independent of $g \in W_2$,
where $j = 1 + (d-1)(\frac{1}{p} - \frac{1}{q})$.
For the case case $p= 1$ and $q= \infty$ (that is for the case $j = d$) 
one uses the fact that $S_t$ is bounded on $L_1(\Gamma)$ and $L_\infty(\Gamma)$ 
uniformly for $t \in (0,1]$ and that $\delta_g^d(\cn)$ is bounded from $L_1(\Gamma)$ into $L_\infty(\Gamma)$.
The latter follows  
since its kernel $(z,w) \mapsto (g(z) -g(w))^d \, K_\cn(z,w)$ is bounded 
by a constant by Theorem~\ref{tcom205}.
Inserting all these estimates in the above 
formula for $\delta_g^d(S_t)$ gives
\begin{equation}
\| \delta_g^d(S_t) \|_{1\to \infty} \le c\, t 
\label{etcom103;1}
\end{equation}
for some $c > 0$, uniformly for all $t \in (0,1]$ and $g \in W_2$.
All this is explained in detail in Section~8 of \cite{EO6} and no self-adjointness 
of the operator is needed in this step.
The estimate (\ref{etcom103;1}) gives
\[
| (g(z) - g(w) )^d K_t(z,w) | \le c\, t
\]
for all $t \in (0,1]$ and $z,w \in \Gamma$.
The constant $c$ is uniform with respect to $g \in W_2$.
Optimising over $g \in W_2$ and using Lemma~\ref{lemma-poi1} gives the Poisson bound. 
\end{proof}

\begin{cor} \label{ccom1005}
Adopt the assumptions and notation of Theorem~\ref{tcom103}.
Let $S$ be the semigroup generated by $-\cn$.
Then $S$ leaves the space $C(\Gamma)$ invariant.
For all $t > 0$ let $T_t = S_t|_{C(\Gamma)} \colon C(\Gamma) \to C(\Gamma)$.
Then $T = (T_t)_{t > 0}$ is a $C_0$-semigroup.
\end{cor}
\begin{proof}
It follows from \cite{EW1} Proposition~5.7 that $S$ leaves the space $C(\Gamma)$ invariant.
Let $M \geq 1$ be as in Proposition~\ref{pcom9802}.
Now let $\Phi \in C^\infty(\Ri^d)$ and set $\varphi = \Phi|_\Gamma$.
Then $\varphi \in D(\cn)$ and $\cn \varphi \in C(\Gamma)$ by 
Lemma~\ref{lcom901}\ref{lcom901-4}.
Let $t \in (0,1]$.
Then 
\begin{eqnarray*}
\|T_t \varphi - \varphi\|_{C(\Gamma)}
& = & \|S_t \varphi - \varphi\|_{C(\Gamma)} 
= \Big\| \int_0^t S_s \, \cn \varphi \, ds \Big\|_{ C(\Gamma)}  
\leq \int_0^t  \| S_s \, \cn \varphi\|_{ C(\Gamma)} \, ds  \\
& \leq & \int_0^t \|T_s\| \, \|\cn \varphi\|_{ C(\Gamma)} \, ds  
\leq M \, t \, \|\cn \varphi\|_{ C(\Gamma)}
.
\end{eqnarray*}
Hence $\lim_{t \downarrow 0} \|T_t \varphi - \varphi\|_{C(\Gamma)} = 0$.
Since $ \{ \Phi|_\Gamma : \Phi \in C^\infty(\Ri^d) \} $ is dense in $C(\Gamma)$, 
the corollary follows.
\end{proof}

\begin{cor}\label{ccom1006}
Adopt the assumptions and notation of Corollary~\ref{ccom1005}.
Then the semigroup $S$  is holomorphic on $L_p(\Gamma)$ for all $p \in [1, \infty)$. 
The semigroup $T$ is holomorphic on $C(\Gamma)$. 
Moreover, for all $p \in [1, \infty)$ the angle of holomorphy on $L_p(\Gamma)$ 
and on $C(\Gamma)$ is the same as on $L_2(\Gamma)$.
\end{cor}
\begin{proof} 
The Dirichlet-to-Neumann operator, as a sectorial operator, generates a 
holomorphic semigroup on $L_2(\Gamma)$.
Let us denote by $\theta_0 \in (0, \frac{\pi}{2}]$ the maximum angle of holomorphy.
Following \cite{EO7} and using the holomorphic extension of a small sector 
of \cite{DR1}, the Poisson bound of Theorem~\ref{tcom103} extends to 
a Poisson bound for complex time.
More precisely, for any $\theta \in (0, \theta_0)$, there exist $c > 0$  
and $\omega \in \Ri$ such that 
\[
| K_z(w_1, w_2) | 
\le c \, (\RRe z)^{-(d-1)} \, e^{\omega \RRe z} \left(1 + \frac{|w_1 - w_2|}{|z|} \right)^{-d}
\]
for all $z \in \Sigma(\theta)$, the open sector of (half-)angle $\theta$. 
As a consequence, this implies that $(e^{-t \cn})_{t > 0}$ is also holomorphic on 
$L_p(\Gamma)$ for all $p \in [1,\infty)$  and on $C(\Gamma)$ with angle $\theta_0$ 
as on $L_2(\Gamma)$.
The proofs are  the same as in the symmetric case of \cite{EO7}, 
where one replaces $\frac{\pi}{2}$ by $\theta_0$.
Let us only mention that for the holomorphy on $C(\Gamma)$ it is used there 
that $e^{-t \cn } L_\infty(\Gamma) \subset C(\Gamma)$ for all $t > 0$.
This later property is contained in \cite{EW1} Proposition~5.7. 
\end{proof}

\subsection*{Conflict of interest}
The authors have no conflict of interest to declare.

\subsection*{Ethics approval}
Not applicable.

\subsection*{Funding}
The research of A.F.M. ter Elst was partly supported by the 
Marsden Fund Council from Government funding, 
administered by the Royal Society of New Zealand. 
The research of E.M. Ouhabaz was partly supported by the ANR project 
`Real Analysis and Geometry' ANR-18-CE-0012-01.

\subsection*{Data availability}
Data availability is not applicable because the manuscript has no associated data.


\begin{thebibliography}{AEG20}

\bibitem[AE1]{AE2}
{\sc Arendt, W. {\rm and} Elst, A. F.~M. ter}, Sectorial forms and degenerate
  differential operators.
\newblock {\em J. Operator Theory} {\bf 67} (2012),  33--72.

\bibitem[AE2]{AE11}
\leavevmode\vrule height 2pt depth -1.6pt width 23pt, Operators with continuous
  kernels.
\newblock {\em Integral Equations Operator Theory} {\bf 91:45} (2019).

\bibitem[AEG]{AEG}
{\sc Arendt, W., Elst, A. F.~M. ter {\rm and} Gl{\"u}ck, J.}, Strict positivity
  for the principal eigenfunction of elliptic operators with various boundary
  conditions.
\newblock {\em Adv. Nonlinear Stud.} {\bf 20} (2020),  633--650.

\bibitem[AM]{ArM2}
{\sc Arendt, W. {\rm and} Mazzeo, R.}, Friedlander's eigenvalue inequalities
  and the Dirichlet-to-Neumann semigroup.
\newblock {\em Commun. Pure Appl. Anal.} {\bf 11} (2012),  2201--2212.

\bibitem[APL]{APL}
{\sc Astala, K., P{\"a}iv{\"a}rinta, L. {\rm and} Lassas, M.}, Calder\'on's
  inverse problem for anisotropic conductivity in the plane.
\newblock {\em Comm. Partial Differential Equations} {\bf 30} (2005),
  207--224.

\bibitem[Aus]{aus1}
{\sc Auscher, P.}, Regularity theorems and heat kernels for elliptic operators.
\newblock {\em J. London Math.\ Soc.} {\bf 54} (1996),  284--296.

\bibitem[BR]{BehR}
{\sc Behrndt, J. {\rm and} Rohleder, J.}, An inverse problem of Calder\'on type
  with partial data.
\newblock {\em Comm. Partial Differential Equations} {\bf 37} (2012),
  1141--1159.

\bibitem[BE]{BinzE1}
{\sc Binz, T. {\rm and} Elst, A. F.~M. ter}, Dynamic boundary conditions for
  divergence form operators with H{\"o}lder coefficients.
\newblock {\em Ann. Sc. Norm. Super. Pisa Cl. Sci. (5)} {\bf 25}, 2173--2199 (2024).

\bibitem[BMN]{BMN}
{\sc Behrndt, J., Malamud, M.~M. {\rm and} Neidhardt, H.}, Scattering matrices
  and Dirichlet-to-Neumann maps.
\newblock {\em J. Funct. Anal.} {\bf 273} (2017),  1970--2025.

\bibitem[Bra]{Bramanti}
{\sc Bramanti, M.}, Singular integrals in nonhomogeneous spaces: $L^2$ and
  $L^p$ continuity from H{\"o}lder estimates.
\newblock {\em Rev. Mat. Iberoam.} {\bf 26} (2010),  347--366.

\bibitem[CM]{CM2}
{\sc Coifman, R.~R. {\rm and} Meyer, Y.}, Commutateurs d'int\'egrales
  singuli\`eres et op\'erateurs multilin\'eaires.
\newblock {\em Ann. Inst. Fourier, Grenoble} {\bf 28} (1978),  177--202.

\bibitem[Cou]{Cou6}
{\sc Coulhon, T.}, Dimensions of continuous and discrete semigroups on the
  $L^p$-spaces.
\newblock In {\em Semigroup theory and evolution equations, Delft 1989},
  Lecture Notes in Pure and Appl. Math. 135,  93--99. Dekker, New York, 1991.

\bibitem[DE]{DeE}
{\sc Denis, C. {\rm and} Elst, A. F.~M. ter}, The Stokes Dirichlet-to-Neumann
  operator.
\newblock {\em J. Evol.\ Equ.} {\bf 24} (2024),  Paper No. 22.

\bibitem[DR]{DR1}
{\sc Duong, X.~T. {\rm and} Robinson, D.~W.}, Semigroup kernels, Poisson
  bounds, and holomorphic functional calculus.
\newblock {\em J. Funct.\ Anal.} {\bf 142} (1996),  89--129.

\bibitem[EO1]{EO4}
{\sc Elst, A. F.~M. ter {\rm and} Ouhabaz, E.~M.}, Analysis of the heat kernel
  of the Dirichlet-to-Neumann operator.
\newblock {\em J. Funct. Anal.} {\bf 267} (2014),  4066--4109.

\bibitem[EO2]{EO6}
\leavevmode\vrule height 2pt depth -1.6pt width 23pt, Dirichlet-to-Neumann and
  elliptic operators on $C^{1+\kappa}$-domains: Poisson and Gaussian bounds.
\newblock {\em J. Differential Equations} {\bf 267} (2019),  4224--4273.

\bibitem[EO3]{EO7}
\leavevmode\vrule height 2pt depth -1.6pt width 23pt, Analyticity of the
  Dirichlet-to-Neumann semigroup on continuous functions.
\newblock {\em J. Evol. Eq.} {\bf 19} (2019),  21--31.

\bibitem[EO4]{EO8}
\leavevmode\vrule height 2pt depth -1.6pt width 23pt, The diamagnetic
  inequality for the Dirichlet-to-Neumann operator.
\newblock {\em Bull. Lond. Math. Soc.} {\bf 54} (2022),  1978--1997.

\bibitem[ER]{ERe2}
{\sc Elst, A. F.~M. ter {\rm and} Rehberg, J.}, H{\"o}lder estimates for
  second-order operators on domains with rough boundary.
\newblock {\em Adv. Diff. Equ.} {\bf 20} (2015),  299--360.

\bibitem[EW]{EW1}
{\sc Elst, A. F.~M. ter {\rm and} Wong, M.~F.}, H{\"o}lder kernel estimates for
  Robin operators and Dirichlet-to-Neumann operators.
\newblock {\em J. Evol. Equ.} {\bf 20} (2020),  1195--1225.

\bibitem[Fri]{Friedlander}
{\sc Friedlander, L.}, Some inequalities between Dirichlet and Neumann
  eigenvalues.
\newblock {\em Arch. Rational Mech. Anal.} {\bf 116} (1991),  153--160.

\bibitem[Fri]{Fri}
{\sc Friedman, A.}, {\em Partial differential equations of parabolic type}.
\newblock Prentice-Hall, Inc., Englewood Cliffs, N.J., 1964.

\bibitem[GKLP]{GKLP}
{\sc Girouard, A., Karpukhin, M., Levitin, M. {\rm and} Polterovich, I.}, The
  Dirichlet-to-Neumann map, the boundary Laplacian, and H{\"o}rmander's
  rediscovered manuscript.
\newblock {\em J. Spectr. Theory} {\bf 12} (2022),  195--225.

\bibitem[GG]{GG}
{\sc Gimperlein, H. {\rm and} Grubb, G.}, Heat kernel estimates for
  pseudodifferential operators, fractional Laplacians and Dirichlet-to-Neumann
  operators.
\newblock {\em J. Evol. Equ.} {\bf 14} (2014),  49--83.

\bibitem[HZ]{HofmannZhang}
{\sc Hofmann, S. {\rm and} Zhang, G.}, $L^2$ estimates for commutators of the
  Dirichlet-to-Neumann map associated to elliptic operators with complex-valued
  bounded measurable coefficients on $\Ri^{n+1}_+$.
\newblock {\em J. Math. Anal. Appl.} {\bf 504} (2021),  Paper No. 125408.

\bibitem[KLS]{KLS}
{\sc Kenig, C.~E., Lin, F. {\rm and} Shen, Z.}, Periodic homogenization of
  Green and Neumann functions.
\newblock {\em Comm. Pure Appl. Math.} {\bf 67} (2014),  1219--1269.

\bibitem[Kre]{Kre3}
{\sc Krein, M.~G.}, Compact linear operators on functional spaces with two
  norms.
\newblock {\em Integral Equations Oper. Theory} {\bf 30} (1998),  140--162.
\newblock Translation of {\em Sbirnik Praz' Instituta Matem. Akademii Nauk,
  Ukrainsk. SSR}, {\bf 9} (1947), 104--129.

\bibitem[McS]{McShane}
{\sc McShane, E.~J.}, Extension of range of functions.
\newblock {\em Bull. Amer. Math. Soc.} {\bf 40} (1934),  837--842.

\bibitem[Lan]{Lannes}
{\sc Lannes, D.}, {\em The water waves problem}.
\newblock Mathematical Surveys and Monographs 188. Amer. Math. Soc.,
  Providence, RI, 2013.

\bibitem[Lax]{Lax}
{\sc Lax, P.~D.}, Symmetrizable linear transformations.
\newblock {\em Comm. Pure Appl. Math.} {\bf 7} (1954),  633--647.

\bibitem[LU]{LeeUhlmann}
{\sc Lee, J.~M. {\rm and} Uhlmann, G.}, Determining anisotropic real-analytic
  conductivities by boundary measurements.
\newblock {\em Comm. Pure Appl. Math.} {\bf 42} (1989),  1097--1112.

\bibitem[Ouh]{Ouh5}
{\sc Ouhabaz, E.~M.}, {\em Analysis of heat equations on domains}.
\newblock London Mathematical Society Monographs Series 31. Princeton
  University Press, Princeton, NJ, 2005.
  
\bibitem[Ouh2]{Ouh8} \leavevmode\vrule height 2pt depth -1.6pt width 23pt,  
A ``milder'' version of Calder\'on's inverse problem for
  anisotropic conductivities and partial data.
\newblock {\em J. Spectr. Theory} {\bf 8} (2018),  435--457.

\bibitem[She]{She2}
{\sc Shen, Z.}, Commutator estimates for the Dirichlet-to-Neumann map in
  Lipschitz domains.
\newblock In {\sc Li, J., Li, X. {\rm and} Lu, G.}, eds., {\em Some topics in
  harmonic analysis and applications}, Advanced Lectures in Mathematics 34,
  369--384. International Press; Higher Education Press, Somerville, MA;
  Beijing, 2016.
  
\bibitem[Tay]{Tay5}
{\sc Taylor, M.~E.}, {\em Partial differential equations. II{}. Qualitative
  studies of linear equations}, vol.\ 116 of Applied Mathematical Sciences.
\newblock Springer-Verlag, New York, 1996.

\bibitem[Uhl]{uhlmann2009}
{\sc Uhlmann, G.}, Electrical impedance tomography and Calder\'on's problem.
\newblock {\em Inverse Problems} {\bf 25} (2009),  Paper No. 123011.
  
\bibitem[XZZ]{XuZhaoZhou}
{\sc Xu, Q., Zhao, W. {\rm and} Zhou, S.}, Commutator estimates for the
  Dirichlet-to-Neumann map of Stokes systems in Lipschitz domains.
\newblock {\em J. Math. Anal. Appl.} {\bf 458} (2018),  219--241.


\end{thebibliography}
\end{document}